\documentclass[preprint,12pt]{elsarticle}

\usepackage[english]{babel}
\usepackage{amssymb}
\usepackage[utf8]{inputenc}
\usepackage[a4paper, total={6in, 9.5in}]{geometry}
\usepackage{graphicx}
\usepackage{amsmath}
\usepackage{natbib}
\usepackage[version=4]{mhchem}
\usepackage{siunitx}
\usepackage{longtable,tabularx}
\usepackage{microtype}
\usepackage{subcaption}
\usepackage{booktabs} 
\usepackage{float}
\usepackage[algo2e,ruled,vlined,linesnumbered]{algorithm2e}
\usepackage{amsfonts}
\usepackage{bm,bbm}
\usepackage{cleveref}
\usepackage{xcolor}
\usepackage{fixmath}

\newcommand{\mbb}{\mathbb}
\newcommand{\mbf}{\mathbf}
\newcommand{\mcl}{\mathcal}
\newcommand{\bs}{\boldsymbol}

\newcommand{\f}{\frac}

\newcommand{\T}{\textnormal}
\newcommand{\x}{\mathbf{x}}
\newcommand{\X}{\bm{\mathcal{X}}}
\newcommand{\D}{\bm{\mathcal{D}}}
\newcommand{\Sc}{\bm{\mathcal{S}}}
\newcommand{\upsilona}[2]{\upsilon_{#1,#2}}
\newcommand{\defeq}{:=}

\newcommand{\bI}{\pmb{\mbb{I}}}
\newcommand{\E}{\mbb{E}}
\newcommand{\pfmc}{p_\mcl{F}^{\T{MC}}}
\newcommand{\pf}{p_\mcl{F}}

\newcommand{\pfis}{p_\mcl{F}^\T{IS}}
\newcommand{\gx}{\bs{\gamma}_{\x}}
\newcommand{\gs}{\bs{\gamma}_{s}}

\newcommand{\tC}{\tilde{\mcl{C}}}
\newcommand{\C}{\mcl{C}}
\newcommand{\mP}{\mbb{P}r}
\newcommand{\bw}{\bs{\omega}}
\newcommand{\tx}{\tilde{\x}}
\newcommand{\mg}{\textcolor{black}}

\usepackage[normalem]{ulem} 
\newcommand{\argmax}{\operatornamewithlimits{arg\,max}}

\usepackage{amsthm}
\newtheorem{lemma}{Lemma}

\newtheorem{remark}{Remark}
\newtheorem{theorem}{Theorem}
\newtheorem{assumption}{Assumption}

\title{\texttt{CAMERA}: A Method for Cost-aware, Adaptive, Multifidelity, Efficient Reliability Analysis}

\begin{document}

\begin{frontmatter}

\author{S. Ashwin Renganathan}
\address{{The University of Utah},
            {1495 E 100 S}, 
            {Salt Lake City},
            {84121}, 
            {UT},
            {USA}}
\author{Vishwas Rao}
\address{{Argonne National Laboratory},
             {9700 S Cass Ave}, 
             {Lemont},
             {60439}, 
             {IL},
             {USA}}
\author{Ionel M. Navon}
\address{{Florida State University},
            {600 W College Ave}, 
            {Tallahassee},
            {32306}, 
            {FL},
            {USA}}
            
\begin{abstract}
Estimating probability of failure in aerospace systems is a critical requirement for flight certification and qualification. Failure probability estimation involves resolving tails of probability distribution, and Monte Carlo sampling methods are intractable when expensive high-fidelity simulations have to be queried. We propose a method to use models of multiple fidelities that trade accuracy for computational efficiency. Specifically, we propose the use of multifidelity Gaussian process models to efficiently fuse models at multiple fidelity, thereby offering a cheap surrogate model that emulates the original model at all fidelities. Furthermore, we propose a novel sequential \emph{acquisition function}-based experiment design framework that can automatically select samples from appropriate fidelity models to make predictions about quantities of interest in the highest fidelity. We use our proposed approach in an importance sampling setting and demonstrate our method on the failure level set and probability estimation on synthetic test functions and two real-world applications, namely, the reliability analysis of a gas turbine engine blade using a finite element method and a transonic aerodynamic wing test case using Reynolds-averaged Navier-Stokes equations. We demonstrate that our method predicts the failure boundary and probability more accurately and computationally efficiently while using varying fidelity models compared with using just a single expensive high-fidelity  model.
\end{abstract}




\end{frontmatter}


\section{Introduction}
Flight certification by analysis is a new paradigm in which airworthiness certificates for aircraft and handling qualities compliance for aircraft subsystems can be obtained via calculations with quantifiable accuracy.\footnote{See \S 25.21 of Chapter 1, Title 14 in the electronic Code for Regulations  of the Federal Aviation Administrations.} This opens up possibilities for leveraging physics-based high-fidelity models in aircraft design, certification, and qualification. However, the high computational costs associated with high-fidelity models necessitates the development of advanced mathematical methods that judiciously spend computational resources on problems where such models may have to be queried many ($\geq \mcl{O}(10^2) - \mcl{O}(10^6)$) times.

One such problem we are interested in  is reliability analysis~\cite{rackwitz2001reliability}. Specifically, we want to estimate probabilities of failure in aerospace subsystems, whose performance are governed by models that can be queried by perturbing some control parameter. Rigorous quantification of the probability of subsystem failure is a critical requirement for flight certification. Failure probability estimation (FPE) involves accurately resolving failure boundaries (contours or level sets) in the design space and tails of probability distributions. The popular Monte Carlo (MC)~\cite{hammersley2013monte} sampling-based methods are intractable when expensive high-fidelity simulations have to be queried. Furthermore, the high dimensionality of the control parameters exacerbates the tractability issue. In this work we focus on a method that circumvents direct MC and instead seeks to develop a surrogate model of the mapping between the input parameters and the quantity of interest.  Our hope is that a surrogate model can be constructed with far fewer queries to the expensive high-fidelity model, following which we use conventional variance-reducing sampling-based techniques for FPE.

Computational models typically include one or more customizable \emph{fidelity} parameters that allow trading model accuracy for a gain in computational speedup. For example, the exact Navier--Stokes equations simulating fluid flows require numerical mesh resolutions on the order of the fine Kolmogorov scale, which shares an inverse nonlinear power relationship to the flow Reynolds number~\cite{tennekes2018first}, resulting in a large number of degrees of freedom that in turn makes the computational cost of evaluating such models intractable. However, filtering out the smaller scales, leading to simplified models such as  Reynolds-averaged Navier--Stokes  and large eddy simulations, significantly reduces computational cost for a (typically acceptable) loss of accuracy. In general, physics-based models governing conservation laws can be computationally simplified by using coarse-grained mesh resolutions and/or fewer solver iterations for convergence (in the iterative solver) to achieve a similar effect. We exploit the availability of models with tunable fidelities and propose multifidelity methods~\cite{takeno2019multi, kandasamy2017multi, klein2017fast, swersky2013multi} that can judiciously utilize the cost-accuracy trade-off in models to estimate rare-event probabilities in aerospace design.

Surrogate models for conservation laws can be efficiently constructed via reduced-order models~\cite{renganathan2018koopman, renganathan2020koopman, renganathan2020machine,renganathan2018methodology} for state variables and general regression/interpolation techniques~\cite{renganathan2021lookahead, renganathan2021enhanced, rajaram2020deep, rajaram2021empirical} for scalar quantities of interest (QoI).
In this work we are interested in a setting where a data-driven surrogate model for the QoI is adaptively constructed; that is, starting with an initial seed sample set (observations of the QoI), the surrogate model is sequentially updated with new samples that are optimally chosen to improve the prediction of the the failure boundaries in the design space. Furthermore, we account for the fact that our model of the subsystem has tunable fidelity parameters that trade predictive accuracy for computational cost. Therefore, our framework develops a surrogate model that is cost-aware while learning the relationships between the predictions between any two given fidelities. Additionally, the adaptive model construction chooses the appropriate control parameter and fidelity level to query at each step of the model-building process. For these reasons and since we are interested in reliability analysis with scalar QoI in this work, we use a Gaussian process (GP)~\cite{rasmussen:williams:2006} regression model to learn the mapping between the QoI and the augmented input-fidelity space.


The use of surrogate models for FPE is not new. \citet{li2010evaluation} show that a hybrid approach of using both surrogate models and high-fidelity function evaluations can lead to computationally cheap failure probability estimates. Moreover, they show that if the chosen surrogate model is of known accuracy,  their proposed approach provides a consistent failure probability estimate (i.e., approaching the true function in the limit). Extending this work, \citet{li2011efficient} show that the same hybrid approach can be used with importance sampling \mg{and} a cross-entropy method to estimate biasing distributions. However, neither \cite{li2010evaluation} nor \cite{li2011efficient} addresses the choice of high-fidelity model queries necessary to construct the surrogate model (that is, they do not actively train their surrogate model); furthermore they do not consider the availability of multiple-fidelity models. \citet{peherstorfer2016multifidelity} demonstrate the idea of using a surrogate model to construct biasing distributions for importance sampling (IS), but that does not include details on adaptively constructing the surrogate model. 

In the context of using GP regression as surrogate models,  a large body of  work exists, of which we briefly review what we believe are most relevant to our work. Seminal work on using adaptive GP models for \mg{contour finding} and FPE was performed by \citet{ranjan2008sequential}, who use \citet{jones1998efficient}'s expected improvement  framework to propose an improvement function that seeks to place points close (within a user-specified distance) to the target level set. They show that their acquisition function efficiently balances exploration and exploitation, and  they prove that their approach asymptotically predicts the contour exactly. \citet{bichon2008efficient} then modify the formulation of \cite{ranjan2008sequential}; both \cite{ranjan2008sequential} and \cite{bichon2008efficient} have essentially the same meaning. As we will show later, our framework uses the acquisition function from \cite{bichon2008efficient} (and hence indirectly \cite{ranjan2008sequential})---which we consider single-fidelity acquisition functions---to develop multifidelity extensions. Indeed, our acquisition function can extend those beyond \cite{bichon2008efficient} and  \cite{ranjan2008sequential} to the multifidelity setting. 

 Other noteworthy works that propose an acquisition function for adaptive GP construction are as follows. \citet{bect2012sequential} propose a lookahead acquisition function using GP models, called  stepwise uncertainty reduction. \citet{echard2010kriging} use a sequential sampling strategy that greedily (myopically) seeks points with high probability of being in the failure domain. 
\citet{picheny2010adaptive} propose a weighted integrated mean-squared prediction error criterion, where the weight function is designed to be very high in regions where the  GP mean is close to the target level set with high certainty, or when the GP uncertainty is high. \citet{gotovos2013active} use GP to construct an acquisition function that picks the smallest distance to the target level set from one of the confidence bounds. They leverage the theoretical results of GP-\mg{upper confidence bound} (UCB)~\cite{srinivas2009gaussian} to derive a bound on the number of samples required to achieve a certain accuracy in resolving the level set. \citet{dubourg2013metamodel} use an adaptive Kriging surrogate model along with importance sampling to construct a quasi-optimal biasing distribution. Similar to \cite{li2010evaluation}, they propose a hybrid approach for FPE that combines the use of both the true limit state function and its (cheap) surrogate prediction. \citet{wang2016gaussian} use GP to approximate the forward model. Then they use Bayesian experimental design to sequentially add points toward the contour estimation. They use \mg{Simultaneous perturbation stochastic approximation} (SPSA)~\cite{spall2005introduction} to appropriately optimize their stochastic acquisition function. More recently, \citet{marques2018contour} develop an entropy-based acquisition function that seeks to reduce the entropy of the failure contour. Furthermore, their method (CLoVER) allows for multiple information sources with varying costs of querying, and their acquisition function allows them to choose the information source. Their method, however, involves an integral in the input-dimension space, which can become intractable in higher dimensions. 
Recently, \citet{cole2021entropy} improve upon \cite{marques2018contour} by developing a closed-form entropy expression for the contour that is (relatively) easier to optimize. While their method does not naturally include accounting for multiple information sources, they show that they outperform CLoVER on several problems. 

Our work differs from the aforementioned work that use GP models in the following way. We propose the use of multifidelity GP~\cite{rasmussen:williams:2006} models to efficiently fuse models at multiple fidelity, where the fidelity space can be continuous or discrete. Even though \cite{marques2018contour} allows for multifidelity models, their approach depends on a Kennedy--O'Hagan~\cite{kennedy2001bayesian}-like framework (see also \cite{poloczek2017multi}), where they learn a linear correction factor between each fidelity level and the highest-fidelity model. This has two drawbacks:  (i) when the fidelity space is continuous, the approach can become intractable, and (ii) even for discrete fidelity spaces, they have to query each fidelity at a regular grid to learn the discrepancies. Our approach does not face either of these limitations. The method proposed in \cite{kramer2019multifidelity} involves the fusion of unbiased estimators from multiple fidelity models; however, the authors do not consider the choice of input and fidelities to query for an adaptive construction of the surrogate model,  as we do. An adaptive construction of the surrogate model potentially results in a more judicious choice of input-fidelity pairs to query and thus a more parsimonious use of the computational resources.
Furthermore, as stated previously, we propose a novel \emph{acquisition function} that can provide a multifidelity extension to any single-fidelity acquisition function\footnote{although ones with a closed-form expressions suit us better}, for example, \cite{ranjan2008sequential, bichon2008efficient}. Our acquisition function optimization allows for sequentially selecting samples from appropriate fidelity models to make predictions about QoI in the highest fidelity. We use our proposed approach in an importance sampling setting, similar to \cite{peherstorfer2016multifidelity}, where biasing distributions are constructed with the cheap multifidelity surrogate model, to predict failure probabilities. We demonstrate our method on several synthetic test functions as well as two real-world experiments, namely, the static stress analysis of a gas turbine blade and the transonic flow past an aircraft wing. Our main contributions are  summarized as follows:
\begin{enumerate}
    \item Introduction of a generalized multifidelity acquisition framework that extends existing single-fidelity acquisition functions in GP-based approaches for reliability analysis.
    \item Introduction of a framework that applies to both discrete and continuous fidelity spaces.
    \item Theoretical demonstration that our estimates are unbiased and our surrogate model is asymptotically consistent.
    \item Introduction of novel multifidelity synthetic test functions that can be leveraged to benchmark multifidelity methods in general.
    \item Demonstration of our method on synthetic as well as real-world application problems.
\end{enumerate}

The remainder of the article is organized as follows. In \Cref{sec:method} we provide the background and preliminaries including details about FPE and multifidelity GP models. In \Cref{sec:proposed_method} we provide details of our method and discuss key theoretical properties in \Cref{s:theory}. In \Cref{sec:num_expts} we show results of our method demonstrated on the synthetic test functions and a real-world stress analysis of a gas turbine blade.
We provide concluding remarks in \Cref{sec:conclusion}.

\section{Background and Preliminaries}
\label{sec:method}
Let $\x \in \X \subset \mbb{R}^d$ be the uncertain input to the function $f: \X \rightarrow \mbb{R}$. We assume that evaluations of $f(\x)$ are computationally expensive. We define the \emph{limit state function} $g:\mbb{R} \rightarrow \mbb{R}$ that operates on $f(\x)$; and we define failure to occur, without loss of generality, when $g(f(\x)) < 0$. The failure level set is then defined as
\begin{equation}
    \mcl{C} = \{\x \in \X: g(f(\x))=0 \}.
\end{equation}
For example, for a given failure threshold $a \in \mbb{R}$, $g(\x) = f(\x) - a$.
\begin{remark}
We restrict our analysis to affine limit state functions of the form $g(\x) = \rho f(\x) - a,~ \rho \in \mbb{R},~a \in \mbb{R}$, in order to preserve the Gaussianity of their sample paths, when we place a GP prior on $f$. This is not a restrictive assumption, however, because for nonaffine limit state functions we simply place a GP prior on $g(\x)$ instead of $f(\x)$.
\end{remark}

\subsection{Failure probability estimation}
Consider the probability space $(\mathfrak{S} , \mathfrak{F}, \mathbb{P}) = (\X, \mathfrak{B}(\X), \mathbb{P})$, where $\mathfrak{S}$ denotes the sample space that in our setting is $\X \subset \mathbb{R}^d$;  $\mathfrak{B}(\X)$ denotes the Borel sets in $\X$; and $\mathbb{P}$ denotes a probability measure. We model the uncertain parameter $\x$ as a random variable in $\mathbb{R}^d$ and  assume that $\x$ is distributed according to the distribution $\mbb{P}_\x$, which has density $q_\x$. 
The failure hypersurface defined by $\mcl{C}$ divides the parameter space into a failure set $\mathcal{F} = \{\x: g(f(\x)) \leq 0\}$ and safe set $\mcl{F}' \defeq \X \backslash \mcl{F}$. We are interested in estimating the probability of failure $p_{\mathcal{F}}$ that is defined by
\begin{align}\label{eqn:defintion}
    p_{\mathcal{F}} = \displaystyle \int_{\x \in \mcl{F}} d\mbb{P}_\x = \int_{\x \in \X} \mbb{I}_{\{\x \in \mcl{F}\}} d\mbb{P}_\x = \int_{\mcl{X}} \mathbb{I}_{\{\x \in \mcl{F}\}} q_\x(\x)~d\x,
\end{align}
where $\mbb{I}_{\{\x \in \mcl{F}\}}$ denotes the indicator function that takes the value $1$ when $\x \in \mcl{F}$ and $0$ otherwise. MC methods are often used to estimate $ p_{\mathcal{F}}$; the idea is to draw $N$ independent and identically distributed (iid) samples of $\x$, $\{\x^i,~i=1,\ldots,N\}$, $\x^i \sim \mbb{P}_\x$ and approximate  $p_{\mathcal{F}}$ as
\begin{align}\label{eqn:mcs}
    p_\mcl{F} \approx p_{\mathcal{F}}^{\textrm{MC}} = \frac{1}{N}\sum_{i=1}^{N}  \mathbb{I}_{\{\x^i \in \mcl{F}\}}\,,
\end{align}
where ${\x}^i$ is the $i$th realization of $\x$. The MC estimator is unbiased, as shown below.
\begin{equation*}
    \begin{split}
        \E [\pfmc] =& \f{1}{N}\sum_{i=1}^N \E[\mathbb{I}_{\{\x^i \in \mcl{F}\}}]\\
        =& \f{1}{N}\sum_{i=1}^N \int_{\mcl{X}} \mathbb{I}_{\{\x \in \mcl{F}\}} q_\x(\x)~d\x = \f{1}{N} \sum_{i=1}^N \pf = \pf.
    \end{split}
\end{equation*}
The variance of the MC estimator is given by
\[\mbb{V}(p^{MC}_\mcl{F}) = \f{p^{MC}_\mcl{F} (1 - p^{MC}_\mcl{F})}{N}.\]
Using the approach described by \eqref{eqn:mcs} requires a large number of samples in order to obtain a failure probability estimate of desired accuracy. For instance, consider $p_\mcl{F} = 10^{-3}$ and a root mean-squared error in the estimate of $10^{-4}$ (and therefore a variance of $10^{-8}$). Then the number of samples necessary to achieve this accuracy via MC is $\approx 10^5$. This could be prohibitively expensive in the reliability analysis of complex engineered systems with expensive forward models. A common means to (partially) circumvent this shortcoming is to employ an auxiliary distribution that concentrates the samples around the failure region \cite{kahn1953methods, srinivasan2002importance}; this approach is referred to as importance sampling. Equation \eqref{eqn:defintion} can be equivalently written as 
\begin{equation}
    p_{\mathcal{F}} = \displaystyle \int_{\x \in \X} \mbb{I}_{\{\x \in \mcl{F}\}} \f{q_\x(\x)}{q'_\x(\x)} q'_\x(\x) d\x = \mathbb{E}_{\mbb{P}'_\x}\left\lbrack \mathbb{I}_{\{\x \in \mcl{F}\}} \f{q_\x(\x)}{q'_\x(\x)} \right\rbrack\,
    \label{e:is}
\end{equation}
where $q'_{\x}$ is the density of the biasing distribution $\mbb{P}_\x'$. The key idea is that samples drawn from the biasing distribution are more likely to be from the set $\mathcal{F}$. From \eqref{e:is}, the IS estimate of $p_{\mathcal{F}}$ can be written as 
\begin{align}\label{eqn:isexpression}
    p_{\mathcal{F}}^{\textrm{IS}} = \frac{1}{N} \sum_{i=1}^N \mathbb{I}_{\{ \x^i \in \mathcal{F} \}} w(\x^i)\,,
\end{align}
where $w(\x^i) = \f{q_\x(\x^i)}{q'_\x(\x^i)}$
are the importance weights and $\x^i \sim \mbb{P}'_\x$. The estimator in \eqref{eqn:isexpression} is also unbiased, provided that $q'_{\x}(\x) > 0$ whenever $\mbb{I}_{\{\x \in \mcl{F}\}} \times q_\x(\x) > 0$. Let $\mcl{Q} \subset \X \defeq \{\x \in \X : \mbb{I}_{\{\x \in \mcl{F}\}} \times q_\x(\x) > 0\}$. Then

\begin{equation*}
    \begin{split}
        \E [\pfis] =& \f{1}{N}\sum_{i=1}^N \E[\mathbb{I}_{\{\x^i \in \mcl{F}\}} w(\x^i)]\\
        =& \f{1}{N}\sum_{i=1}^N \int_{\mcl{X}} \mathbb{I}_{\{\x \in \mcl{F}\}} \f{q_\x(\x)}{q'_\x(\x)} q'_\x(\x)~d\x = \f{1}{N}\sum_{i=1}^N \int_{\mcl{Q}} \mathbb{I}_{\{\x \in \mcl{F}\}} \f{q_\x(\x)}{q'_\x(\x)} q'_\x(\x)~d\x\\
        =& \f{1}{N} \sum_{i=1}^N \pf = \pf.
    \end{split}
\end{equation*}
As mentioned previously, we will construct the biasing distribution via a surrogate model in order to keep the computational costs tractable. As we will show later, similar to the construction of our surrogate model, our biasing distribution is also adaptively improved. We  now introduce the surrogate modeling technique via GP regression.

\subsection{Multifidelity Gaussian process regression}
\label{subsec:gp}
Crucially, we assume that $f(\x)$ is approximated by models $\hat{f}(\x, s)$, where $\x$ are common inputs to all the models and $s \in \Sc\subset \mbb{R}$ is a tunable \emph{fidelity} parameter for each model. For the sake of simplicity, in this work we set $s$ to be scalar valued and without loss of generality $\Sc= [0,1]$, where $s=1$ and $s=0$ represent the model at the highest and lowest fidelity, respectively; therefore,
\[ \hat{f}(\x, 1) \defeq f(\x).\]
Additionally, we assume there is known a cost function $c(s): \Sc \rightarrow \mbb{R}$, which is monotonic increasing in $s$ and models the computational cost of querying $\hat{f}$ at a specific fidelity. Note that we assume $c(\cdot)$ is independent of $\x$ for the sake of simplicity, but our framework applies to more general  $c(\x,s)$ as well. Overall, we are interested in computing \eqref{e:is} by querying $\hat{f}$, while keeping the overall cost lower than computing \eqref{e:is} by querying $f(\x)$ alone. Our multifidelity method depends on learning 
a GP model that maps the augmented input-fidelity space $(\x, s) \in \X \times \Sc$ to output quantities of interest $\hat{f}(\x, s)$. 
\begin{remark}
We mention that by assuming a monotonic cost function $c(s)$ and the fact that fidelity increases as $s$ increases in $[0,1]$, we assume that a hierarchy of fidelities exists among the set of models. However, \mg{such a hierarchy is not a necessary condition for} our method \mg{to be applicable}. In the absence of any known hierarchy among the models, the highest fidelity can be chosen as $\argmax_{s \in \Sc}~c(s)$. If $c(s)$ is unknown a priori, it can be learned, for example,  by placing another GP prior on it and making a few observations.
\end{remark}

We assume that we can make evaluations of the function $\hat{f}$ at a given $\x$ and fidelity level $s$. That is,
\begin{equation}
    \hat{y}_i = \hat{f}(\x_i, s_i), \quad i= 1,\ldots,n.
\end{equation}
The key idea then is to specify GP prior distributions on the
$f$. That is, $\hat{f}(\x, s) \sim \mcl{GP}\left(0, k( (\x, s),
(\x', s'))\right)$. The covariance function (or \emph{kernel}) $k$ captures the correlation between the observations in the joint $(\x, s)$ space; here we use the product composite form given by $k((\x, s), (\x', s')) = k_{\x}(\x,\x'; \gx) \times k_s(s,s'; \gs) $, where $\gx \in \mbb{R}_+^d$ and $\gs \in \mbb{R}_+^d$ parameterize the covariance functions for $\x$ and $s$, respectively. We estimate the
GP hyperparameters 
$\bs{\Omega}= \lbrace \gx, \gs \rbrace$
from data by maximizing the marginal
likelihood. We choose an anisotropic Mat\'{e}rn-type kernel for both $k_\x$ and $k_s$. The
posterior predictive distribution of the output $Y$, conditioned on available observations from the oracle, is given by~\cite{rasmussen:williams:2006}
\begin{equation}
    \begin{split}
        Y(\x, s) |
      \D_n, \bm{\Omega} &\sim \mcl{GP}(\mu_n(\x, s), \sigma^2_n (\x, s)),~ \quad \D_n \defeq  \lbrace
(\x_i, s_i), \hat{y}_i \rbrace _{i=1}^n \\
\mu_n(\x, s) &= \mbf{k}_n^\top [\mbf{K}_n]^{-1} \mbf{y}_n\\
\sigma^2_n(\x, s) &=  k((\x, s), (\x, s)) - \mbf{k}_n^\top [\mbf{K}_n]^{-1} \mbf{k}_n,
    \end{split}
    \label{e:GP}
\end{equation}
where $\mbf{k}_n$ is a vector of covariance between $(\x, s)$ and all observed
points in $\D_n$; $\mbf{K}_n$ is a sample covariance matrix of observed points in
$\D_n$; $\mbf{I}$ is the identity matrix; and $\mbf{y}_n$ is the vector of all
observations in $\D_n$. Equation \eqref{e:GP} is then used as a surrogate model for
$f$ in FPE. Note that $\mu_n$ and $\sigma_n^2$
are the posterior mean and variance of the GP, respectively, where the
subscript $n$ implies the conditioning based on $n$ past observations. Our surrogate model for the limit state function then is given as
\[ \tilde{g}_n(\x, s) \defeq g(\mu_n(\x, s)). \]
\begin{remark}
Note that we assume our models at all fidelities are deterministic and observations are noise-free. Estimating failure probabilities with noisy observations is an interesting problem, but outside the scope of our work.
\end{remark}
\begin{remark}
For the remainder of this manuscript, we assume that the high fidelity model $\hat{f}(\x, 1) = f(\x)$ is our ground truth. On the other hand, our approach naturally extends to situations involving disparate information sources, where the ground truth could be a physical experiment.
\end{remark}
An important component of our proposed method is that the data required to build the surrogate model is adaptively selected. This ensures that the inputs and fidelities are judiciously chosen to be tailored to our specific goals. We call this adaptive approach ``sequential experimental design," which we discuss next.

\subsection{Acquisition-based sequential experiment design}
We use principles from Bayesian optimization (BO)~\cite{shahriari2015taking, frazier2018tutorial, snoek2012practical} to sequentially update the GP to improve our knowledge about our goal---in this case, estimation of $\mcl{C}$ and the subsequent FPE. Akin to BO, we  proceed by defining an acquisition function $\alpha(\x,s)$, in terms of the GP posterior that is
optimized (typically maximized) to select---given previous $n$ observations---the next point $\x_{n+1}$ at $s_{n+1}$, and the process continues until a budget for computation is reached; see Algorithm~\ref{a:BO}. The choice of $\alpha(\x, s)$ is key to the performance of the quality of the sequential point selection, which we discuss in the next section. 

Naturally, the acquisition function is designed according to the specific needs of the sequential process. Classical acquisition functions for optimization include the
probability of improvement (\texttt{PI})~\cite{kushner1964new}, the expected improvement
(\texttt{EI}) \cite{jones1998efficient, mockus1978application}, and the GP upper
confidence bound (\texttt{UCB})~\cite{srinivas2009gaussian}; these approaches work with a probabilistic estimate of the optimum via the GP. Other acquisition functions for optimization include entropy-based approaches~\cite{wang2017max, hennig2012entropy, villemonteix2009informational}. Similarly, acquisition functions for active learning (purely learning the function $f(\x)$) also exist~\cite{mackay1992information, binois2019replication}. Therefore,  acquisition functions to learn a level set (e.g., failure boundary) can also be defined, such as the expected level set improvement of \citet{ranjan2008sequential} and \citet{bichon2008efficient}, entropy-based approaches of \citet{cole2021entropy} and \citet{marques2018contour}, stepwise uncertainty reduction by ~\citet{bect2012sequential}, and other methods as in \citet{picheny2010adaptive}. 

In all the aforementioned approaches for failure level set estimation, we point out that application to multifidelity models is not straightforward. Although \mg{in} \cite{marques2018contour} \mg{the authors} show their applicability with multiple information sources, that requires learning individual discrepancies between models at each fidelity and the highest fidelity model. This can be computationally  demanding in high-dimensional problems or when some of the lower-fidelity models are not necessarily  cheap (compared with the highest-fidelity models). Consequently, such methods are restricted to discrete fidelity spaces with small cardinality. Furthermore, outside of the reliability literature, autoregressive approaches such as that of \cite{kennedy2001bayesian} and \cite{perdikaris2017nonlinear} suffer from the same limitation.
Therefore, in this work we propose a new acquisition approach that applies to both discrete and continuous fidelity spaces. Additionally, our approach can extend any existing acquisition function to the multifidelity setting, as we will show later. 
We  provide more details in the following section.

\begin{algorithm2e}[t]
\textbf{Given:} $\D_n = \lbrace
(\x_i, s_i), \hat{y}_i \rbrace _{i=1}^n$, 
 total budget $q$, 
 and GP hyperparameters $\bm{\Omega}$ \\
\KwResult{$(\x_q, \hat{y}_q)$}
  \For{$i=n+1, \ldots, q$, }{
  Find $\x_i, s_i \in \underset{(\x, s) \in \X \times \mcl{S}}{\argmax}~ \alpha(\x, s)$ \qquad (acquisition function maximization)\\
    Observe $\hat{y}_i$ = $f(\x_i, s_i)$\\ 
    Append $\D_i = \D_{i-1} \cup \lbrace (\x_i, s_i), \hat{y}_i \rbrace$\\ 
    Update GP hyperparameters $\bm{\Omega}$ \\
 }
 \caption{Generic acquisition-based sequential design}
 \label{a:BO}
\end{algorithm2e}

\section{Proposed Method}
\label{sec:proposed_method}

\subsection{Multifidelity acquisition functions}
\label{subsec:lookahead}
Drawing from the work of \citet{renganathan2021lookahead}, where they propose an acquisition function for time-dependent problems with a finite horizon, we define a multifidelity \emph{value} function framework that we use to develop our multifidelity acquisition function. Given the data from $n$ observations $\D_n$, we define the value function $\upsilon$ by 
\begin{equation}
    \upsilon_n(\x, s; \bs{\omega}) \defeq \mbb{E}_{y\sim Y|\D_n}\left[
    h(y(\x, s); \bs{\omega})     
    \right],
    \label{e:value_function}
\end{equation}
where $h$ is a scalar-valued function, $\bs{\omega} \in \mbb{R}^p$ parametrizes $h$, 
$Y|\D_n$ is the GP posterior 
given $n$ observations, and $y(\x, s)$ is a sample path realized from $Y|\D_n$. The reasoning behind defining a value function  as in \eqref{e:value_function} is as follows: 
\begin{enumerate}
    \item we can quantify the utility/value at the highest fidelity (our ground truth) by simply setting $s=1$, i.e. 
    \[ \upsilon_n(\x, 1; \bw) : \hfil \T{utility at the highest fidelity}  \]
    \item computing this utility is cheap because it depends only on the posterior \mg{multifidelity} GP
    \item \eqref{e:value_function} generalizes several of the existing (single-fidelity) acquisition functions in GP-based reliability analysis, and finally
    \item as we will show below, \eqref{e:value_function} allows a multifidelity extension of any given single fidelity acquisition function.
\end{enumerate}
For example, some of the existing single fidelity acquisition functions for GP-based reliability analysis can be expressed as special cases of \eqref{e:value_function} as shown in \Cref{tab:existing_acq},
\begin{table}[h!]
    \centering
    \begin{tabular}{c|c|c|c}
    \hline
        $\bs{\omega}$ & $h(\x, s)$ & $v_n(\x, s)$ & acquisition function  \\
        \hline
        $\{a, \eta\}$ & $\delta ^2(\x, s) - \min\{(y(\x, s) - a)^2, \delta^2(\x, s)\}$ & $EI_r(\x, s)$ & \citet{ranjan2008sequential}, see \eqref{e:ranjan}. \\
        $\{a, \eta\}$ & $\delta(\x, s) - \min\{| (y(\x, s) - a)|, \delta(\x, s)\}$ & $EI_b(\x, s)$ &  \citet{bichon2008efficient}, see \eqref{e:Exp_RandR}. \\
        \hline
    \end{tabular}
    \caption{Parametric representation of existing acquisition functions via our framework \eqref{e:value_function}.}
    \label{tab:existing_acq}
\end{table}
where $a$ is the failure threshold and $\delta^2(\x) = \eta \sigma^2(\x),~ \eta \in \mbb{R}_+$ defines the width of the uncertainty band around the GP posterior mean prediction. Then, the generalized multifidelity acquisition function is defined as 
\begin{equation}
\begin{array}{rl}
\alpha(\x,s)
&=  \displaystyle \int \max_{\x' \in \X}~ \left[ \underbrace{\upsilon(\x', 1; \bs{\omega})}_{\T{value at highest fid.}} | \, 
\underbrace{\D_n}_{\T{prev. obs.}} \bigcup \{ \underbrace{(\x,s), y}_{\T{candidate obs.}} \}
\right]~ q_y(y|\D_{n}) ~dy, 
\end{array}
\label{e:twostep}
\end{equation}
where $q_y(y|\D_{n})$ is the probability density of the GP posterior distribution conditioned on $\D_n$. Notice that in the inner maximization we use $\upsilon(\cdot, 1)$;  this indicates that we place our value on the highest fidelity. In essence, \eqref{e:twostep} computes the expected value of the maximum value function at the highest fidelity due to a candidate input-fidelity pair $(\x, s)$, where the expectation is with respect to the posterior GP. We write \eqref{e:twostep}  more concisely as 
\begin{equation}
  \alpha(\x,s) = \mbb{E}_n \left[\max_{\x' \in \X}~\upsilona{n}{1}(\x',1;\bs{\omega}) \right].
    \label{e:E_twostep}
\end{equation}
In \eqref{e:E_twostep}, note that $\upsilona{n}{1} = \upsilon | \D_n \bigcup \{ (\x,s), y \}$, which is the value function given the union of $\D_n$ and the $n+1$th observation beinga drawn from the GP posterior $Y|\D_n$. In general, 
\eqref{e:E_twostep} is not guaranteed to admit a closed-form expression. Hence we make the
MC approximation
\begin{equation}
    \alpha(\x, s) \approx \hat{\alpha}(\x, s) \defeq \frac{1}{N} \sum_{j=1}^N \max_{\x' \in \X}~\upsilona{n}{1}^j(\x', 1),
    \label{e:twostep_mc}
\end{equation}
where $\upsilona{n}{1}^{j}(\x', 1) = \upsilon(\x', 1)|\D_n \bigcup \{(\x, s), y^j \}$ with the $N$ iid samples $\{y^j, \forall \,j=1,\ldots,N\}$ drawn from $Y|\mcl{D}_{n}$ at $(\x, s)$. 

The inner maximization in \eqref{e:twostep_mc} can be approximated by evaluating the value function $\upsilona{n}{1}$ on a discrete set $\hat{\X}$ and the maximization approximated as (for $\hat{\x}_i \in \hat{\X}$)
\[ \max_{\x' \in \X}~\upsilona{n}{1}^j(\x', 1) \approx \max_i ~\upsilona{n}{1}^j(\hat{\x}_i, 1).\]
As noted above, the value function is a function of the GP only, and hence its evaluation on $\hat{\X}$ is computationally  cheap. 
We now sequentially choose points by maximizing the cost-normalized acquisition function defined as
\begin{equation}
    (\x_{n+1}, s_{n+1}) = \argmax_{(\x,s) \in \X \times \Sc}~\f{ \hat{\alpha}_2(\x, s) }{c(s)}.
    \label{e:costaware}
\end{equation}
In \eqref{e:costaware}, since we normalize the acquisition function by the cost of evaluations $c(s)$, the optimization penalizes decisions at high $s$, preventing expensive high-fidelity evaluations unless they offer a significant increase in the maximum value function compared with lower fidelities. For this reason, we consider our approach \emph{cost-aware} and thus name it \texttt{CAMERA}: cost-aware, adaptive, multifidelity, efficient reliability analysis. We now discuss the choice of the value function. 

\subsection{Choice of value function}
As previously mentioned, we propose a multifidelity acquisition function that provides a framework to \mg{extend} existing acquisition functions, defined for a single-fidelity setting, \mg{toward} a multifidelity setting. Therefore, while the choice of the value function can be practically any existing acquisition function, for the experiments in this work we choose the \emph{improvement}-based acquisition functions for contour estimation by \cite{ranjan2008sequential} and \cite{bichon2008efficient}. \citet{ranjan2008sequential} propose the following improvement function, which we extrapolate to the multifidelity setting as
\begin{equation}
    I_r(\x, s) = \delta ^2(\x, s) - \min\{(y(\x, s) - a)^2, \delta^2(\x, s)\},
    \label{e:ranjan}
\end{equation}
where  $a$ is the target level and $\delta^2(\x, s) = \eta \sigma^2(\x, s),~ \eta \in \mbb{R}_+$ defines the width of the uncertainty band around the GP posterior mean prediction. The intuition behind \eqref{e:ranjan} is that improvement is defined as getting points closer to the target level set or in places where the uncertainty of the GP predictions is large, while still being within a certain \emph{confidence region} of the GP prediction. Consequently, $I_r(\x) = 0$, when $(y(\x) -a)^2 > \epsilon^2(\x)$ and therefore regions far away from the potential contour locations are avoided. \citet{bichon2008efficient} modify \eqref{e:ranjan} as 
\begin{equation}
    I_b(\x, s) = \delta(\x, s) - \min\{| (y(\x, s) - a)|, \delta(\x, s)\},
    \label{e:bichon}
\end{equation}
which results in a closed-form expression for the expected improvement
\begin{equation}
\begin{split}
    EI_b(\x, s) = &\delta(\x, s) \left[ \Phi(z^+) - \Phi(z^-) \right] - \sigma(\x, s)  \left[ 2\phi(z) - \phi(z^-) - \phi(z^+) \right] \\
    & + (\mu(\x) - a) \left[ 2\Phi(z) - \Phi(z^-) - \Phi(z^+) \right], 
\end{split}
\label{e:Exp_RandR}
\end{equation}
where $z=\f{a - \mu(\x, s)}{\sigma(\x, s)}$, $z^\pm = \f{a^\pm - \mu(\x, s)}{\sigma(\x, s)}$, and $a^{\pm} = a \pm \delta(\x, s)$. Note that \eqref{e:ranjan} and \eqref{e:bichon} practically have the same meaning, but \eqref{e:bichon} leads to a more convenient analytical form upon taking the expectation, as opposed to that in \eqref{e:ranjan}.  We therefore propose to specify our value function, also extending it to the multifidelity setting, as
\begin{equation}
    \upsilon(\x, 1) = EI_b(\x, 1),
\end{equation}
and define our acquisition function as follows:
\begin{equation}
    \alpha(\x, s) = \mbb{E}_{y \sim Y_n} \left[\max_{\x' \in \mcl{X}} EI_b(\x', 1) | \D_n \bigcup \{(\x, s), y\} \right],
    \label{e:LELSI}
\end{equation}
which essentially quantifies the \emph{average} \mg{value} of the maximum $EI_b$ at the highest fidelity ($s=1$) due to a candidate decision-fidelity pair $(\x, s)$. Note that we extend the acquisition function in \cite{bichon2008efficient} to general multifidelity cases, where the fidelity space could be continuous or discrete. The acquisition function in \eqref{e:LELSI} is then plugged into \Cref{a:BO} to obtain a final surrogate model $\tilde{g}_q(\x, 1)$ of $\hat{f}(\x, 1)$.

\subsection{Importance sampling}
\label{ss:mis}
The last step of our approach is the estimation of the failure probability. The multifidelity surrogate model that is adaptively constructed is now used to construct a biasing distribution. This distribution is then used to generate a sample set using which the failure probability is  estimated.

Let $q'^{(n)}_\x$ be the density of the biasing distribution, $\mbb{P}'^{(n)}_\x$, after $n$ rounds of updating our multifidelity surrogate model. Then, the importance sampling estimate is given by
\begin{align}\label{eqn:isestimate}
    p_{\mathcal{F}}^{\textrm{IS}} = \frac{1}{N} \sum_{i=1}^N \mathbb{I}_{\{ \x^i \in \mathcal{F} \}} \f{q(\x^i)}{q^{'(n)}_\x}\,,
\end{align}
where $\x^i,~i=1,\ldots,N$ are iid samples drawn from $\mbb{P}^{'(n)}_\x$.

\subsection{Overall \texttt{CAMERA} algorithm}
\label{ss:fpe}
We are now ready to summarize our overall algorithm. We use a Gaussian mixture model (GMM)~\cite{jerome2001elements} to learn the biasing density $q'^{(n)}_\x$ from a sample set generated from the GP. For all the cases, we set the number of components of the GMM to 25 and restrict the covariance matrix to be diagonal. A few trials indicated that these settings are sufficient to get a good approximation to the biasing density across all the experiments. The expectation-maximization algorithm~\cite{dempster1977maximum, wu1983convergence} is used to learn the hyperparameters of the GMM. Once fit, this distribution is sampled to generate a biasing set that is evaluated on the high-fidelity model $\hat{f}(\cdot, 1)$ to compute the IS estimate of failure probability. The exact steps are summarized in \Cref{alg:camera}.

\LinesNumbered
\begin{algorithm2e}[h!]
\SetAlgoLined
\textbf{Input:} 
\begin{itemize}
    \item Surrogate model: Multifidelity GP posterior of $f$, $\mcl{GP}(\mu_n(\x, s), \sigma^2_n(\x, s))$ and observed data $\D_n$.
    \item Cumulative cost $C_n = \sum_{j=1}^n c(s_j)$, total budget $B$.
    \item High-fidelity sample size $N$.
    \item Biasing sample size $M$.
\end{itemize}
\KwResult{$\pfis$}
  \While{$C_i \leq B,~ \forall~ i=n+1, \ldots$, }{
  \begin{enumerate}
      \item Find $\x_i, s_i \in \underset{(\x, s) \in \X \times \mcl{S}}{\argmax}~ \hat{\alpha}(\x, s) / c(s)$ \hfill \# acquisition function maximization
      \item Observe $\hat{y}_i$ = $f(\x_i, s_i) $ \hfill \# new multifidelity observation
      \item Cumulative cost $C_i = C_{i-1} + c(s_i)$
      \item Append $\D_i = \D_{i-1} \cup \lbrace (\x_i, s_i), \hat{y}_i \rbrace$ \hfill \# append data set
      \item Update GP hyperparameters $\bm{\Omega}$ \hfill \# recompute GP hyperparameters
  \end{enumerate}
 }
 Let $q>n$ be the last iteration of updating the surrogate model. 
 \begin{enumerate}
     \setcounter{enumi}{5}
     \item Fit a biasing distribution $\mbb{P}'^{(q)}_{\x}$ with density $q'^{(q)}_\x$
    \begin{enumerate}
        \item \label{it:biasing_dist} Obtain a sample set of cardinality $M$, using the \mg{multifidelity} surrogate model \mg{at $s=1$}: \[\X' = \{\x_i \in \X: \tilde{g}_q(\x_i, 1) - a =  \mu_q(\x_i, 1) - a \leq 0, ~i = 1, \ldots, M \}.\]
        \item Fit a GMM with $\X'$ to obtain the biasing distribution $\mbb{P}'^{(q)}_{\x}$ with density $q'^{(q)}_{\x}$.
    \end{enumerate}
     \item Obtain a biasing sample set  
     \[\X'_\T{biasing} = \{\x_k \sim \bar{\mbb{P}}^{'q}_{\x},~k=1, \ldots, N\}. \]
     \item Compute indicator function (with {\it high fidelity} samples) and IS weights \[ \bI = \{\mbb{I}_{\{\x_k \in \mcl{F} \}}\},~~\mbf{w} = \left\{ \f{q_\x(\x_k)}{q'^{(q)}_{\x}(\x_k)} \right\},~k=1, \ldots, N,\]
     \item compute the IS failure probability estimate \\
        
        \[ \pfis = \f{\bI ^\top \mbf{w}}{N}\]
 \end{enumerate}
 \caption{\texttt{CAMERA}: Cost-aware Adaptive Multifidelity Efficient Reliability Analysis.}
 \label{alg:camera}
\end{algorithm2e}


Note that $M$ in step \ref{it:biasing_dist} in \Cref{alg:camera} can be picked arbitrarily large since it does not involve any high-fidelity model queries, only the surrogate model. On the other hand, $N$ must be tractable since it involves high-fidelity queries. Instead of directly setting $M$, we first uniformly sample $\X$ with $10^7$ samples and filter out the sample set to obtain $\X'$ as in step 2 using the surrogate model.

\section{Theoretical Properties}
\label{s:theory}

Recall that our surrogate-based IS estimator is unbiased. We now show that our adaptive surrogate model is consistent; that is, with high probability the predicted level set converges to the true level set asymptotically. We first demonstrate that the points selected by maximizing our multifidelity acquisition function~\eqref{e:costaware} asymptotically reduce the GP posterior variance to $0$ everywhere; see \Cref{lm:unique_points} and \Cref{lm:gpvariance}. Then, in \Cref{thm:ass_consistency}, we show that our surrogate model is asymptotically consistent.

We begin by making a smoothness assumption on our unknown function $f$, followed by an assumption on the limit state function.
\begin{assumption}
\label{ass:lipschitz}
We assume that the sample paths drawn from the GP prior on $f$ are twice continuously differentiable. Such a GP is achieved by choosing the covariance kernel to be four times differentiable~\citep[Theorem 5]{ghosal2006posterior}, for  example, the squared-exponential or Matérn-class kernel. Then, the following holds (for $L \in \mbb{R}_+$):
\[\mP \left[ |f(\x) - f(\x')| > L \|\x - \x'\| \right] \leq e^{-L^2 / 2}.\]
\end{assumption}

\begin{assumption}
\label{ass:affine_g}
The limit state function $g$ is an affine function of $f$, defined as $g(f(\x)) = \rho f(\x) - a$. W.l.o.g., we pick $\rho = 1$.
\end{assumption}

First, we  show that our proposed two-step lookahead acquisition function will never choose the same point twice. We present this in Lemmas~\ref{lm:g_identity} and \ref{lm:unique_points}. 
\begin{lemma}
{\label{lm:g_identity}} 
Without the $\max$ operator in \eqref{e:twostep}, the multifidelity acquisition function reduces to the value function. That is, 
\[\mbb{E}_{y} [\upsilon(\x', 1) | \D_n \bigcup \{\x, y\}] = \upsilon(\x', 1) | \D_n \].
\end{lemma}
\begin{proof}
Let us write 
\[ \mbb{E}_{y} [\upsilon(\x', 1) | \D_n \bigcup \{(\x,s), y \}] = \mbb{E}_{y(\x,s)} [\upsilon_n(\x',1) | y(\x,s)], \]
where $\upsilon_n(\cdot,\cdot) = \upsilon(\cdot,\cdot) | \D_n$. Using the definition of the value function, we have
\begin{equation}
    \begin{split}
        \mbb{E}_{y(\x,s)} [\upsilon_n(\x',1) | y(\x,s)] =& \mbb{E}_{y(\x,s)} [ \mbb{E}_{y(\x',1)} [h(y(\x',1), \bw) ] | y(\x,s)] \\
        =& \mbb{E}_{y(\x,s)} [ \mbb{E}_{y(\x',1)} [h(y(\x',1), \bw) | y(\x,s) ] ] \\
        =& \mbb{E}_{y(\x',1)} [h(y(\x',1), \bw)] \\
        =& \upsilon_n(\x',1).
    \end{split}
\end{equation}
\end{proof}

\begin{lemma}
\label{lm:unique_points}
$\forall i,j \in \{1, \ldots, n\}$, $(\x_i, s_i) \neq (\x_j, s_j)$, where $(\x_i, s_i)$ and $(\x_j, s_j)$ are the input-fidelity pairs that maximize $\alpha(\x,s) | \D_{i-1}$ and $\alpha(\x,s) | \D_{j-1}$, respectively, \mg{and hence are points previously observed by the algorithm}.
\end{lemma}
\begin{proof}
For a given $(\x,s) \in \X \times \Sc$, let $\sigma^2_n(\x, s) = 0$;   in other words, a noise-free observation was previously made at $(\x, s)$. Then,
\begin{equation*}
    \begin{split}
        \alpha(\x,s) = \mbb{E}_{y(\x, s)}[\max_{\x' \in \X}\upsilon(\x', 1)|\{(\x, s), y(\x, s)\}] = \max_{\x' \in \X}\upsilon(\x', 1)|\{(\x, s), \mu_n(\x)\} = \max_{\x' \in \X}\upsilon(\x', 1),
    \end{split}
\end{equation*}
which follows from the fact that $y(\x, s) | (\sigma_n^2(\x, s) = 0) = \mu_n(\x, s)$ and  that the value function conditioned on a previously observed point is unchanged. Now consider a candidate point (i.e., an unobserved point) $(\x'', s'') \in (\X \times \Sc) \backslash (\x, s)$,
\begin{equation*}
    \begin{split}
        \alpha(\x'', s'') = \mbb{E}_{y(\x'', s'')}[\max_{\x' \in \X}\upsilon(\x', 1)|\{ (\x'', s''), y(\x'', s'')\}] \geq \max_{\x' \in \X}\upsilon(\x', 1) = \alpha(\x, s),
    \end{split}
\end{equation*}
where the inequality follows from Jensen's inequality~\citep[Ch. 3]{rudin1970real} and we have also used \Cref{lm:g_identity}. Furthermore, the equality holds only if $(\x, s) = (\x'', s'')$ or $\sigma^2_n(\x, s) = 0 = \sigma^2_n(\x'', s'')$, and therefore $(\x, s) \neq (\x'', s'')$, since $(\x'', s'')$ is unobserved and hence $\sigma^2_n(\x'', s'') > 0$.
\end{proof}
Now, using \Cref{lm:unique_points}, we will show that the multifidelity acquisition strategy reduces the GP posterior variance to $0$ at $s=1,~\forall \x \in \X$.

\begin{lemma}
\label{lm:gpvariance}
Let $\{\x_i\},~i=1,\ldots,n$ be a sequence of points selected via our multifidelity acquisition strategy. Then, $\lim_{n \rightarrow \infty}~\sup_{\x \in \X} \sigma^2_n(\x, 1) = 0$. 
\end{lemma}
\begin{proof}
W.l.o.g, let $k((\x,s), (\x,s)) = 1$. Recall that the posterior variance of the GP is given by
\[ \sigma_n^2(\x, s) = k((\x,s), (\x,s)) - \mbf{k}_n^\top \mbf{K}_n^{-1}\mbf{k}_n.\]

Note that as $n \rightarrow \infty$, and if \Cref{lm:unique_points} holds, then $\mbf{k}_n \rightarrow \mbf{K}^{:,i}_n$, where by $\mbf{K}^{:,i}_n$ we mean the $i$th column of $\mbf{K}_n$, $i \in [n]$, and therefore $\mbf{k}_n^\top \mbf{K}_n^{-1} = \mbf{e}_i$, where $\mbf{e}_i$ is an $n$-vector with $1$ at the $i$th location and $0$ elsewhere. Then
\[ \lim_{n \rightarrow \infty} \sigma_n^2(\x,1) = 1 - \mbf{K}^{:,i~\top}_n \mbf{e}_i = 1 - k((\x_i, 1), (\x_i, 1)) = 0.\]
\end{proof}
We are now ready to show that with high probability, the predicted failure level set via our multifidelity surrogate model converges to the true failure level set. For this purpose we assume a discrete set $\tC \subset \mcl{C}$ and show that the absolute difference between the surrogate prediction at $\tC$ and the true failure level set is bounded with high probability.

\begin{theorem}[Asymptotic consistency]
\label{thm:ass_consistency}
Let $\delta \in (0, 1)$ and Assumptions \ref{ass:lipschitz} and \ref{ass:affine_g} hold, with $L^2 = 2 \log \left(\f{2}{\delta}\right)$. Let $\tC$ denote a discretization of $\C$ with $|\tC| < \infty$, where we evaluate the true failure boundary. Let $\epsilon \geq 0$, and $a_n = \sqrt{2\log(b_n |\tC|)/\delta}$, where $\sum_{n=1}^\infty 1/b_n = 1$. Then the predicted failure level set converges to the true failure level set with high probability; that is, the following holds:

\[ \mP \left[\forall \x\in \X, ~ \lim_{n \rightarrow \infty}~|g(f(\x, 1)) - g(\mu_{n-1}(\x, 1))| \leq \epsilon \right] > 1 - \delta. \]
\end{theorem}
\begin{proof}
The  proof is based on \cite{srinivas2009gaussian}, but we differ in the choice of constants. We first present a simple identity that bounds the probability that a Gaussian random variable is greater than some value. Given $r \sim \mcl{N}(0, 1)$ and $c > 0$, consider the following:
\begin{equation}
    \begin{split}
        \mP(r > c) =& \int_c^\infty \f{1}{\sqrt{2\pi}} e^{-r^2 / 2} dr = e^{-c^2 /2}  \int_c^\infty \f{1}{\sqrt{2\pi}} e^{-(r - c)^2 / 2} e^{-c(r-c)} dr\\
        \leq & e^{-c^2 / 2} \int_c^\infty \f{1}{\sqrt{2\pi}} e^{-(r - c)^2 / 2} dr \leq \f{1}{2}e^{-c^2 / 2},
    \end{split}
\end{equation}
where the second line is due to the fact that $0 \leq e^{-c(r-c)} \leq 1$ for $r \geq c$ and $c>0$ and the last line from the fact that the integrand is the density of $\mcl{N}(c, 1)$. Using this result and the prior assumption that the unknown function $f$ has a GP prior, we state, with $c = a_n$ and a fixed $\x_i \in \tC$,
\begin{equation}
    \begin{split}
    \mP \left[ |g(f(\x_i,1)) - g(\mu_{n-1}(\x_i,1))| > a_n \sigma_{n-1}(\x_{i},1) \right] \leq & \f{1}{2} e^{-a_n^2 / 2} = \f{\delta}{2 b_n |\tC|}. 
    \end{split}
\end{equation}
Applying the union bound $\forall i$, we get
\begin{equation*}
    \begin{split}
    \mP \bigcup_{i=1}^{|\tC|} |g(f(\x_i,1)) -g(\mu_{n-1}(\x_i,1))| > & a_n \sigma_{n-1}(\x_{i},1) \leq \\
        &\sum_{i=1}^{\tC}  \mP \left[ |g(f(\x_i,1)) - g(\mu_{n-1}(\x_i,1))| > a_n \sigma_{n-1}(\x_{i},1) \right]\\
        \leq & \sum_{i=1}^{\tC} \f{\delta}{2 b_n |\tC|} = \f{\delta}{2 b_n}.
    \end{split}
\end{equation*}
Now we apply the union bound $\forall n$:
\begin{equation}
    \begin{split}
        \mP \left[\forall \tx \in \tC,~ \bigcup_{n=1}^{\infty} |g(f(\tx,1)) - g(\mu_{n-1}(\tx,1))| > a_n \sigma_{n-1}(\tx,1)\right] \leq & \sum_{n=1}^{\infty} \f{\delta}{2 b_n} = \f{\delta}{2}.    \\  
    \end{split}
    \label{e:bound_for_discreteX}
\end{equation}
We would like to extend the proof to a general $\x \in \X$. For some $\x \in \X$, let $\tilde{\x}$ denote a point $\in \tC$ that is nearest (in Euclidean distance) to $\x$. Then, per Assumption~\ref{ass:lipschitz} and by setting $L^2 = 2 \log \left(\f{2}{\delta}\right)$, the following statement holds:
\begin{equation}
    \mP \left[ |g(f(\x,1)) - g(f(\tilde{\x},1))| > \sqrt{2 \log \left(\f{2}{\delta}\right)} \|\x - \tilde{\x}\| \right] \leq \f{\delta}{2}.
    \label{e:lipshictz_bound}
\end{equation}
Combining \eqref{e:bound_for_discreteX} and \eqref{e:lipshictz_bound}, the following statement holds:
\[ \mP \left[\forall \x \in \X, \forall n,~ |g(f(\x,1)) - g(\mu_{n-1}(\tilde{\x},1))| \leq a_n \sigma_{n-1}(\x,1) + \sqrt{2 \log \left(\f{2}{\delta}\right)} \|\x - \tilde{\x}\| \right] > 1 - \delta. \]
From \Cref{lm:gpvariance}, the $\lim_{n \rightarrow \infty} \sup_{\x \in \X} \sigma_{n-1}(\x, 1) = 0$, and thus
\[ \mP \left[\forall \x \in \X, \lim_{n\rightarrow \infty}|g(f(\x,1)) - g(\mu_{n-1}(\tilde{\x},1))| \leq \epsilon \right] > 1 - \delta, \]
where $\epsilon = \sqrt{2 \log \left(\f{2}{\delta}\right)} \|\x - \tilde{\x}\|$.
\end{proof}

\section{Numerical Experiments}
\label{sec:num_expts}

We first demonstrate our proposed method on synthetic test functions of varying dimensions in \Cref{s:expts_synth} and then on two real-world applications: a gas turbine blade stress analysis test case using finite element analysis in \Cref{ss:expts_turbine} and a transonic wing aerodynamics test case using a finite volume method in \Cref{ss:expts_onera}. Our goal is to compare the multifidelity predictions of the failure probability against the predictions with the high-fidelity model only. For the synthetic test cases and the gas turbine blade test case, we estimate the \emph{true} failure probability using a naive MC with a set of points drawn uniformly at random from $\X$; note that we use this estimate as the ground truth. Then we compute the failure probability error as $\f{|\pf - \pfis|}{\pf}$. For the transonic wing test case, the true failure probability is unknown and  is prohibitive to compute;  we therefore report only the comparative predictions between single fidelity and multifidelity.

For all the experiments, we provide a set of seed points $\mcl{D}_n = \{(\x_i, s_i), y_i,~ i=1,\ldots,n\}$, \mg{selected via a randomized Latin hypercube design}, to start the algorithm. 
Note that we start both the single-fidelity and multifidelity runs with the same \emph{budget} of information; therefore, for any given experiment, the $n$ for single-fidelity and multifidelity runs are different in order to maintain the same total computational cost with which we start the experiment. Naturally, whereas for the multifidelity experiments the seed points are distributed uniformly in $\mcl{S}$, for the single-fidelity experiments the seed points are constrained to $\mcl{S} = \{1\}$. All the synthetic experiments are repeated a total of 20 times with randomized starting seed points. The turbine stress analysis test case is repeated only 5 times, keeping in consideration the higher computational cost of this experiment.  There is no repetition in the transonic wing test case because that would be computationally prohibitive.

For all the experiments, the following cost model is assumed:
\[ c(s) / c_0 = \mg{c_2} + \exp\left(-c_1 \times (1 - s) \right),\]
where $c_0$ is a normalizing constant that is required to ensure that the acquisition function optimization is not biased by the actual costs of the models; note that $c_0$ is also the cost of the high-fidelity ($s=1$) model, for which we fix $c_0 = 500$. The $c_1>0$ determines the steepness of the cost model; we set $c_1 = 10$ for all the experiments. Finally, $c_2> 0$ is a constant to ensure that the cost model is bounded away from zero and is roughly the cost of the lowest-fidelity model ($\hat{f}(\cdot, 0)$); we fix $c_2 = 0.1$. We note that our cost model---by virtue of setting $c_1=10$---is designed to emphasize the disparity between the computational costs of lower- and higher-fidelity model queries; see \Cref{fig:cost_model}. In \Cref{fig:ishigami_cost_models} we show the effect of the cost model with varying $c_1$ and discuss its effect on the predictions.

We evaluate our predictions qualitatively for the 2D cases, in other words, by comparing the failure level sets estimated via the final GP mean $\mu_q(\x, 1)$ against the true level sets in contour plots. For all the synthetic cases (including the 2D cases), we then compute the error in failure probability, as mentioned previously. The true failure probability is estimated via a naive MC with $N=10^6$ points sampled uniformly from $\X$. For the turbine case, we first evaluate the high-fidelity model at a total of $150$ parameter snapshots and use the observed QoIs to construct a radial basis function (RBF) interpolant. We then perform a naive MC on the RBF interpolant to obtain the true $p_{\mcl{F}}$. We  begin by discussing the synthetic experiments in the next section.
\begin{figure}
    \centering
    \includegraphics[width=.7\textwidth]{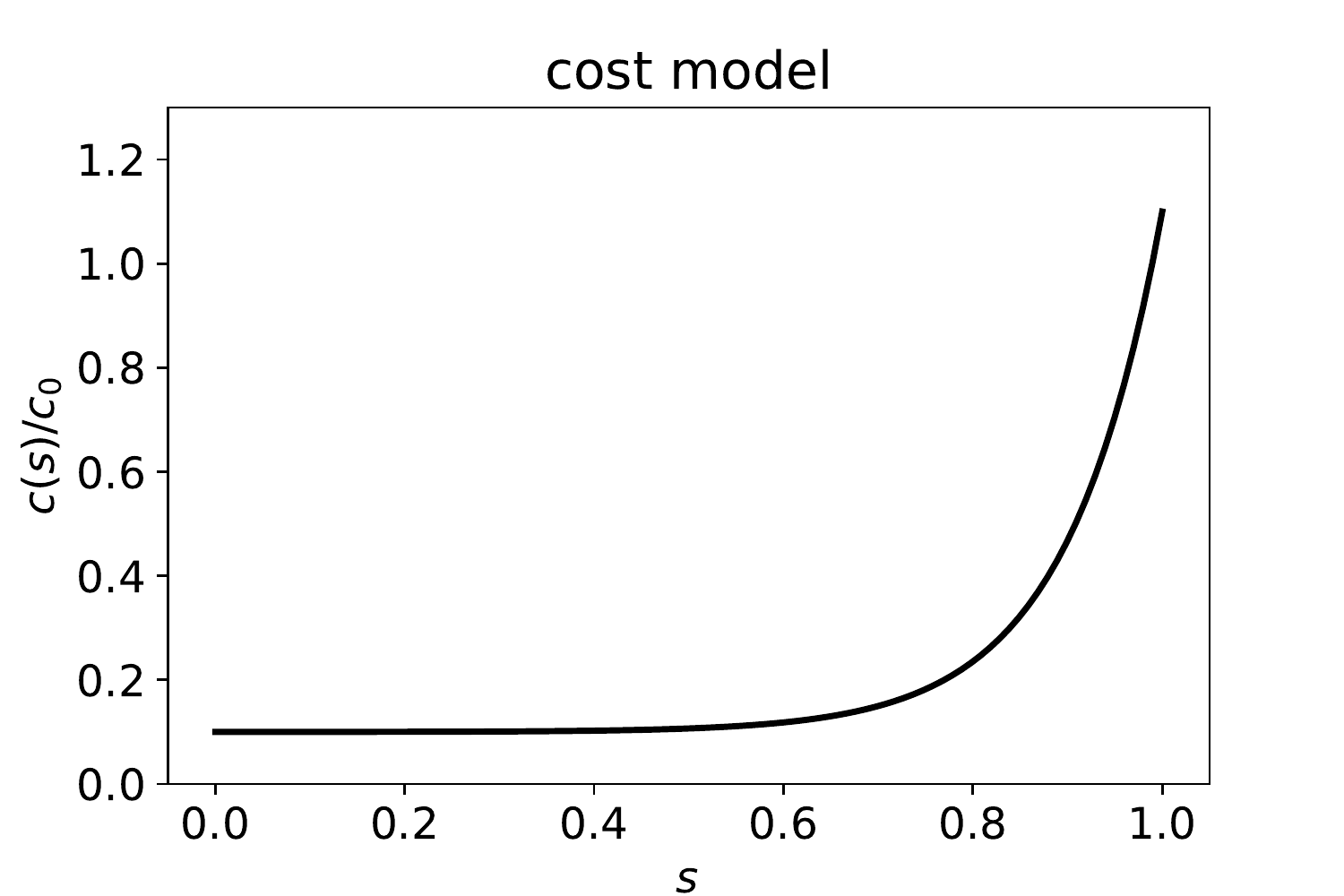}
    \caption{Multifidelity cost model $c(s)$ with $c_1 = 10$, which emphasizes the cost disparity across fidelities as one approaches $s=1$.}
    \label{fig:cost_model}
\end{figure}
\subsection{Synthetic experiments}
\label{s:expts_synth}
We introduce several new multifideity test functions that can be used to benchmark multifidelity methods beyond the scope of this work. Specifically, we introduce the 2D multifidelity multimodal function~\cite{bichon2008efficient}, the 2D multifidelity four-branches function~\cite{schobi2017rare}, and the 3D multifidelity Ishigami function~\cite{ishigami1990importance}. Additionally, we  use the 6D multifidelity Hartmann function provided in~\cite{simulationlib}. The specifics of these test functions are discussed as follows.

\subsubsection{2D multimodal function}
We extend the test function in \cite{bichon2008efficient} to multifidelity setting as shown in \eqref{e:mf_multimodal},

\begin{equation}
   f(\x, s) = \f{(x_1^2 + 4)(x_2 - 1)}{20} - s \times \sin \left( \f{5x_1}{2} \right) - 2,
   \label{e:mf_multimodal}
\end{equation}
where $\x=[x_1, x_2]$, the domain $\X = [-4,7] \times [-3,8]$, and the second term in the right-hand side introduces a fidelity dependence. This simple extension results in a correlated fidelity dependence, which is continuously differentiable in $\mcl{X}\times \mcl{S}$; see \Cref{f:mftestfunc1} for snapshots.

\begin{figure}
    \centering
    \includegraphics[width=1\textwidth]{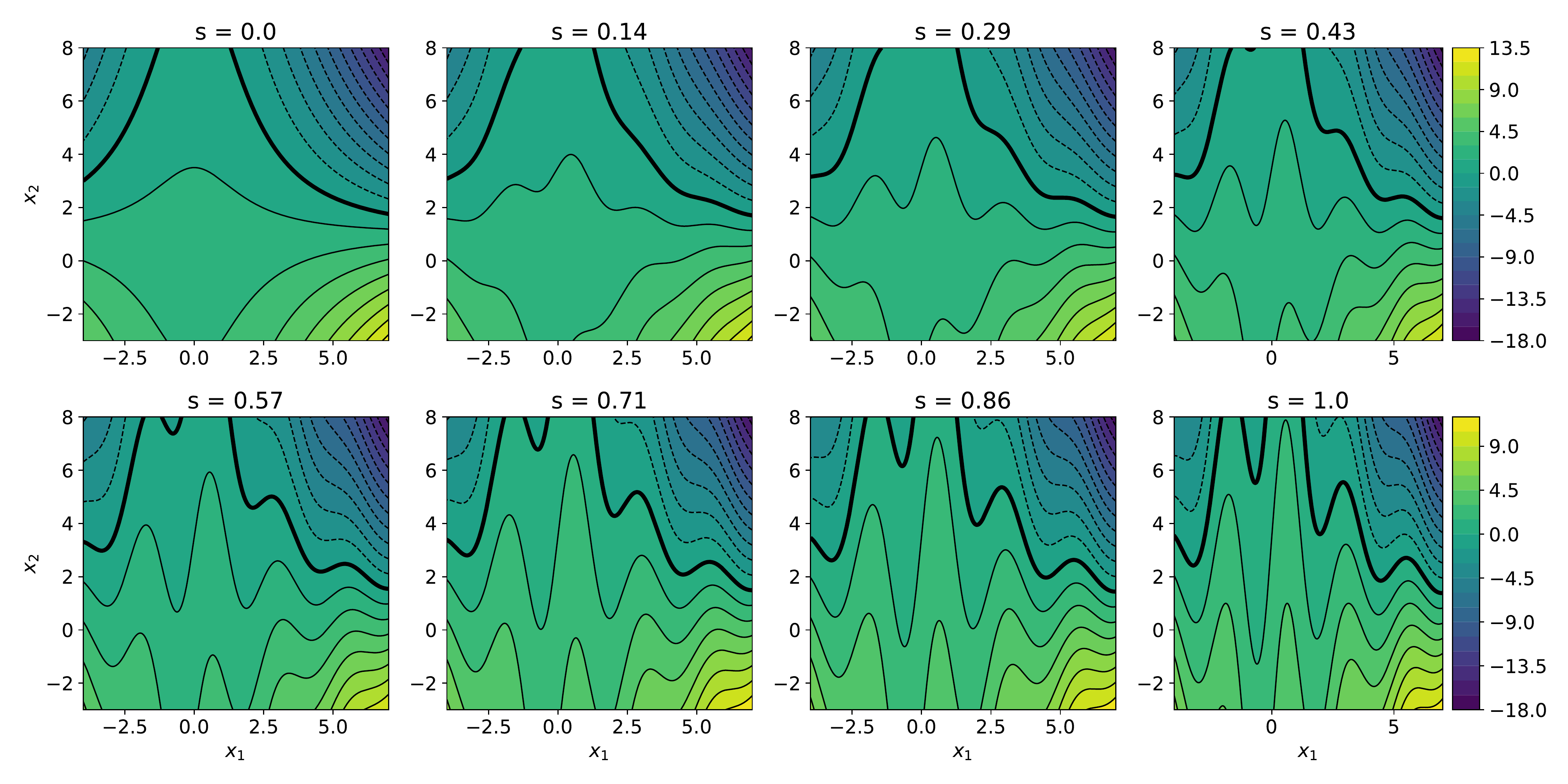}
    \caption{Multifidelity multimodal function; the thick black line represents the limit state level set $g(f(\x, \cdot)) = 0$.}
    \label{f:mftestfunc1}
\end{figure}

\subsubsection{2D multifidelity four-branches function}
A system with four distinct component limit-states \cite{schobi2017rare} is a commonly used test problem in reliability analysis to demonstrate the efficacy of the algorithms in the presence of multiple failure regions. Here, in contrast to \eqref{e:mf_multimodal}, we introduce fidelity by making a coordinate transformation as $x_i = x'_i - a_0 \times s$, where $a_0 \in \mbb{R}$. The response function can be written as
\begin{align}
    f({\x}, s) = \textrm{min} \begin{cases} 3 + 0.1(x'_1 -x'_2)^2 - \frac{x'_1 + x'_2}{\sqrt{2}}\,, \\
    3 + 0.1(x'_1 - x'_2)^2 + \frac{x'_1 + x'_2}{\sqrt{2}}\,, \\
    x'_1 - x'_2 + \frac{7}{\sqrt{2}}\,, \\
    x'_2 - x'_1 + \frac{7}{\sqrt{2}}.
    \end{cases}\\
    \T{and}~x'_i = x_i - 5.0 \times s,~i=1,2,3,
\end{align}
where $\x = [x_1, x_2]$ and $\X = [-8,8]^2$. The multifidelity extension we provide offers a translation of the level sets across fidelities, where we set $a_0 = -5.0$, and is shown in \Cref{f:mf_fourbranches}.
\begin{figure}[htb!]
    \centering
    \includegraphics[width=1\textwidth]{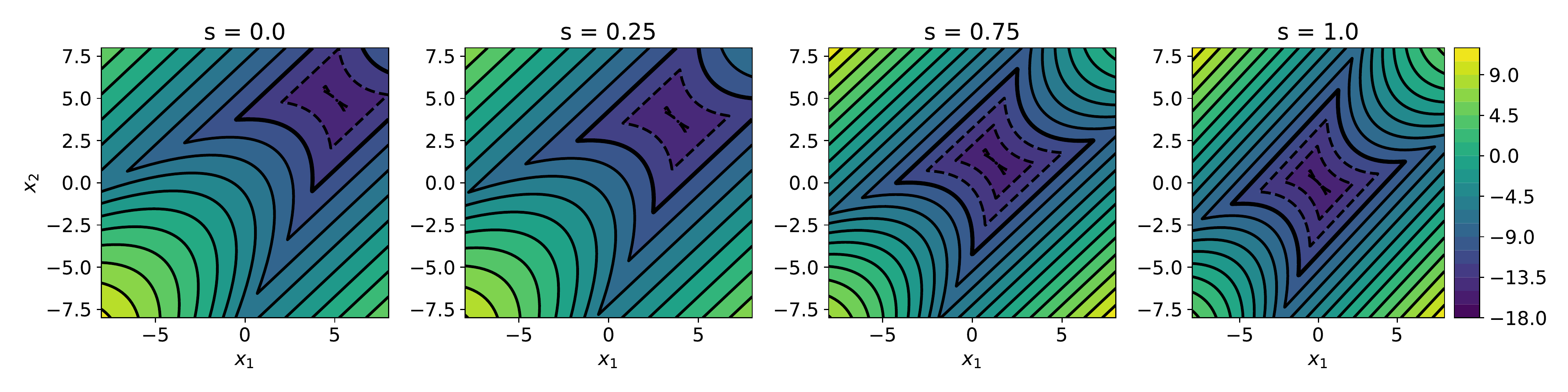}
    \caption{Multifidelity four-branches test function. The thick black line represents the failure boundary $g(f(\x)) = 0$.}
    \label{f:mf_fourbranches}
\end{figure}

\subsubsection{3D Multifidelity Ishigami function}
We modify the 3D Ishigami test function \cite{ishigami1990importance}, similar to the four-branches function, by introducing a shift in the variables using the fidelity parameter $s$, as shown in \eqref{e:mf_ishigami}. Note that the input $\x$ is normalized to be in $[0,1]^3$ before introducing the translation: 
\begin{equation}
    f(\x, s) = \sin(x_1 - s) + 7.0 \sin^2(x_2 - s) + 0.1 x_3 ^4  \sin(x_1 - s),
    \label{e:mf_ishigami}
\end{equation}
where $\x = [x_1, x_2, x_3]$ and $\X=[-\pi, \pi]^3$. The resulting function is shown for two slices $x_3=-\pi$ and $x_3=0$ in \Cref{f:mf_ishigami}.
\begin{figure}
    \centering
    \includegraphics[width=1\textwidth]{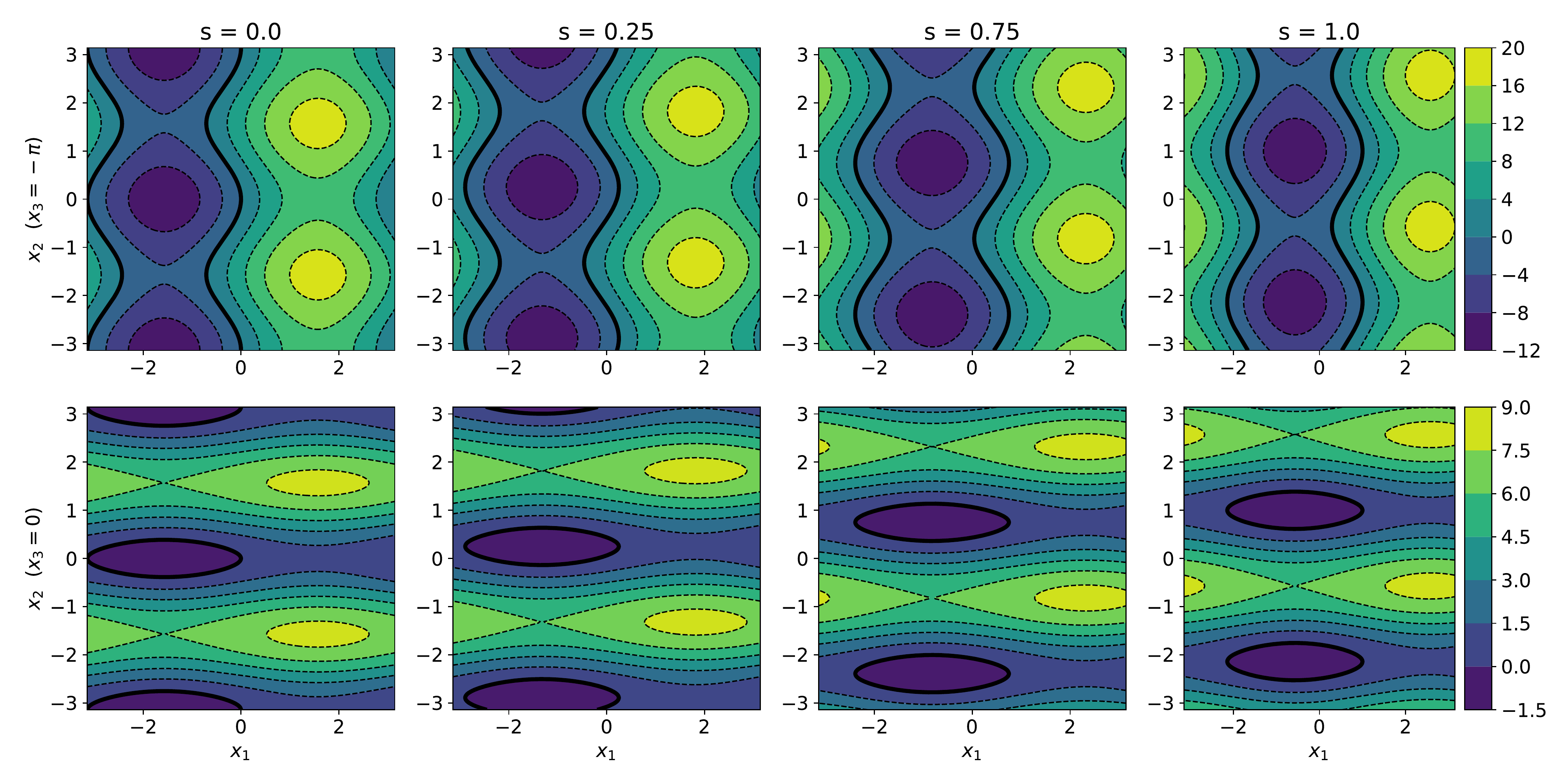}
    \caption{Multifidelity extension of the Ishigami~\cite{ishigami1990importance} test function.}
    \label{f:mf_ishigami}
\end{figure}

\subsubsection{6D Hartmann function}
\label{ss:6d_hartmann}
For a higher-dimensional test function, we use the          ``augmented" Hartmann 6D function introduced in \cite{balandat2019botorch} and restated below:
\begin{equation}
    f(\x, s) = -\left(\beta_1 - 0.1 * (1-s)\right) * \exp \left(- \sum_{j=1}^6 A_{1j} (x_j - P_{1j})^2 \right) -
            \sum_{i=2}^4 \beta_i \exp\left( - \sum_{j=1}^6 A_{ij} (x_j - P_{ij})^ 2\right),
    \label{e:mf_hartmann}
\end{equation}
where $\x = [x_1, \ldots, x_6]$ and $\X = [0,1]^6$ and
\[ 
A = \begin{bmatrix}
10 & 3 & 17 & 3.5 & 1.7 & 8 \\
0.05 & 10 & 17 & 0.1 & 8 & 14 \\
3 & 3.5 & 1.7 & 10 & 17 & 8 \\
17 & 8 & 0.05 & 10 & 0.1 & 14
\end{bmatrix}
\]
and
\[ 
P = \begin{bmatrix}
1312 & 1696 & 5569 & 124  & 8283 & 5886 \\
2329 & 4135 & 8307 & 3736 & 1004 & 9991 \\
2348 & 1451 & 3522 & 2883 & 3047 & 6650 \\
4047 & 8828 & 8732 & 5743 & 1091 & 381
\end{bmatrix}.
\]
For the 2D test functions previously discussed, the failure threshold was set as $a = 0.$. For the Ishigami and Hartmann functions, we show the histogram of the output at $s=1$ along with the failure thresholds (vertical lines), which are set to $a=-9.0$ and $a=-2.0$, respectively, in \Cref{fig:hist}; these thresholds were chosen particularly to make the failure probabilities on the order of $1\%$. 

\begin{figure}
    \centering
        \begin{subfigure}{.5\textwidth}
           \includegraphics[width=1\linewidth]{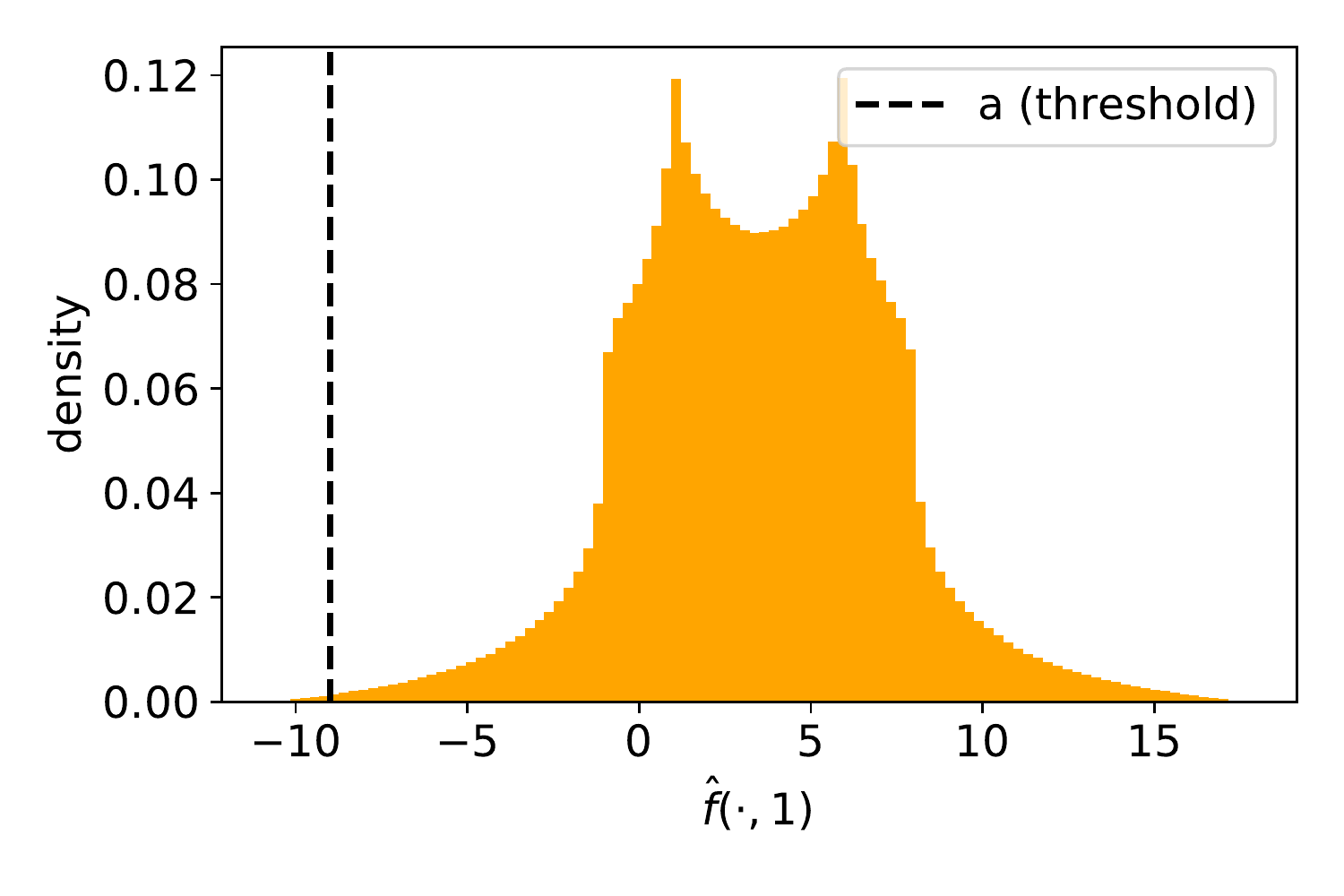}
           \caption{Ishigami}
        \end{subfigure}%
        \begin{subfigure}{.5\textwidth}
           \includegraphics[width=1\linewidth]{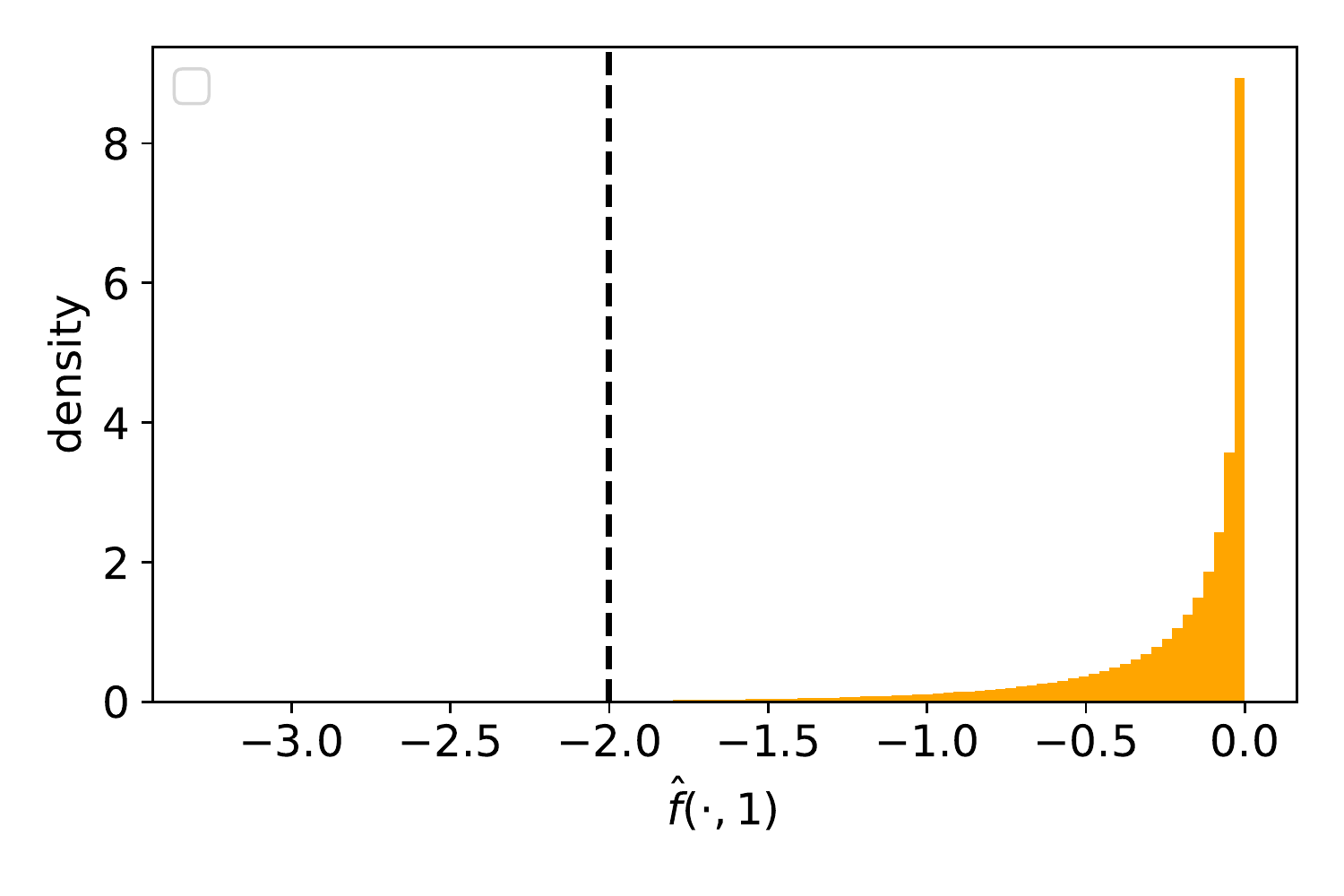}
           \caption{Hartmann}
        \end{subfigure}
    \caption{Histograms of the high-fidelity output $\hat{f}(\cdot, 1)$ of the multifidelity 3D Ishigami (left) and 6D Hartmann test functions. Vertical lines indicate failure thresholds.}
    \label{fig:hist}
\end{figure}

The results of the 2D test functions are shown in terms of the predicted contours in \Cref{f:multimodal_contours} and \Cref{f:fourbranches_contours}; in both figures, whereas the true level set is plotted in the red dashed-dotted line, the background (filled contours) shows the final GP posterior mean $\mu_q(\x, 1)$ with the predicted level set as the thick black line. We hope that readers are able to visually compare the predictions against the true failure boundary. Note that in both figures we show only the predictions of one of the twenty random repetitions. Whereas for the multimodal function the predictions from the single- and multifidelity approaches are  similar, the multifidelity achieves this prediction with a computational cost that is more than 45\% cheaper. For the four-branches test function, the prediction of the multifidelity approach is visually much better than that of the single-fidelity approach, albeit using only half the computational cost.

We use the contours to provide a visual interpretation of the results, since they are feasible for the 2D cases.  For a quantitative evaluation, however, we also show the error in computing the failure probability. 
The absolute errors in computing the failure volume are shown in \Cref{fig:synthetic} for all the synthetic test functions. We run both the single- and multifidelity experiments for $100$ iterations after providing them with seed samples, as previously mentioned. 
In all four cases, we observe that the multifidelity approach is able to predict the failure volume with greater accuracy compared with the single-fidelity approach, for a certain computational budget. \mg{In order to establish the utility of the adaptive nature of our proposed algorithm, we also compare our predictions with a ``non-adaptive'' Latin hypercube design. In \Cref{fig:synthetic}, we show the predictions from a non-adaptive design (horizontal dashed line) with a total budget of the $n$ seed points plus the 100 used in each experiment. This way, we keep the budget consistent with each experiment and repeat this 20 times to randomize the selection of the Latin hypercube points; \Cref{fig:synthetic} shows the average of the repetitions. We observe that the non-adaptive design leads to a worse prediction of the failure probability compared to \texttt{CAMERA} for the same budget. This is intuitive because a (space-filling) Latin hypercube design may provide good coverage of $\X$ but can fail to resolve the failure boundary which is critical to fitting a good biasing density and hence to a good estimate of the failure probability. On the other hand, \texttt{CAMERA} attempts to adaptively select points that would best predict the failure boundary at the highest fidelity level.}

\mg{In addition to the FPE, we are interested in observing how the \texttt{CAMERA} algorithm balances the trade-off between accuracy and cost across the fidelity space. In this regard, we show the distribution of the number of queries to models at each fidelity level, across all the 100 adaptive selections, in \Cref{fig:synthetic_queries}, where all the 20 repetitions are overlaid on top of each other. We observe that when there is high correlation between the low and high-fidelity models, then the algorithm exploits this to obtain most samples from the lowest fidelity model, e.g., Multimodal, Four branches, and Ishigami functions. For these test functions, only a few queries at the intermediate fidelities were necessary. Most of the evaluations are at $s=0$ (lowest fidelity) and $s=1$ (highest fidelity). In the case of the Ishigami test function, which also has the rarest failure probability, the algorithm has most queries at $s=1$, with queries at intermediate fidelity levels almost uniformly distributed. Due to the lack of sufficient correlation between the high and low fidelity levels, the algorithm is unable to exploit the cheaper information as in the other experiments. However, notice that across all the experiments, the highest fidelity model is queried only about 25\% of the time. When sufficient correlations exist, the lowest fidelity model is queried 50-75\% of the time. This demonstrates the efficiency of the method in exploiting correlations between models at multiple fidelities. Finally, notice that the distribution of the number of queries is roughly consistent across all the repetitions. This demonstrates the robustness of our algorithm amidst the randomness of the seed point selection.}

\begin{figure}
    \centering
    \begin{subfigure}{.5\textwidth}
        \includegraphics[width=1\linewidth]{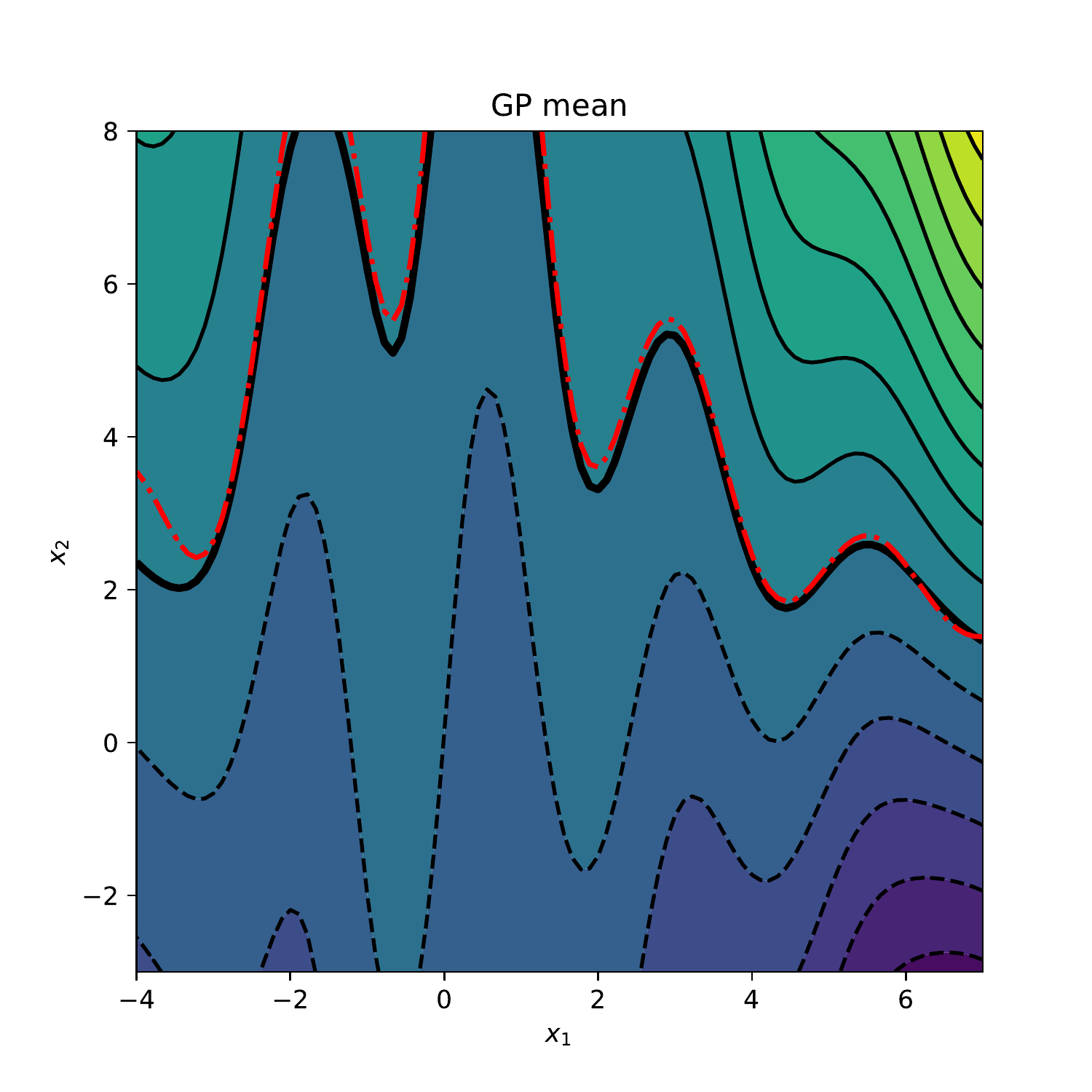}    
        \caption{Single-fidelity prediction}
    \end{subfigure}%
    \begin{subfigure}{.5\textwidth}
        \includegraphics[width=1\linewidth]{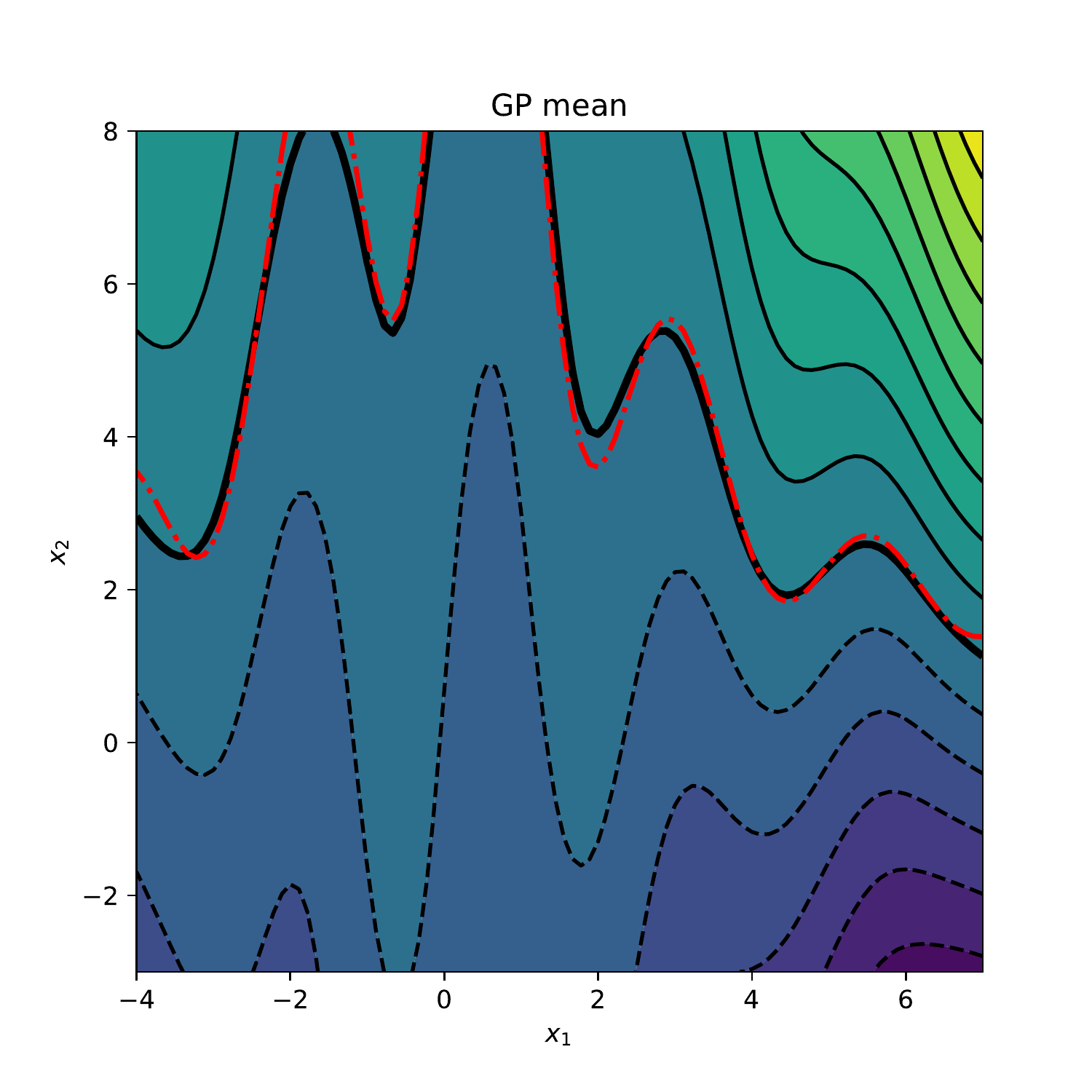}
        \caption{Multifidelity prediction}
    \end{subfigure}
    
    \caption{Results of the 2D multimodal test function; red dash-dotted line represents the true level set. Cost for multifidelity prediction is $30,010$ and for single-fidelity prediction $55,000$. Both figures share the same colormap.}
    \label{f:multimodal_contours}
\end{figure}

\begin{figure}
    \centering
    \begin{subfigure}{.5\textwidth}
        \includegraphics[width=1\linewidth]{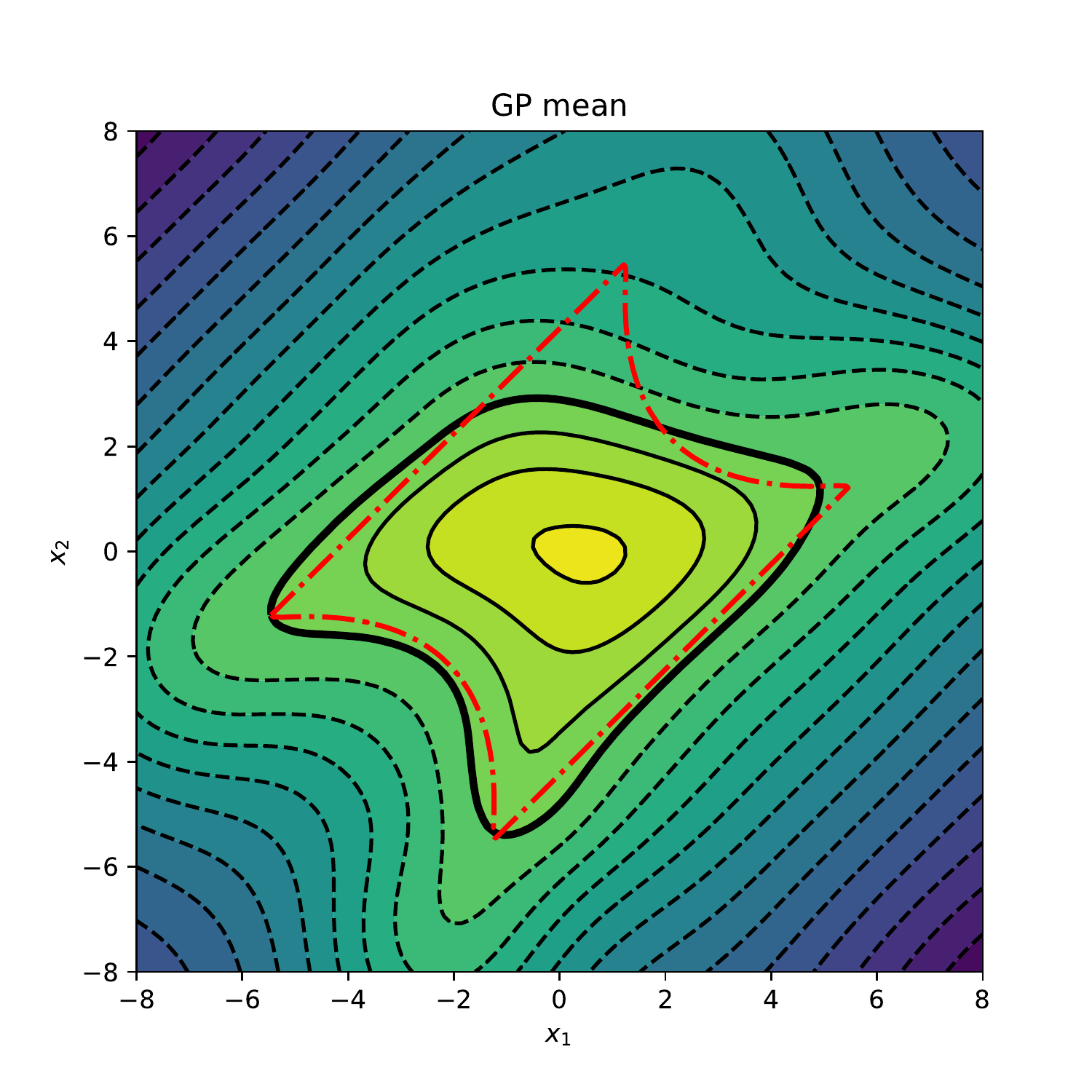}    
        \caption{Single-fidelity prediction}
    \end{subfigure}%
    \begin{subfigure}{.5\textwidth}
        \includegraphics[width=1\linewidth]{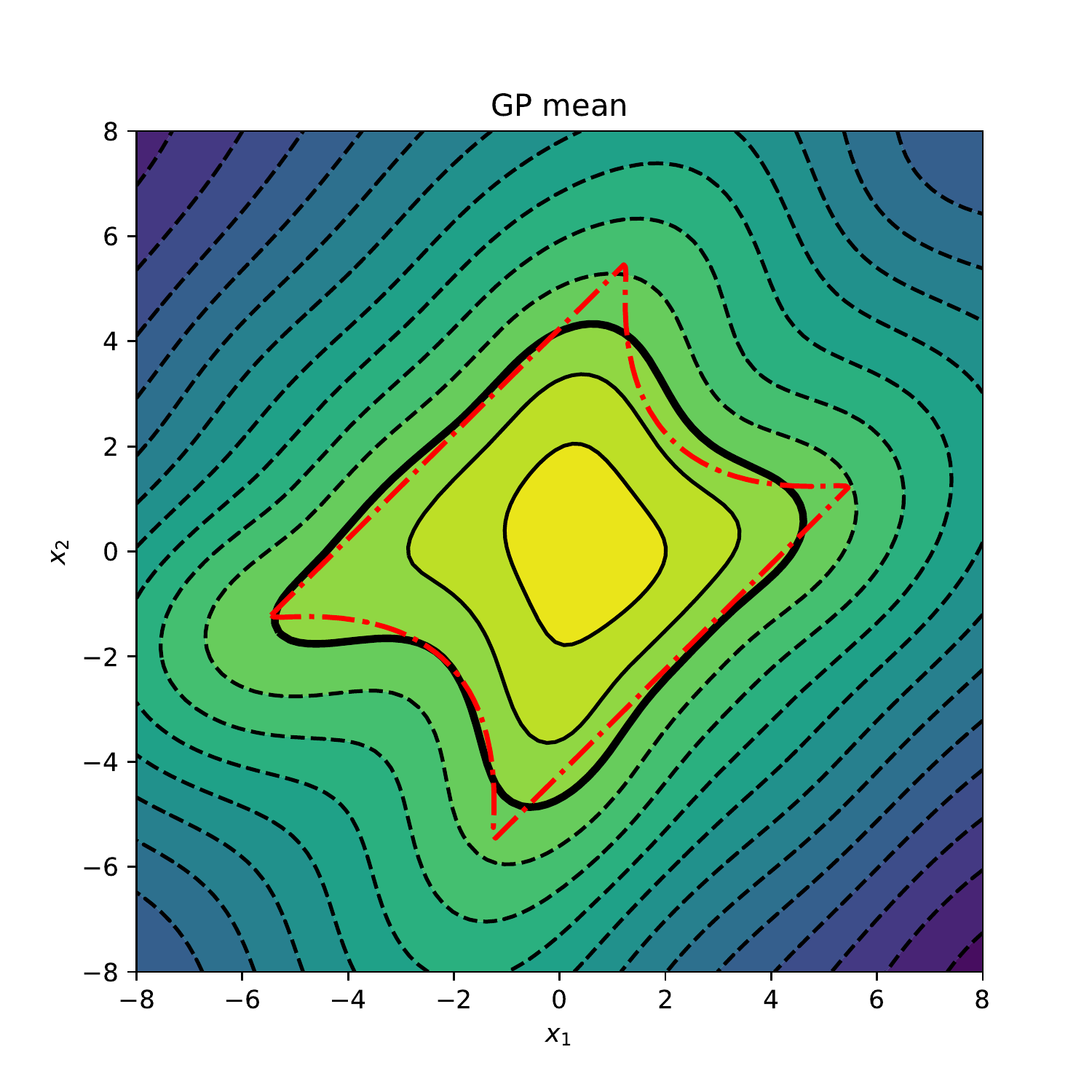}
        \caption{Multifidelity prediction}
    \end{subfigure}
    
    \caption{Results of the 2D four-branches test function; red dash-dotted line represents the true level set. Cost for multifidelity prediction is $30,005$ and for single-fidelity prediction $55,000$. Both figures share the same colormap.}
    \label{f:fourbranches_contours}
\end{figure}

\begin{figure}
    \centering
    \begin{subfigure}{.5\textwidth}
        \includegraphics[width=1\linewidth]{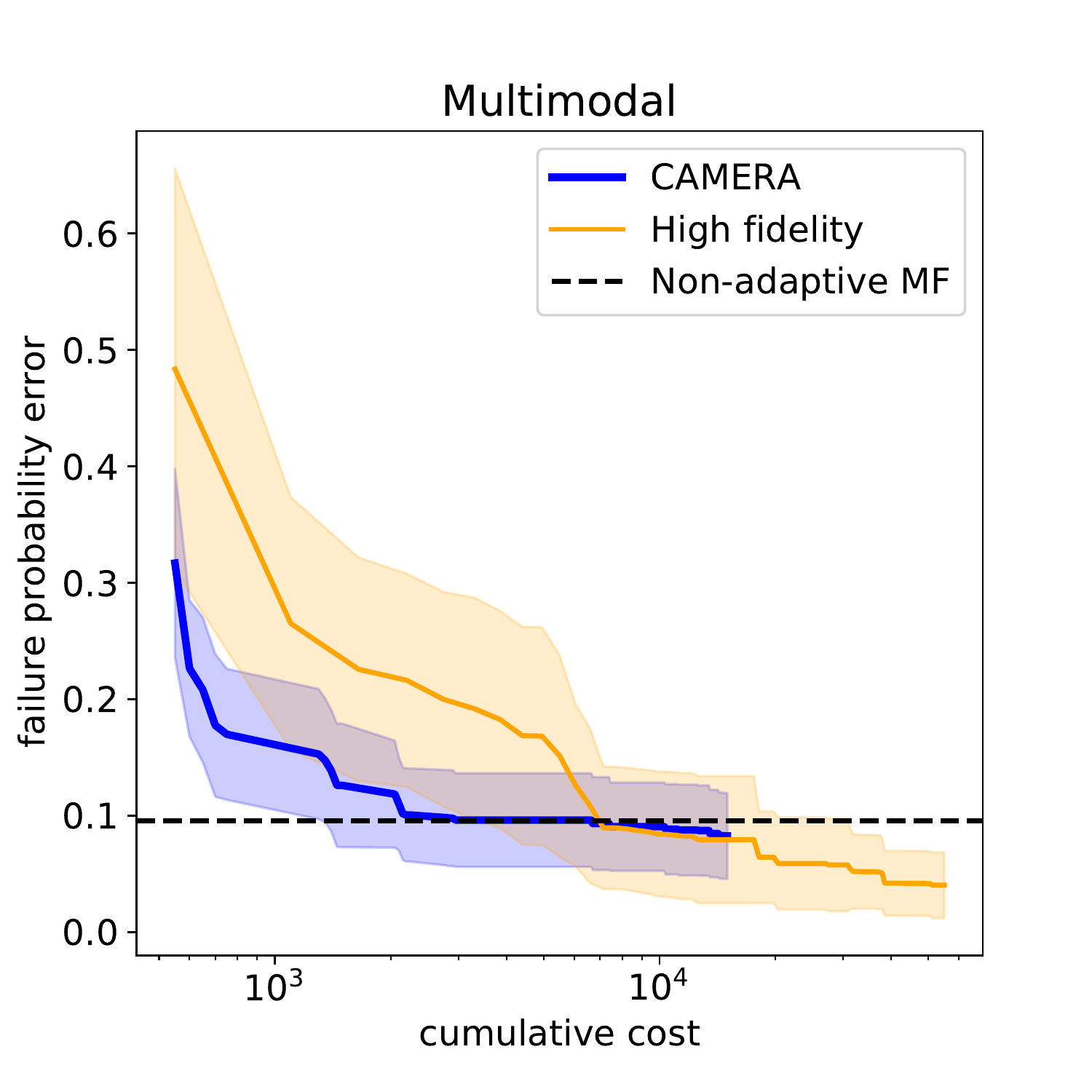}    
        \caption{$a=0.0$, $p_\mcl{F} = 0.30215$}
    \end{subfigure}%
    \begin{subfigure}{.5\textwidth}
        \includegraphics[width=1\linewidth]{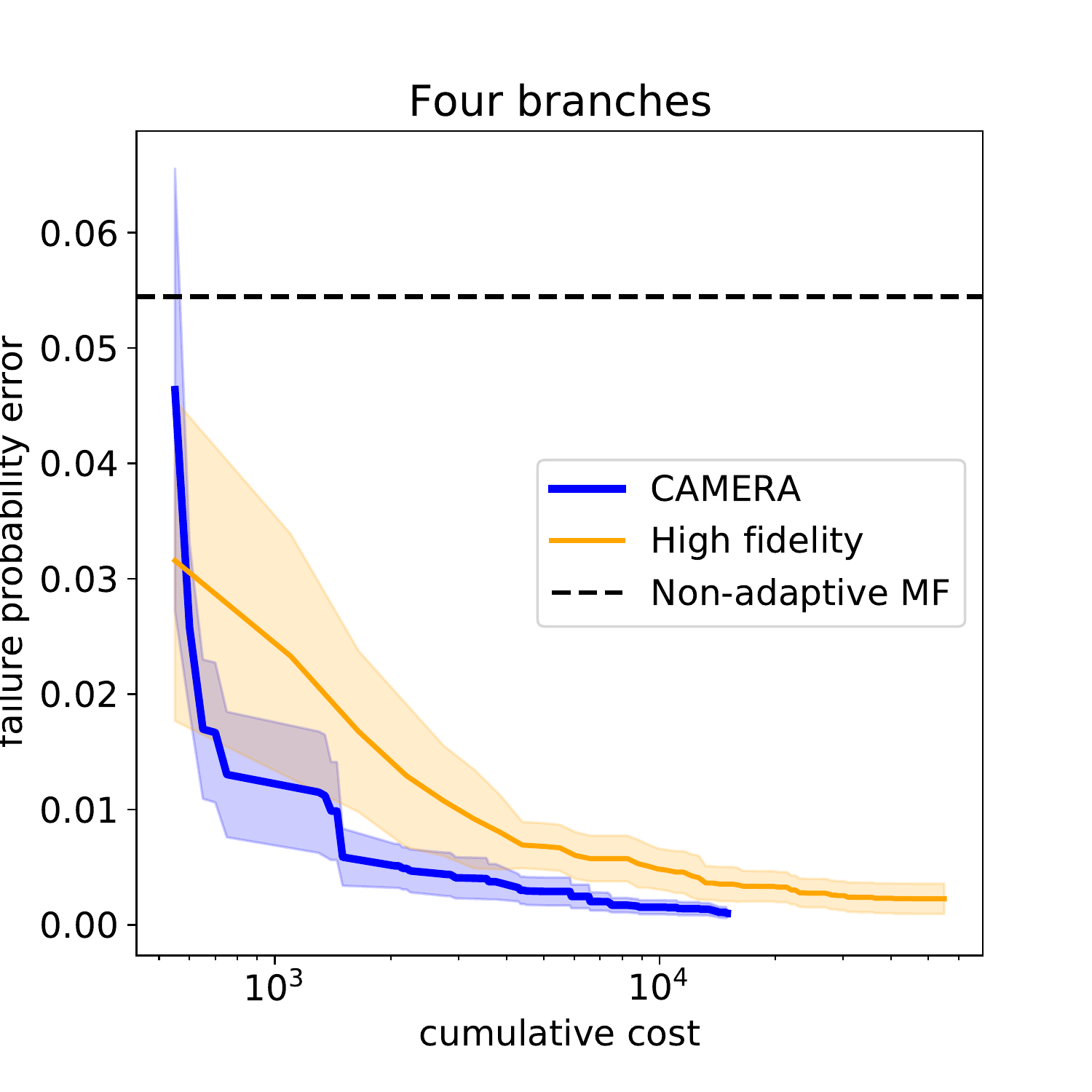}    
        \caption{$a=0.0$, $p_\mcl{F} = 0.1689$}
    \end{subfigure}\\
    \begin{subfigure}{.5\textwidth}
        \includegraphics[width=1\linewidth]{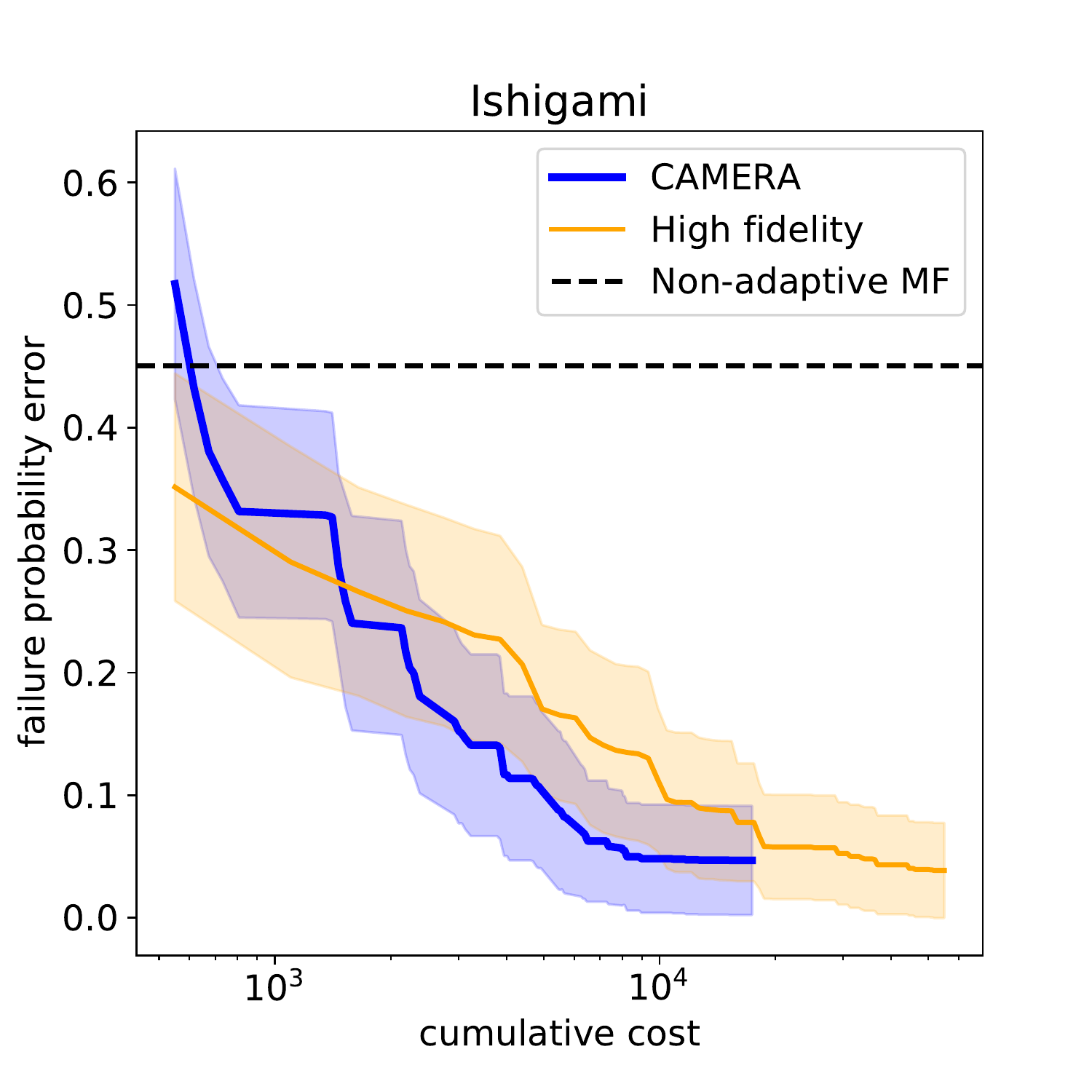}    
        \caption{$a=-9.0$, $p_\mcl{F} = 0.0011$}
    \end{subfigure}%
    \begin{subfigure}{.5\textwidth}
        \includegraphics[width=1\linewidth]{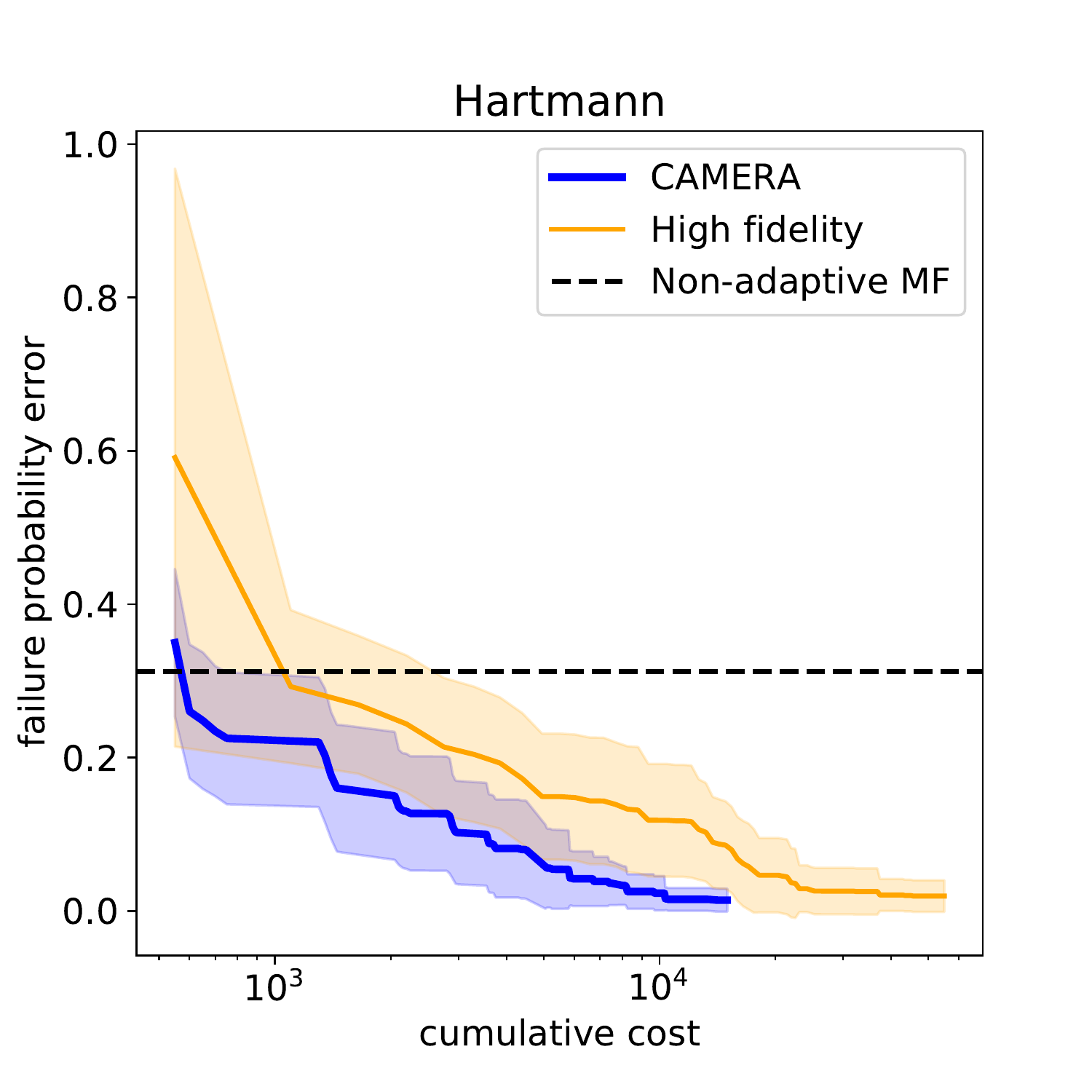}    
        \caption{$a=-2.0$, $p_\mcl{F} = 0.00737$}
    \end{subfigure}    
    \caption{Error in computing the failure probability for the synthetic test functions.}
    \label{fig:synthetic}
\end{figure}

\begin{figure}
    \centering
    \begin{subfigure}{.5\textwidth}
        \includegraphics[width=1\linewidth]{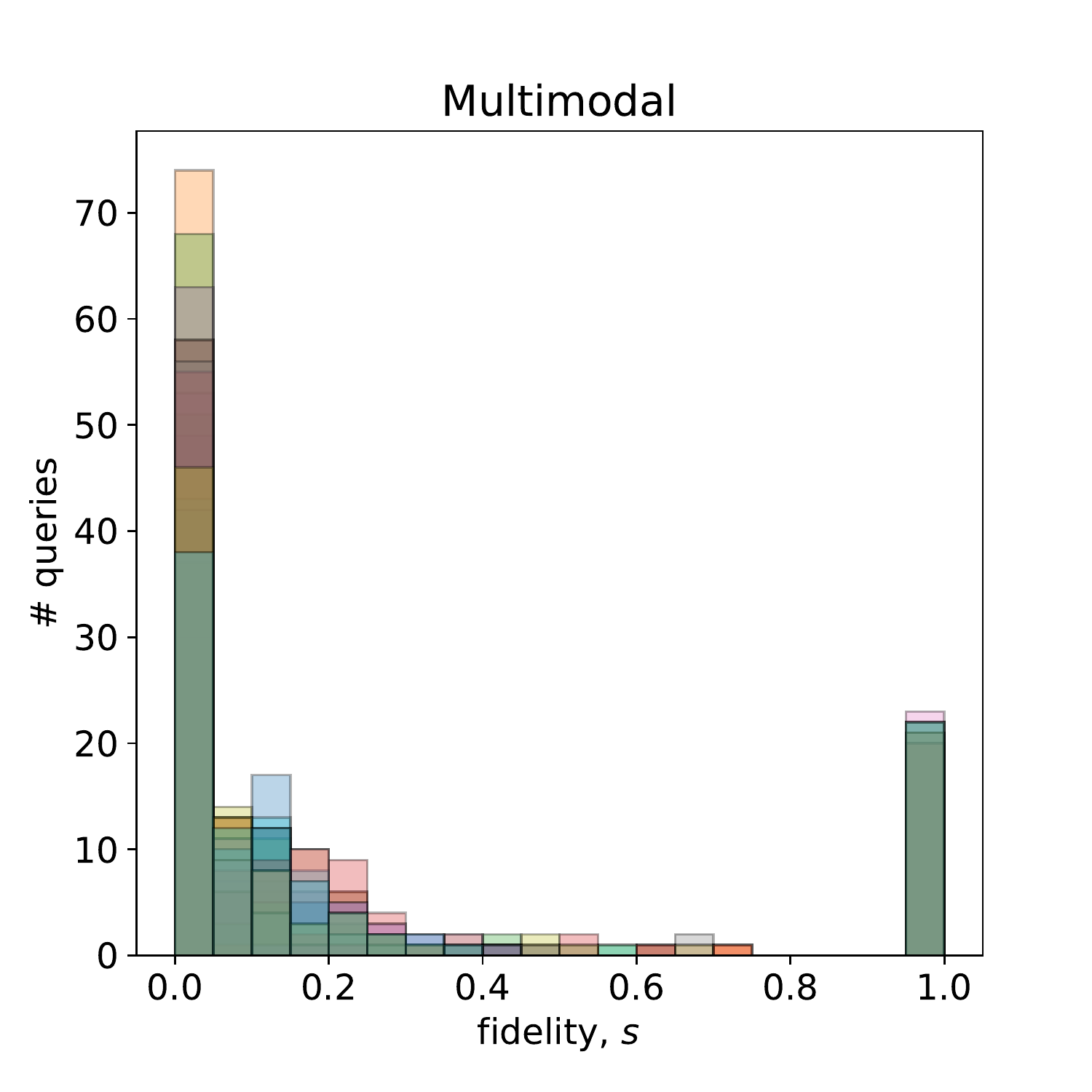}    
        \caption{$a=0.0$, $p_\mcl{F} = 0.30215$}
    \end{subfigure}%
    \begin{subfigure}{.5\textwidth}
        \includegraphics[width=1\linewidth]{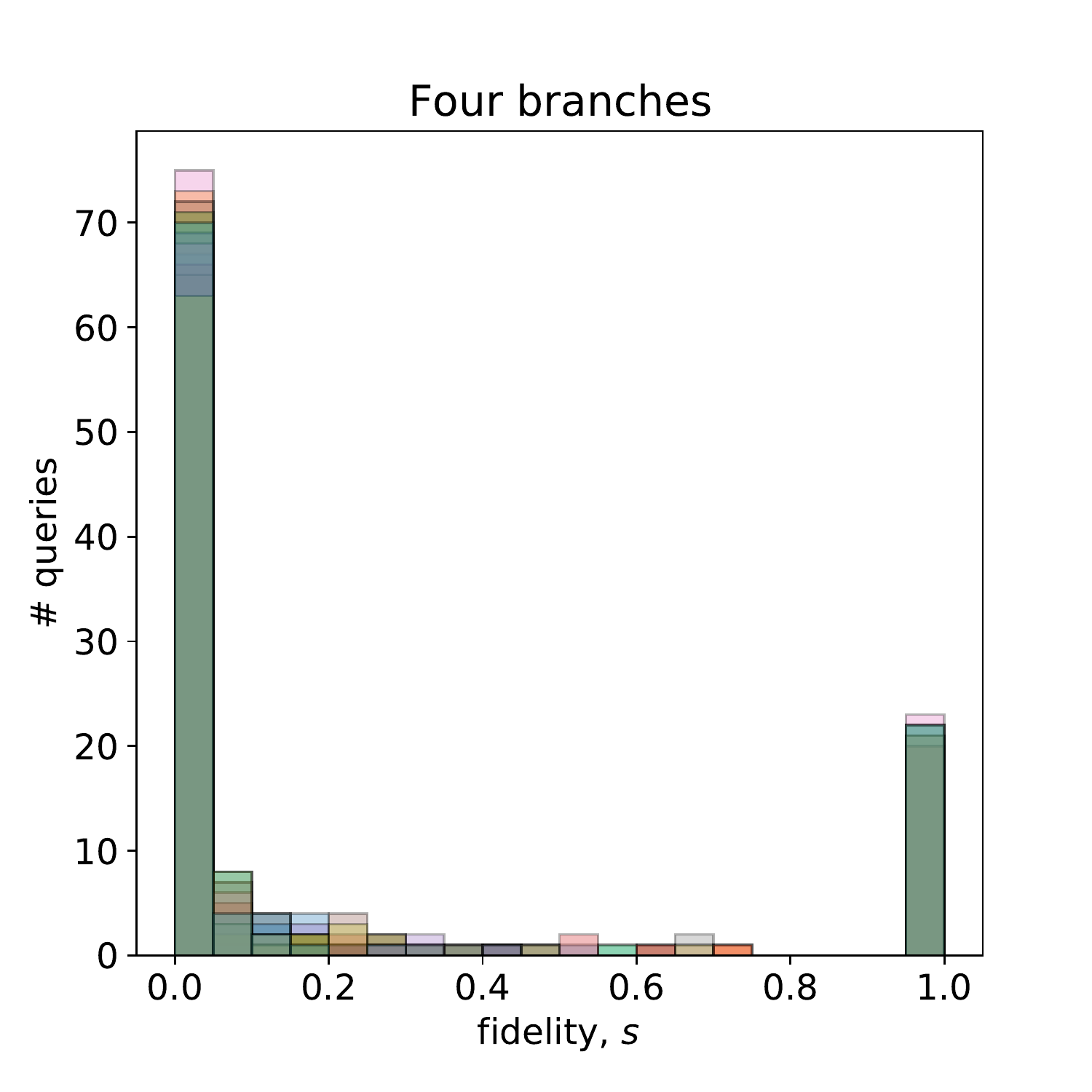}    
        \caption{$a=0.0$, $p_\mcl{F} = 0.1689$}
    \end{subfigure}\\
    \begin{subfigure}{.5\textwidth}
        \includegraphics[width=1\linewidth]{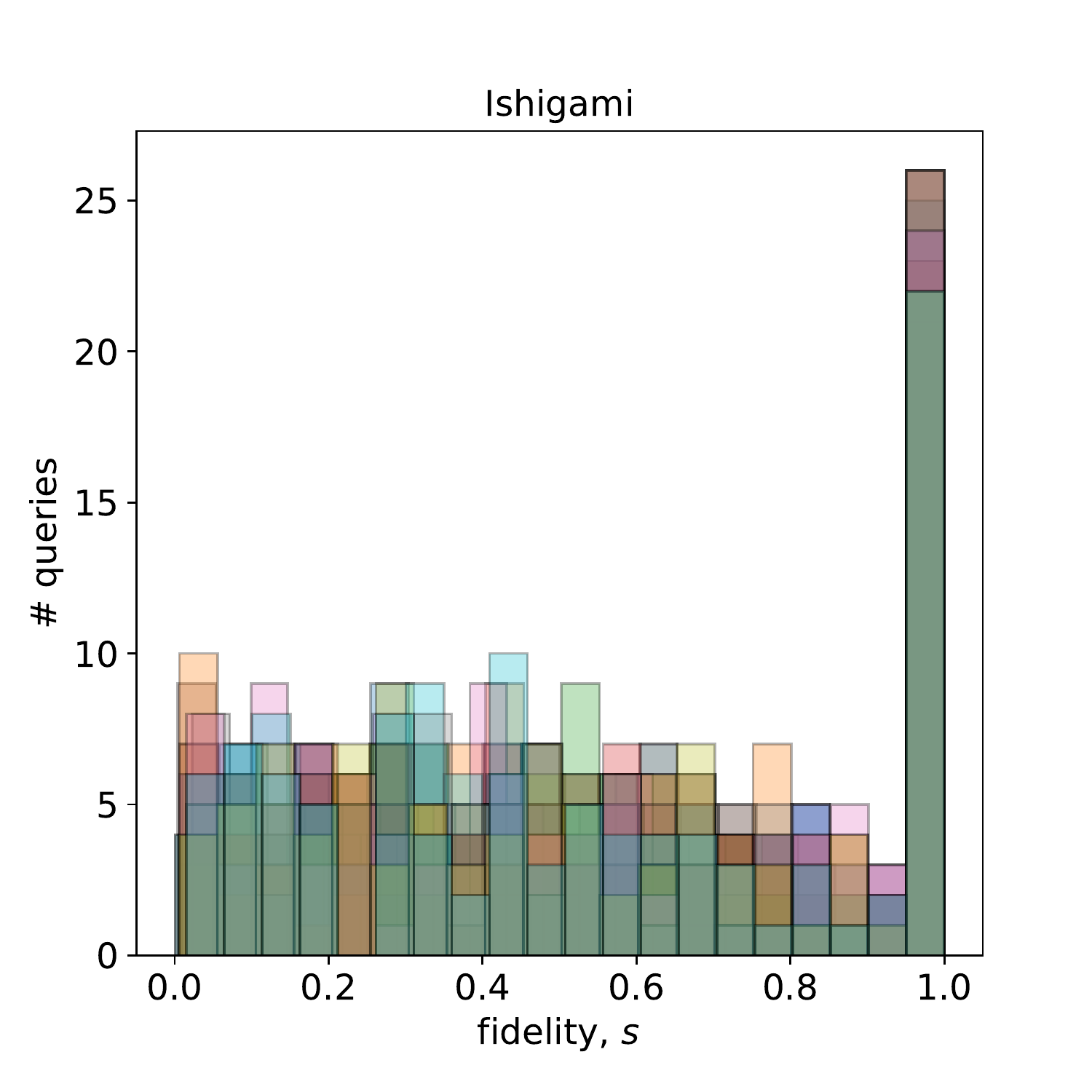}    
        \caption{$a=-9.0$, $p_\mcl{F} = 0.0011$}
    \end{subfigure}%
    \begin{subfigure}{.5\textwidth}
        \includegraphics[width=1\linewidth]{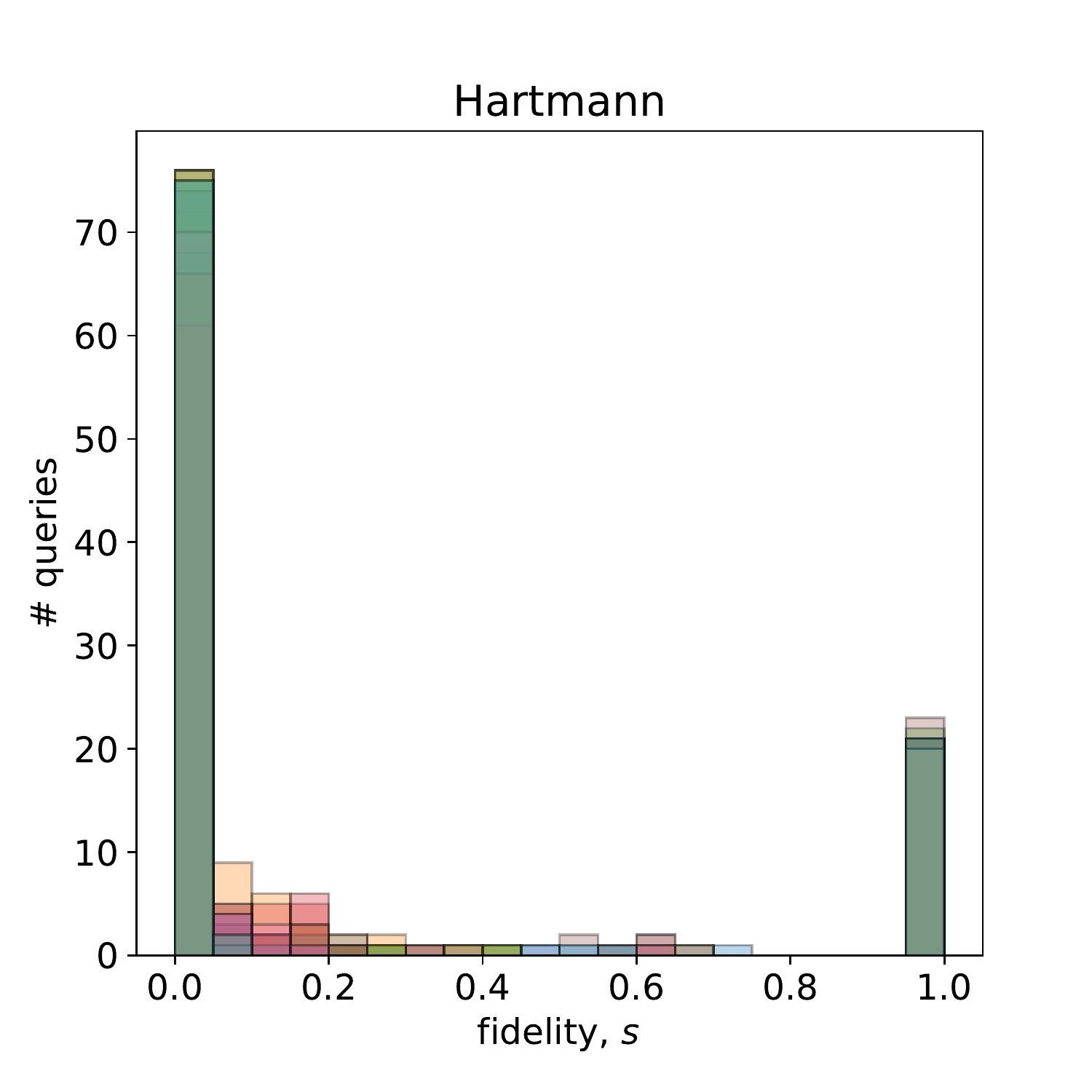}    
        \caption{$a=-2.0$, $p_\mcl{F} = 0.00737$}
    \end{subfigure}    
    \caption{\mg{Number of model queries versus fidelity for the synthetic test functions. Plots show the 20 repetitions overlaid on top of each other.}}
    \label{fig:synthetic_queries}
\end{figure}

\subsection{Turbine reliability analysis with discrete fidelity space}
\label{ss:expts_turbine}

We now demonstrate our method on the structural analysis of a gas turbine blade at steady-state operating conditions. Gas turbines have a radial arrangement of blades downstream of the combustor. Hence, these blades are subject to very high pressures and temperatures that can cause concentrated zones of high stress that can in turn lead to structural deformation and potential catastrophic damage of the engine parts. Therefore, we are interested in the probability that the maximum stress acting on the thermal blade  exceeds a certain threshold. \mg{The variables are defined as follows.}

\begin{equation*}
    \begin{split}
        x_1:& \T{ Pressure-side loading} \\
        x_2:& \T{ Suction-side loading} \\
        s :& \T{ mesh resolution},
    \end{split}
\end{equation*}
In this regard we vary the pressure-side and suction-side pressures as boundary conditions, as shown above, while ignoring thermal effects, to observe the maximum von Mises stress acting on the blade; that is, $f(\x) = \max_{\mbf{z}} u(\mbf{z})$, where $u(\mbf{z}),~\mbf{z} \in \partial \mbf{z}$ is the Von Mises stress distribution on the blade surface $\partial \mbf{z}$. Note that we negate the output,  $f(\x) \times -1$, to comply with our definition of failure. The material properties are fixed to the following constants: Young's modulus $227E+9$ Pa, Poisson's ratio $0.27$, and coefficient of thermal expansion  $12.7E-6$. The governing equations are solved using a finite element method  using Matlab's PDE toolbox. The computational mesh density is determined via a minimum element size which also serves as the fidelity parameter. Specifically, we use the minimum mesh element sizes $\{0.05, 0.025, 0.01\}$, which correspond to fidelities $\mcl{S} = \{0.0, 0.5, 1.0\}$, respectively; note that this case corresponds to a discrete fidelity space. The mesh resolutions for the specified three fidelity levels are shown in \Cref{fig:turbine_fidelities}. As previously mentioned, the maximum von Mises stress is the quantity of interest in this experiment; see \Cref{fig:turbine_solution} for a couple of snapshots of the von Mises stress distribution \mg{and blade deformation} corresponding to two different pressure loadings. 

The suction-side $x_1$ and pressure-side $x_2$ pressures are varied in the following ranges: $x_1 \in [201160, 698390],~ x_2 \in [301460, 598020]$. A total of $n=40$ solution snapshots are generated and supplied as seed points for the multifidelity run; the equivalent $n$ for the single-fidelity run is determined to match the multifidelity seed budget. Both experiments are then run for 100 iterations and are repeated 5 times with randomized seed points. 

The contour plots of maximum von Mises stress are shown in \Cref{fig:turbine_mf_contours} for each of the three fidelities considered. The failure boundary is set as the level set corresponding to $95 \%$ of the maximum von Mises stress across a total of 150 high-fidelity simulations that were computed offline. This is shown by the  thick black line in \Cref{fig:turbine_mf_contours} (far right) and corresponds to a probability of failure of $0.004236$. Notice that for the low-fidelity case $(s=0.0)$ this level set is outside the domain of $\hat{f}(\x, 0.0)$ and for the medium-fidelity case $(s=0.5)$ this level set is at the top-left corner of the domain. Therefore, this test case poses a unique challenge of the fidelity levels leading to very different function domains; for example, according to the low-fidelity model, the probability of failure is $0$.

To evaluate our algorithm, we compute the predicted failure probability via IS at the end of each iteration and take the absolute difference from the true failure probability. This is plotted against the cumulative computational costs for the single- and multifidelity cases in \Cref{fig:turbine_reliability}. From the figure, notice that for any given cumulative cost, the multifidelity approach results in a failure probability error that is better than or equal to the single-fidelity approach. Because of the availability of only three discrete fidelity levels and the assumed cost model,  the discrepancies in the cost between low and high fidelities are very pronounced; the single fidelity curve starts at a cumulative cost of $\approx 550$ (which is the cost of a high-fidelity query) and the multifidelity curve starts at a cumulative cost of $\approx 50$ (which is the cost of the lowest-fidelity query). \mg{Notice that, for the chosen budget of evaluations, \texttt{CAMERA} converges to a slightly higher failure probability error than the high-fidelity; however, this discrepancy ($\mcl{O}(0.01\%)$) can be considered negligible.}

\mg{On the right of \Cref{fig:turbine_reliability}, we show the distribution of the number of queries at each fidelity level. Notice that the algorithm is able to exploit the lowest fidelity model while keeping the high-fidelity queries at $\sim 20\%$, owing to the strong correlations between the fidelity levels for this experiment. The consistency across repetitions is also present, similar to what was observed in the synthetic experiments.}

\begin{figure}[h!]
    \centering
    \begin{subfigure}{.34\textwidth}
        \centering
        \includegraphics[width=1\linewidth]{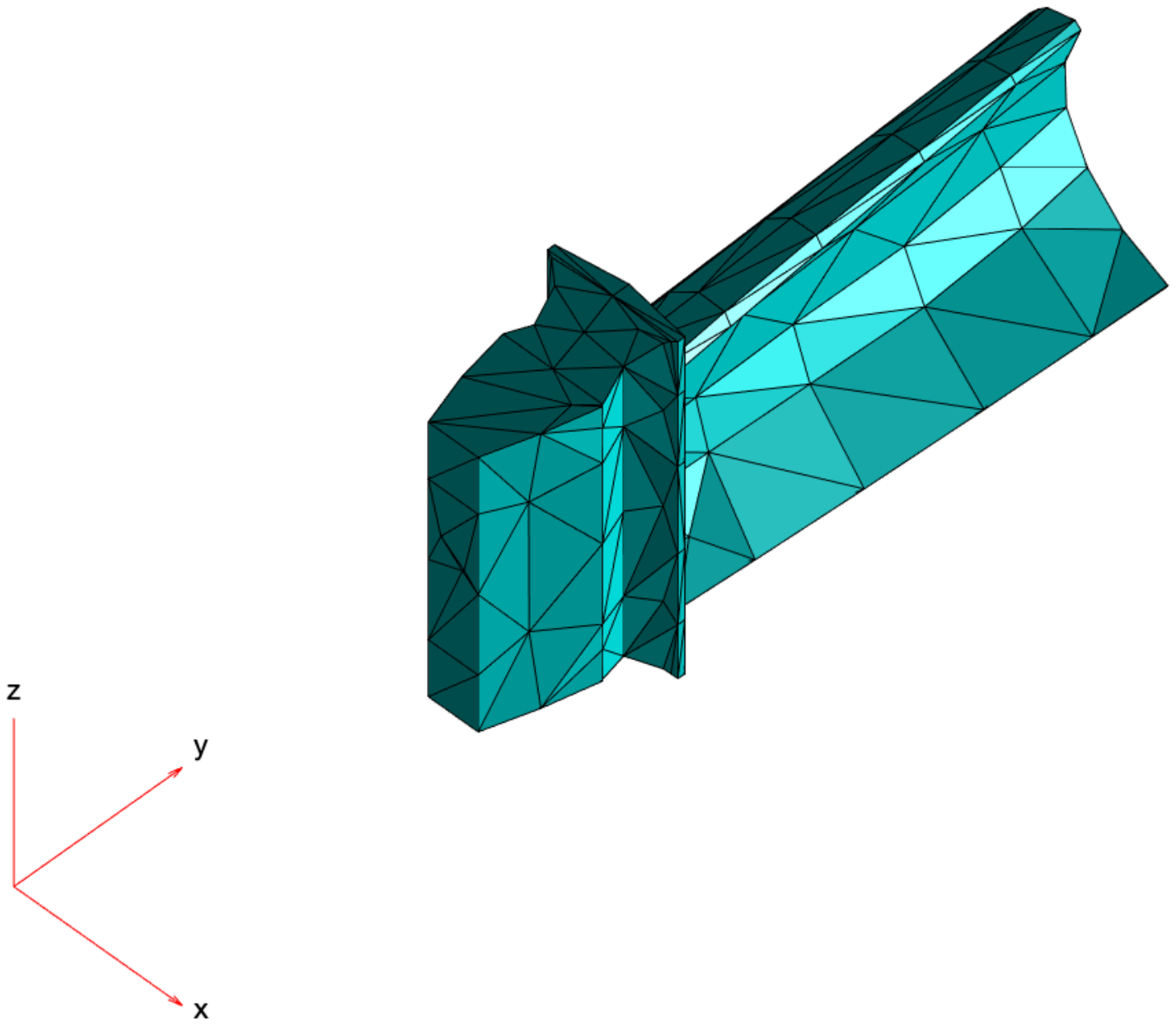}
        \caption{Low fidelity ($s=0.0$)}
    \end{subfigure}%
    \begin{subfigure}{.34\textwidth}
    \centering
        \includegraphics[width=1\linewidth]{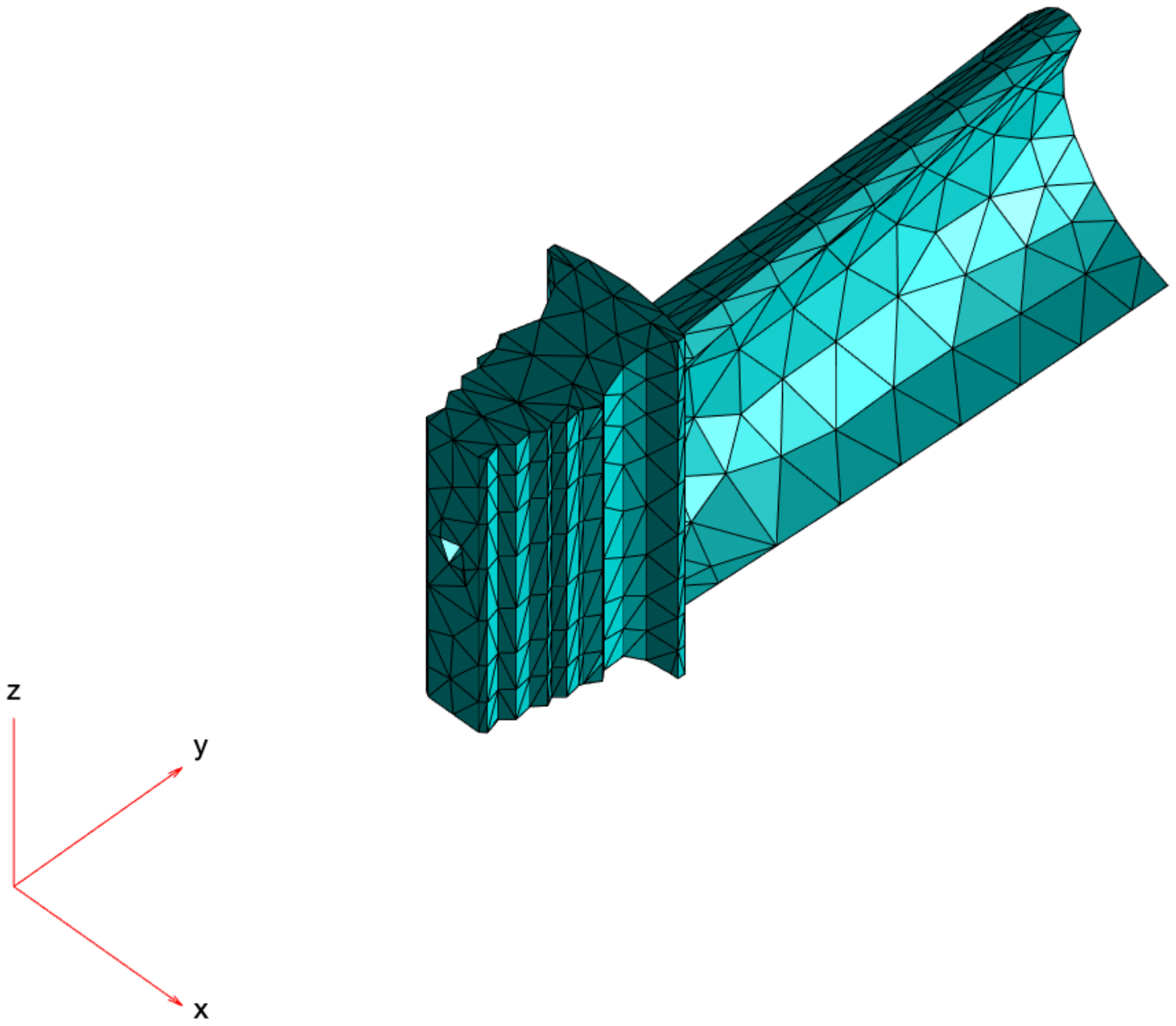}
        \caption{Medium fidelity ($s=0.5$)}
    \end{subfigure}%
    \begin{subfigure}{.34\textwidth}
        \includegraphics[width=1\linewidth]{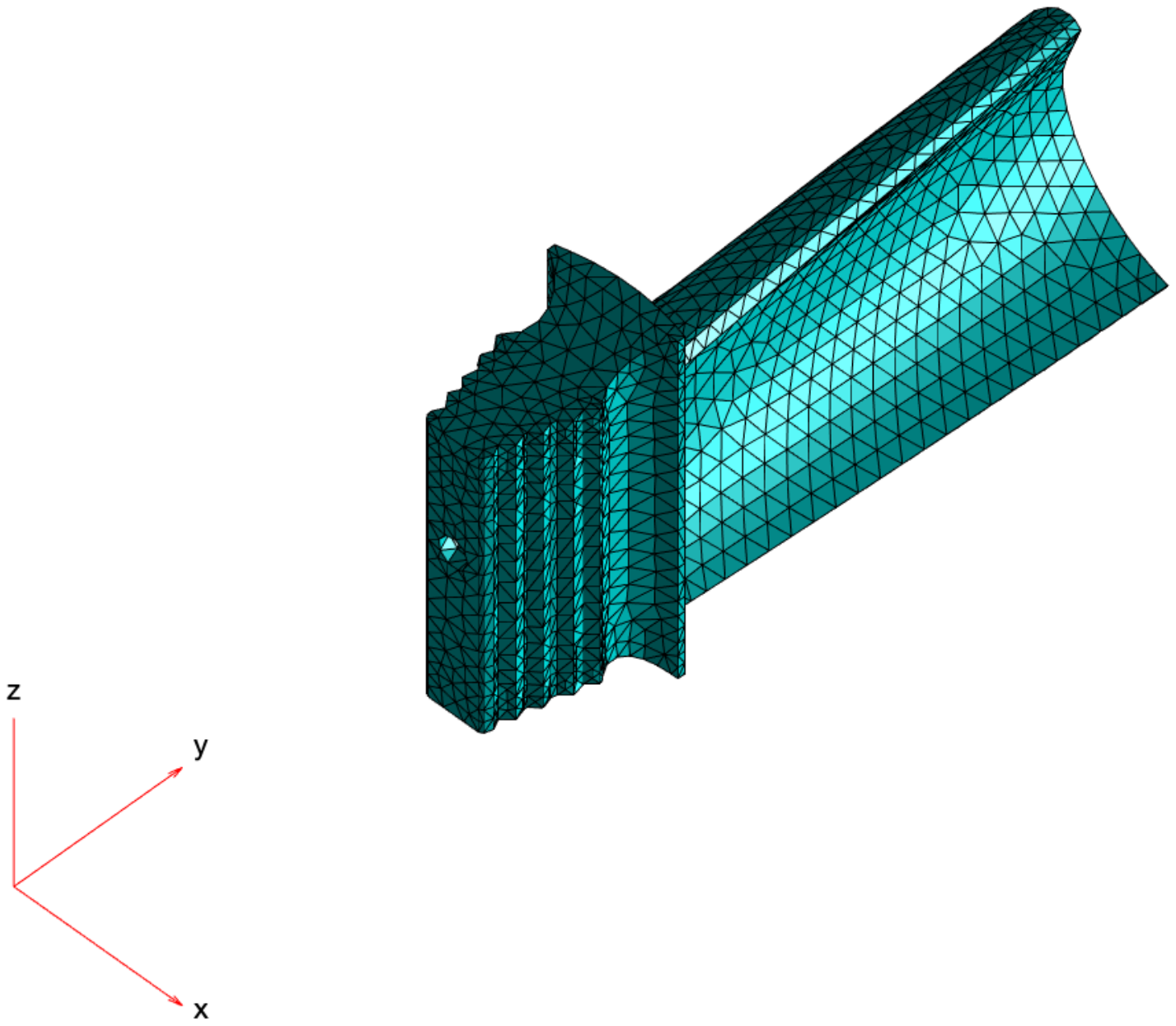}
        \caption{High fidelity ($s=1.0$)}
    \end{subfigure}%
    \caption{Mesh resolutions for the three discrete fidelities considered.}
    \label{fig:turbine_fidelities}
\end{figure}

\begin{figure}
    \centering
    \begin{subfigure}{.5\textwidth}
        \centering
        \includegraphics[width=1\linewidth]{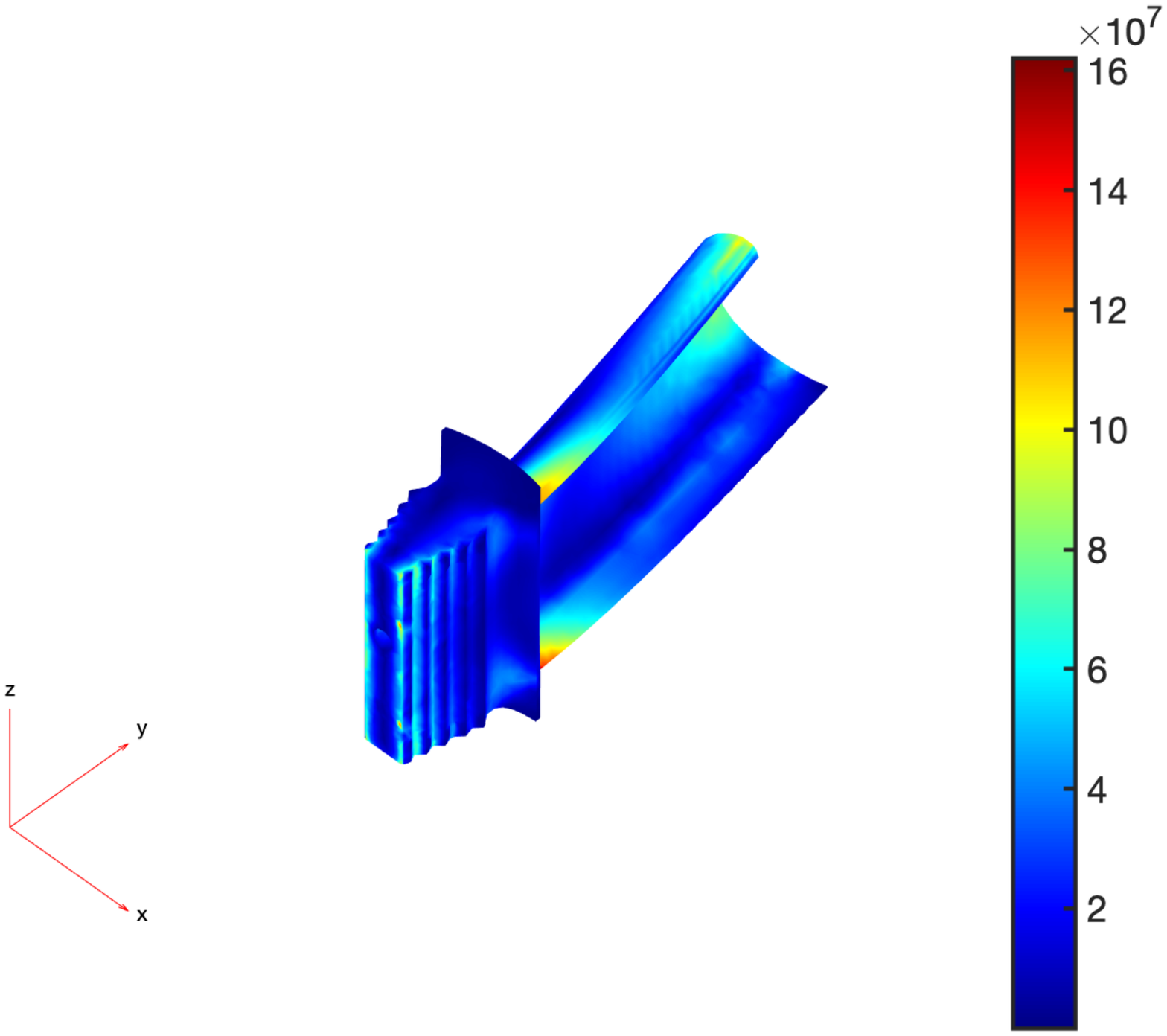}
        \caption{Solution at $\x = [7E5, 4.5E5]$}
    \end{subfigure}%
    \begin{subfigure}{.5\textwidth}
    \centering
        \includegraphics[width=1\linewidth]{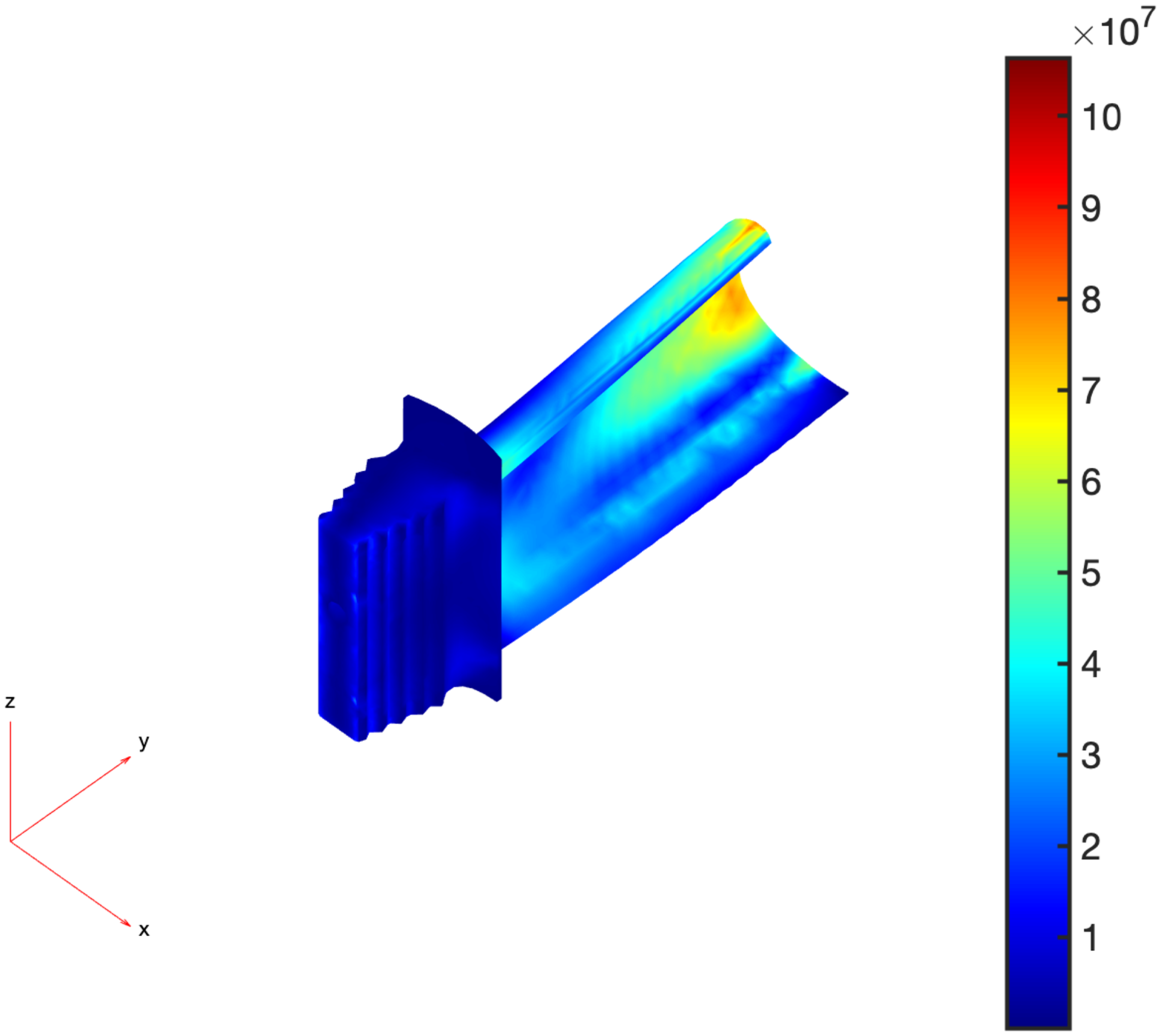}
        \caption{Solution at $\x = [5E5, 4.5E5]$}
    \end{subfigure}%
    \caption{Von Mises stress distribution and deformation on the gas turbine blade, under two different pressure loading boundary conditions.}
    \label{fig:turbine_solution}
\end{figure}

\begin{figure}
    \centering
    \includegraphics[width=1\textwidth]{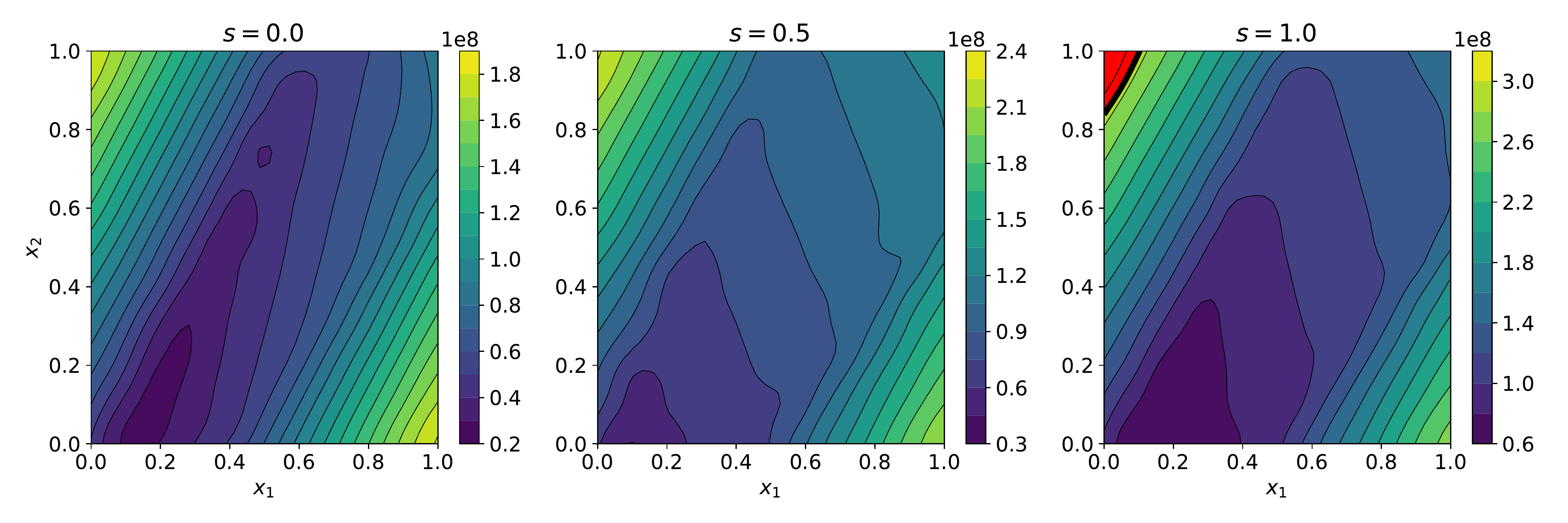}
    \caption{Contours of \emph{maximum} von Mises stress. The failure boundary (thick black line) is outside the range of the low-fidelity approximations $\hat{f}(\x, 0.0)$ and $\hat{f}(\x, 0.5)$.}
    \label{fig:turbine_mf_contours}
\end{figure}

\begin{figure}
    \centering
    \begin{subfigure}{.5\textwidth}
        \includegraphics[width=1\linewidth]{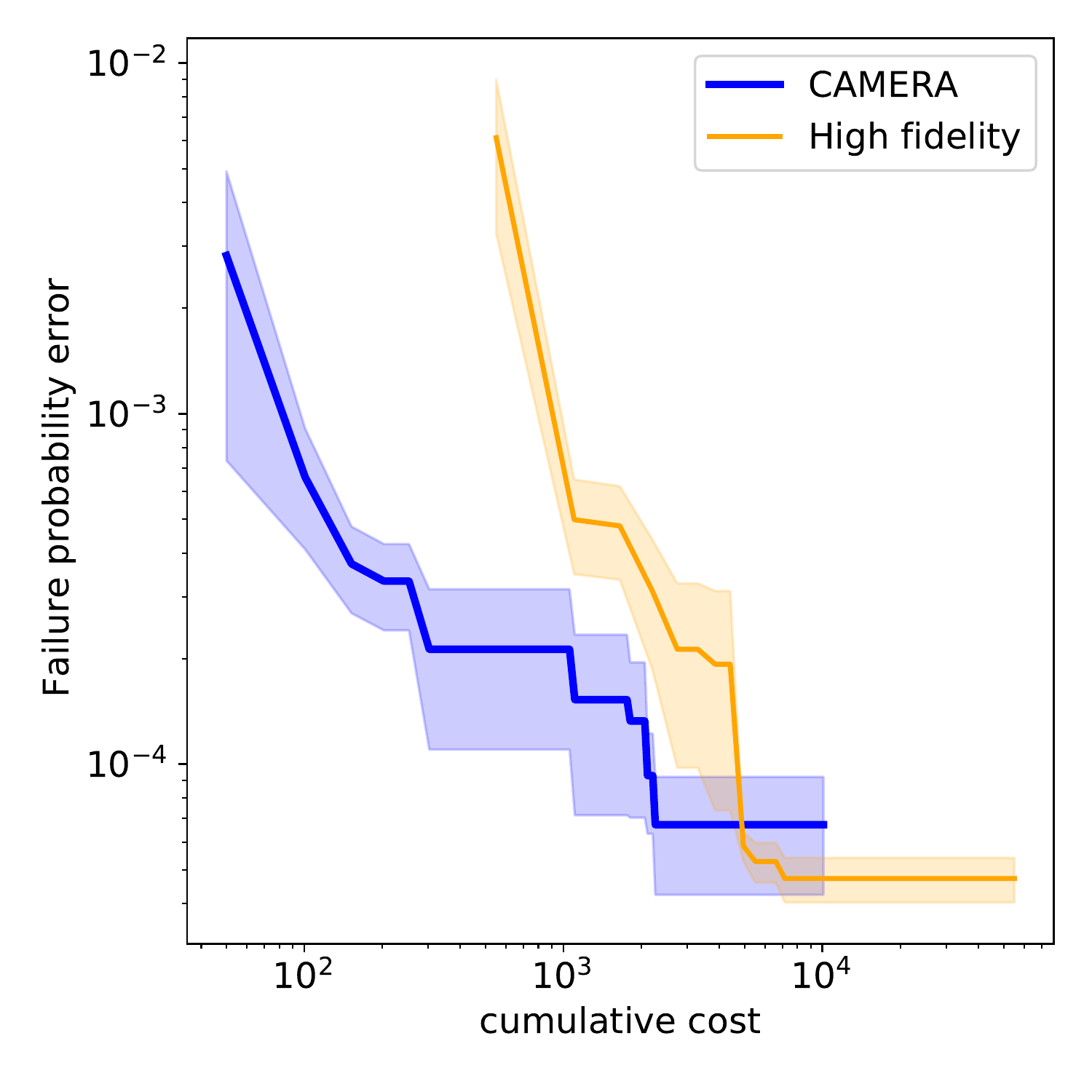}    
    \end{subfigure}%
    \begin{subfigure}{.5\textwidth}
        \includegraphics[width=1\linewidth]{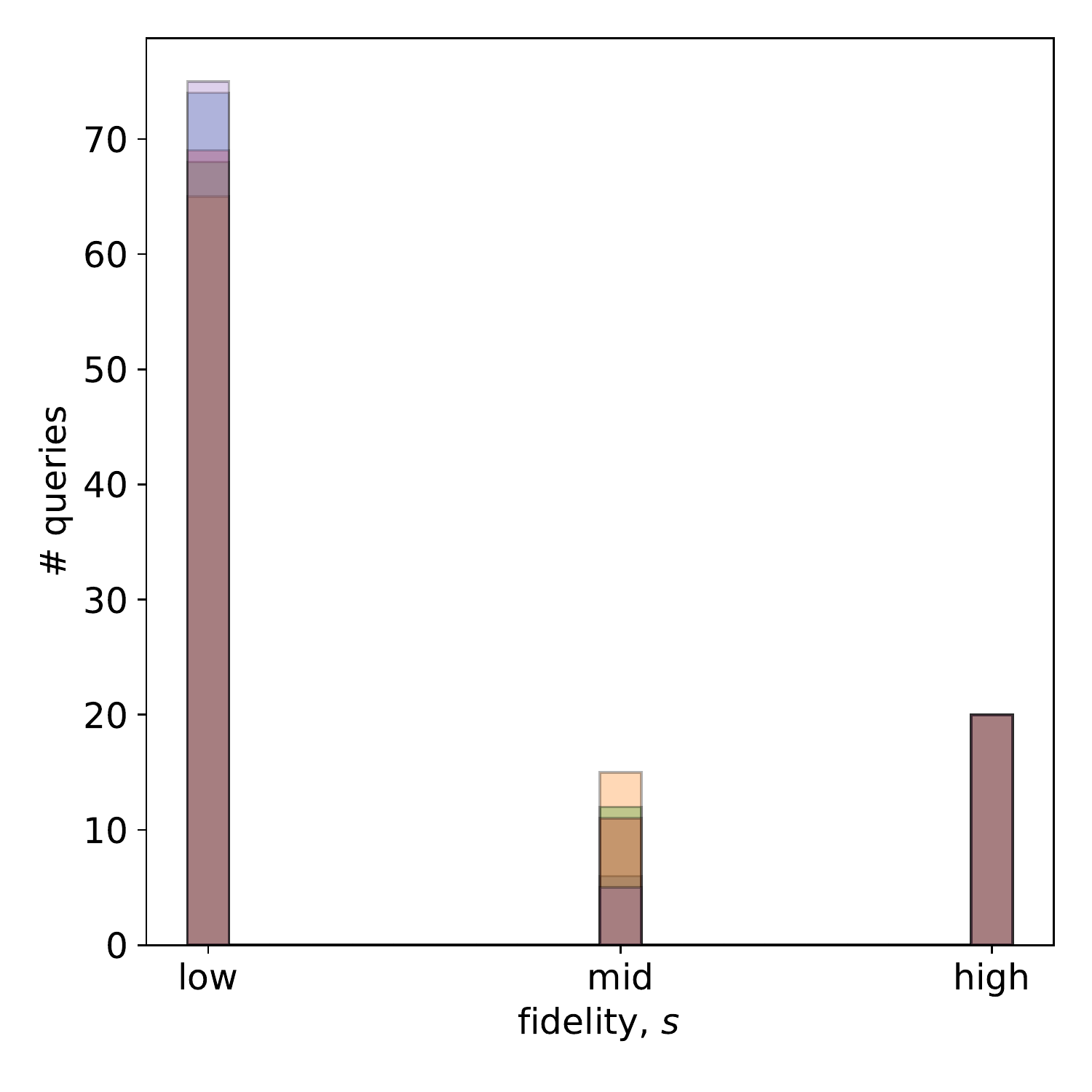}    
    \end{subfigure}    
    \caption{Results of the reliability analysis of the turbine blade. \mg{Left: failure probability error versus cumulative cost.} \mg{Right: distribution of the number of queries at each fidelity level; plot shows all 5 repetitions overlaid on top of each other.}}
    \label{fig:turbine_reliability}
\end{figure}

\subsection{ONERA transonic wing}
\label{ss:expts_onera}
The ONERA-M6 wing~\cite{schmitt1979pressure} is an analytical (shape prescribed using mathematical equations) wing created in the 1970s by the French Aerospace Laboratory for the purposes of understanding three-dimensional airflow in transonic speeds and high Reynolds numbers.\footnote{https://www.onera.fr/en/news/onera-m6-wing-star-of-cfd} It is a semi-span, swept wing with no twist,  and at transonic Mach numbers is characterized by local supersonic flow, shocks, and turbulent boundary layer separation. Therefore, this wing has been frequently used for benchmarking computational software for compressible turbulent flow~\cite{mani1997assessment}. Because of its relevance in the commercial and military aircraft flight regimes, which make it an important candidate for reliability analysis, we use it as one of the test cases in this work.

The wing, shown in \Cref{sf:onera_wingonly}, is surrounded by a fluid domain that is a quadrant of an ellipsoid, as shown in \Cref{sf:onera_domain}. The two external sides of the fluid domain are designated as ``freestream" boundary conditions, with a prescribed ambient temperature, flow speeds, and  angle of attack (that is, the incidence angle between the flight direction and the wing mean aerodynamic chord). The fluid domain is discretized with a mesh of approximately $56,000$ hexahedral elements with local refinement to resolve boundary layers, shocks, and shear layers. The three-dimensional compressible RANS equations are solved with a finite-volume method with second-order spatial discretization and Euler implicit discretization. The turbulence closure is provided by the Spalart--Allmaras~\cite{allmaras2012modifications} turbulence model. All the simulations are performed with the open-source, finite-volume-based multiphysics code SU2.~\footnote{https://su2code.github.io/}~\cite{economon2016su2}

For the reliability analysis, we consider the following input parameters:
\begin{equation*}
    \begin{split}
        x_1 :& \T{Reynolds number} \\
        x_2 :& \T{Mach number} \\
        x_3 :& \T{Angle of attack} \\
        s :& \T{Iterations to converge}.
    \end{split}
\end{equation*}
Therefore, $\x = [x_1, x_2, x_3]$, and our $f(\x)$ is the wing load factor $L/W$, where $L$ is the net aerodynamic force acting perpendicular to the flight path and $W$ is the weight of the aircraft. Note that the maximum wing load factor obtainable during flight is limited by the structural strength of the aircraft wing~\citep[Ch. 4]{anderson1999aircraft}, and hence we want to ensure the $L/W$ does not exceed a certain safety threshold, over a range of operating conditions.

The fidelity parameter $s$ is the number of iterations allowed for the solver to converge;  fewer iterations result in a faster (and hence cheaper) solution that is considered low fidelity and vice versa. We set the following bounds on the control and fidelity parameters: $x_1 \in [5\times 10^6, 11\times 10^6]$, $x_2 \in [0.7, 0.85]$, $x_3 \in [1, 3]$, and $\Sc \in [0,1]$ (which is normalized and corresponds to the number of iterations $\in [500, 5000]$). Note that we round any fractional values for the number of iterations, recommended via our algorithm, to the nearest integer. We show the pressure coefficient distribution on the wing for the different values of $s$ in \Cref{fig:onera_mf}. 
As in the turbine test case, we negate the output of the function $f$ and set the failure threshold $a = -7.5$ and therefore are interested in estimating the probability that $g(\x) = f(\x) - a \leq  0$.

We begin our algorithm with $n = 100$ seed points obtained via running the simulations at a Latin hypercube design of experiments across the joint $\X \times \Sc$ space, and we run the algorithm for 100 iterations. For the single-fidelity run, $n$ is set to keep the total cost of seed points identical to that of the multifidelity run. We keep the same cost model as in \Cref{fig:cost_model}. Note that for this experiment we compute the failure probability only once---at the conclusion of running \Cref{alg:camera} for 100 iterations. Furthermore, we do not repeat the experiment with randomized seed points. We fix $N=400$ high-fidelity samples to  evaluate $\pfis$.

\begin{figure}
    \centering
        \begin{subfigure}{.5\textwidth}
        \includegraphics[width=1\linewidth]{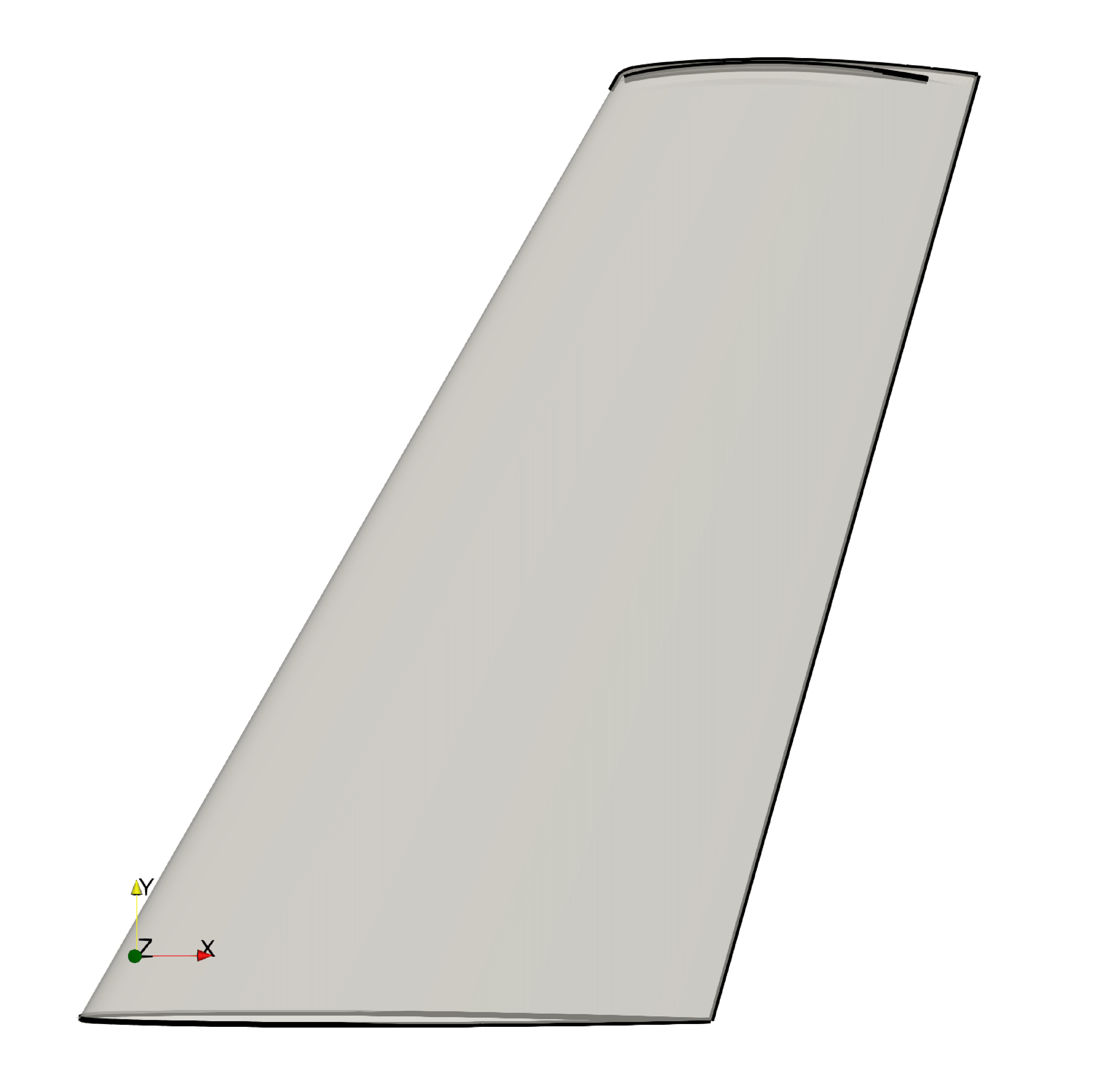}
        \caption{ONERA wing.}
        \label{sf:onera_wingonly}
    \end{subfigure}%
    \begin{subfigure}{.5\textwidth}
        \includegraphics[width=1\linewidth]{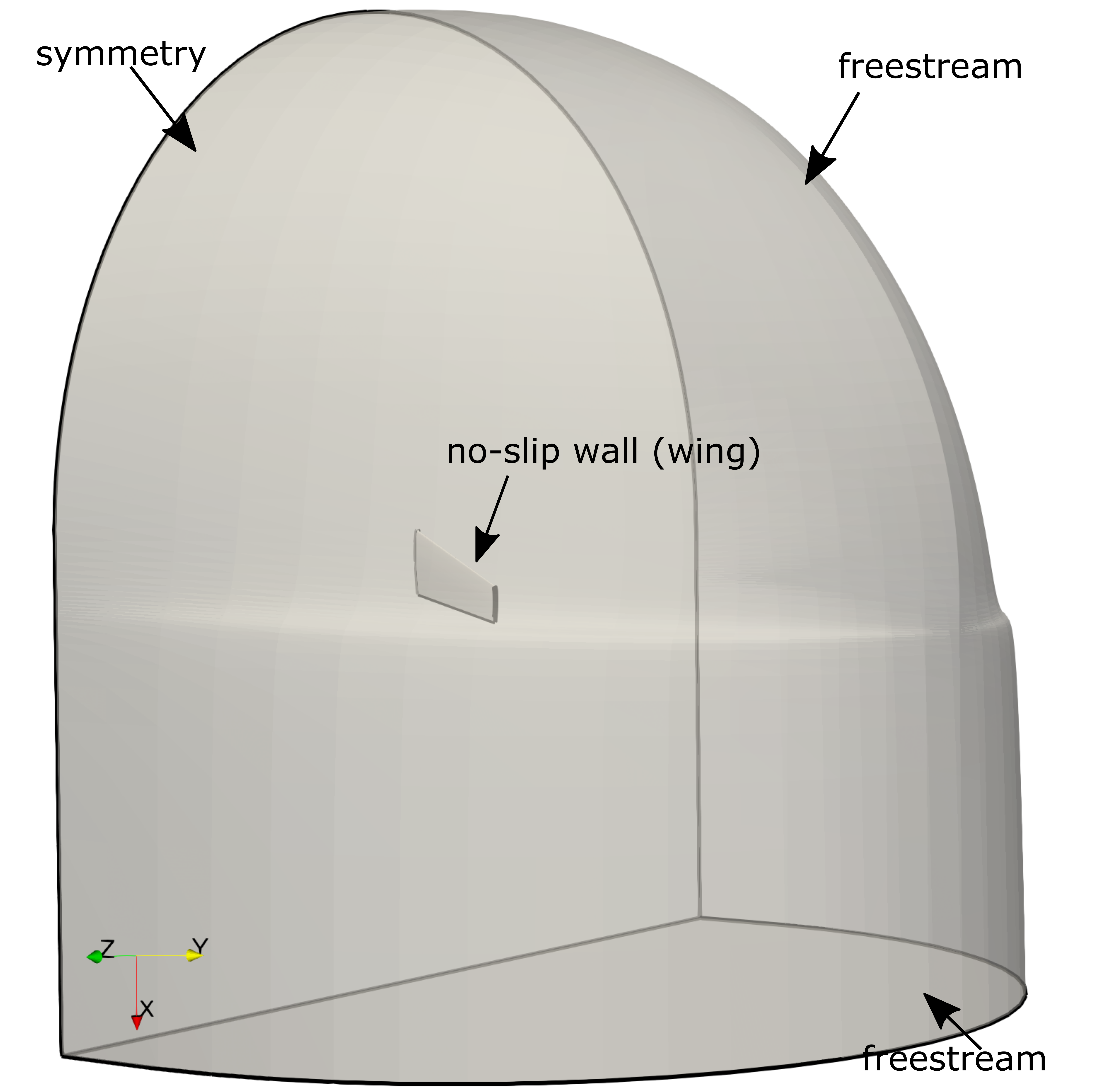}
        \caption{Fluid domain surrounding wing.}
        \label{sf:onera_domain}
    \end{subfigure}%
    \caption{Fluid domain and boundary conditions for the ONERA transonic wing.}
    \label{fig:onera}
\end{figure}

\begin{figure}
    \centering
        \begin{subfigure}{.5\textwidth}
            \centering
            \includegraphics[width=1\linewidth]{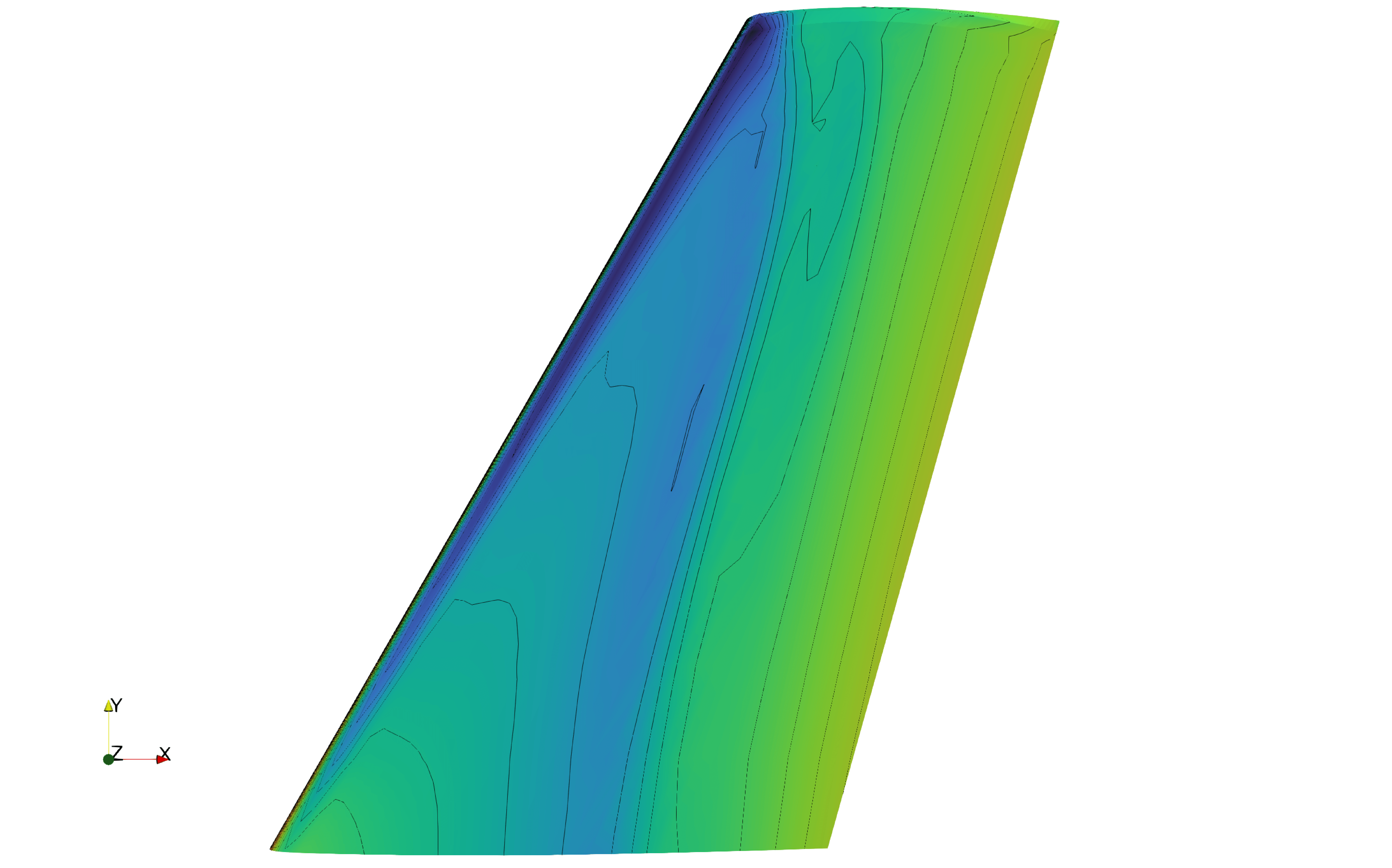}
            \caption{Unconverged (500 iterations).}
            \label{sf:onera_500}
        \end{subfigure}%
        \begin{subfigure}{.5\textwidth}
            \centering
            \includegraphics[width=1\linewidth]{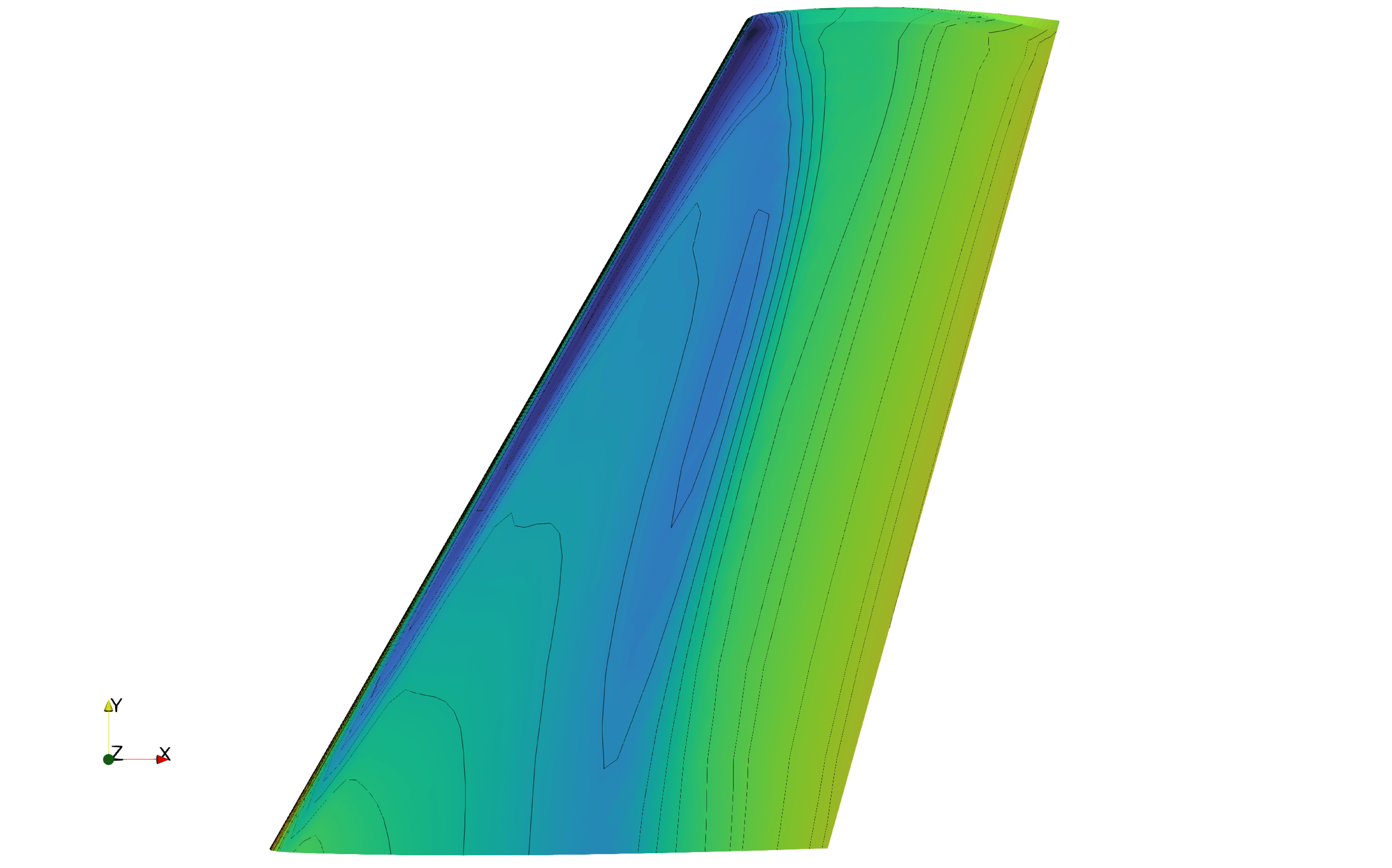}
            \caption{Unconverged (2500 iterations).}
            \label{sf:onera_2500}
        \end{subfigure} \\
        \begin{subfigure}{.5\textwidth}
            \centering
            \includegraphics[width=1\linewidth]{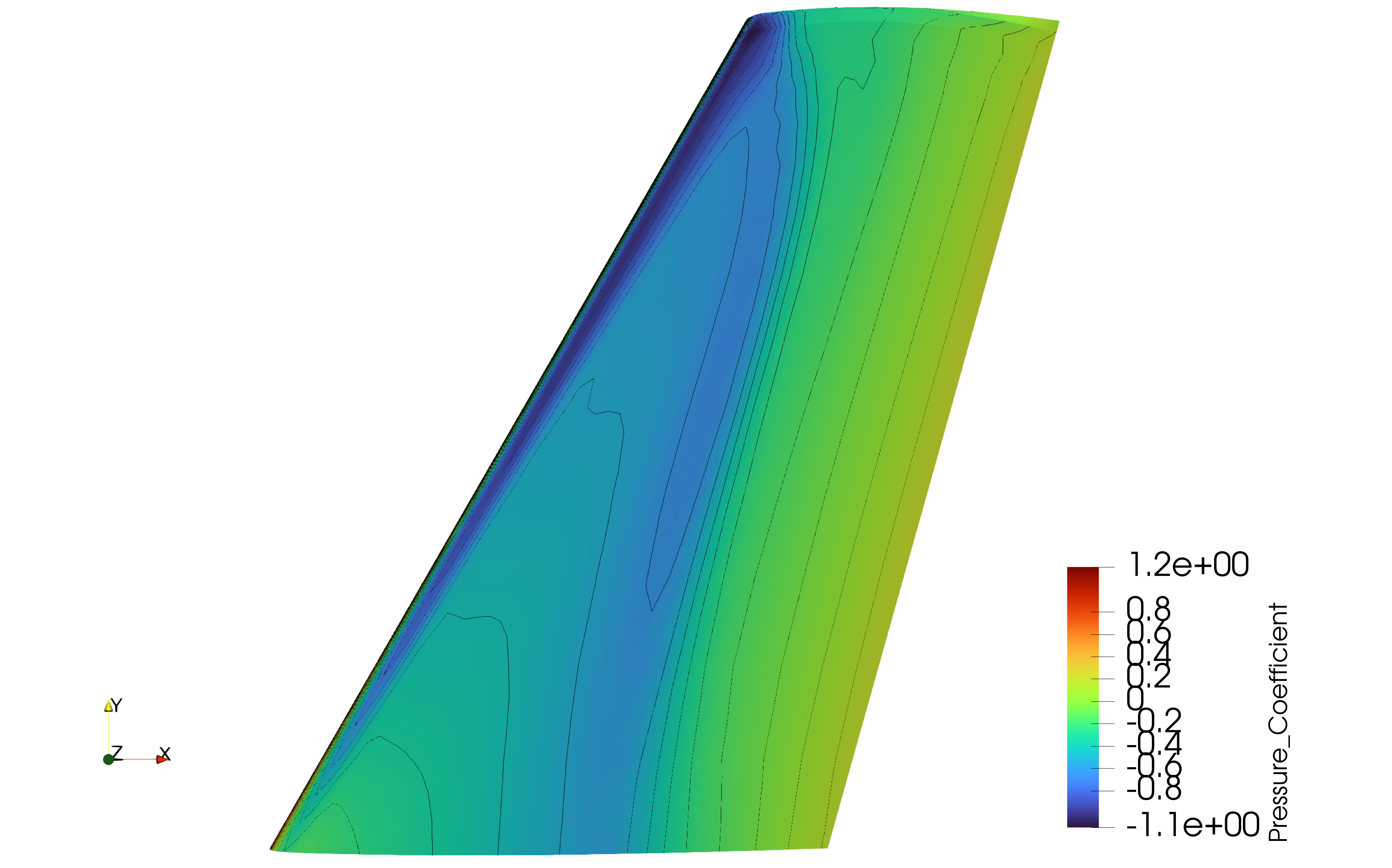}
            \caption{Converged (5000 iterations).}
            \label{sf:onera_5000}
        \end{subfigure}%
    \caption{Surface pressure coefficient distribution for the ONERA wing at 3 different levels of convergence and fixed boundary conditions.}
    \label{fig:onera_mf}
\end{figure}

We show the results in \Cref{fig:onera_reliability}, where $\pfis$ and  $\mbb{V}(\pfis)$ are used to construct the confidence intervals. While we do not know the $\pf$ (truth) for this case, we see that the predictions from $\texttt{CAMERA}$ are with smaller confidence intervals compared with using the high-fidelity model only. The overall cumulative cost of the high-fidelity run is 44,667 compared with only 38,614 for \texttt{CAMERA}, thereby establishing the cost-effectiveness of the method. 

\begin{figure}
    \centering
        \includegraphics[width=.5\linewidth]{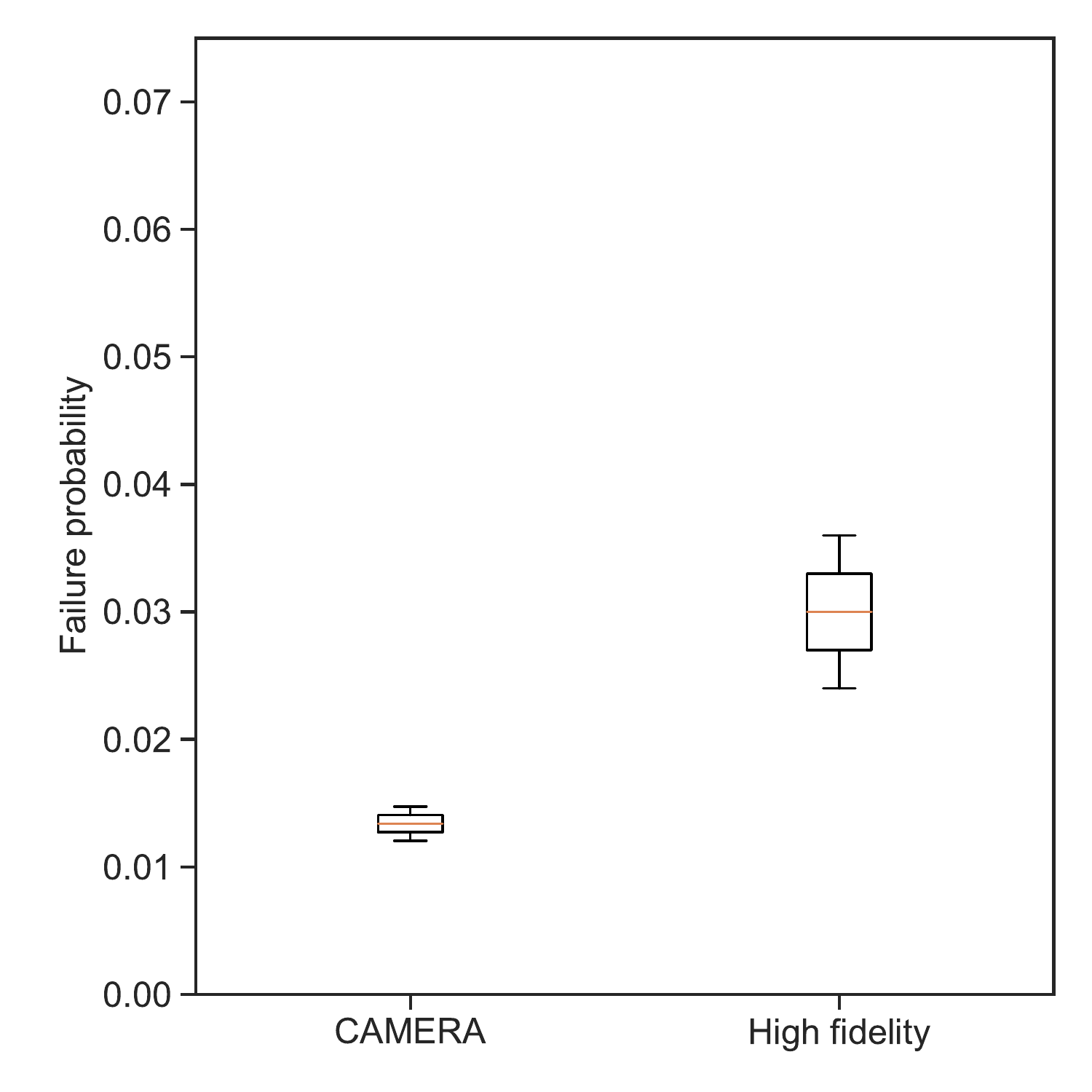}
    \caption{Predicted $\pfis$ via \texttt{CAMERA} and using the high-fidelity model only.}
    \label{fig:onera_reliability}
\end{figure}

\section{Conclusion}
\label{sec:conclusion}
In this work we address the problem of failure probability estimation for complex engineered systems with a computationally expensive model. Additionally, we account for the fact that the expensive model has tunable fidelity parameters that can trade computational cost for predictive accuracy. We propose an approach based on Gaussian process models that can learn the correlations between the different fidelities, given a continuous or discrete fidelity space, and therefore predict the high-fidelity function with potentially fewer queries to the expensive high-fidelity model that are supplemented by cheaper low-fidelity model queries. Further, we propose an acquisition-based sequential experiment design approach to adaptively build the multifidelity surrogate model, so that only the most \emph{useful} fidelities and inputs are queried. The main highlight of our proposed acquisition function is that it provides a generalized multifidelity framework that can extend existing (single-fidelity) acquisition functions to the multifidelity setting and furthermore applies to both discrete and continuous fidelity spaces. Additionally, we demonstrate that our surrogate modeling approach is asymptotically consistent, meaning that the predicted failure level set approaches the true level set with high probability as the number of iterations approaches infinity. We also show that estimates of the failure probability, predicted  by using only the developed multifidelity surrogate model, produces very high accuracy at a lower computational cost, compared with using the high-fidelity model only. \mg{Furthermore, using the multifidelity surrogate model with a single batch of (non-adaptive) Latin hypercube design results in poor estimates of the failure probability, compared to the proposed adaptive approach. This confirms the need for adaptive training of surrogate models for failure probability estimation.}

With the sequential surrogate model building approach, we observe that the predicted failure level sets are very close to the true level set. Furthermore, we observe that in the presence of multiple fidelity models, our approach performs even better by requiring only $\approx 50\%$ of the computational budget. In terms of predicting the failure probability, our approach shows a decrease in the prediction error with increasing evaluations of the models, while the multifidelity approach always outperforms the single-fidelity case. In addition to the synthetic cases, we  demonstrate our method on a real-world turbine reliability case, with discrete fidelities, where the failure probability is $\approx 0.4\%$, and a transonic wing test case. We observe results consistent with the synthetic cases where we are able to obtain very high accuracy in failure probability error, with the multifidelity approach significantly outperforming using the high-fidelity model only. Overall, our experiments reveal that our approach can lead to cost-effective yet accurate predictions for reliability analysis, for example, using only $N=400$ high-fidelity samples for estimating $\pfis$ in the transonic wing example.

We anticipate  conducting follow-up studies of this work in several directions. First, the \texttt{CAMERA} algorithm can be modified to adaptively provide an estimate for $\pfis$; currently, it is estimated at the conclusion of the adaptive surrogate model construction. This is useful for real-world applications, where running expensive high-fidelity simulations in multiple batches is less risky and therefore more practical. Along the same lines, we  plan on evaluating the performance of our sequential experiment design of  adaptive surrogate model construction, in making batch evaluations at every step. Furthermore, whereas we have demonstrated our approach in estimating $\pf$ in $\mcl{O}(10^{-4})$, it would be interesting to evaluate the approach on problems with  smaller orders of magnitude in $\pf$. 

\appendix
\section{Impact of the choice of cost model}
We vary the choice of the constant $c_1$ in the cost model to observe its impact on the algorithm performance. Recall that $c_1$ determines the rate of change of the computational cost of the models with respect to fidelity. In this regard, smaller values of $c_1$ result in ``flatter" cost models, as shown in \cref{sf:cost_models}. We use the Ishigami test function to evaluate the impact of the choice of cost model on the performance of \texttt{CAMERA}, whose results, with 20 repetitions with randomized seed points, are shown in \cref{sf:ishigami_cost}. We observe  no appreciable difference in the algorithm performance  between a choice of $c_1=10$ and $c_1=5$. At $c_1=1.0$, \texttt{CAMERA} still outperforms the high-fidelity-only run, although suffering a higher failure probability error at lower cumulative costs. As $c_1$ approaches $0$, the performance approaches that of the high-fidelity-only run, as expected. Therefore, we conclude that the choice of the cost model, while expected to impact the performance of the algorithm, does not necessarily offer an undue advantage to the algorithm. This experiment justifies the use of $c_1 = 10.0$ in all of our experiments.

\begin{figure}[h!]
    \centering
        \begin{subfigure}{.5\textwidth}
        \includegraphics[width=1\linewidth]{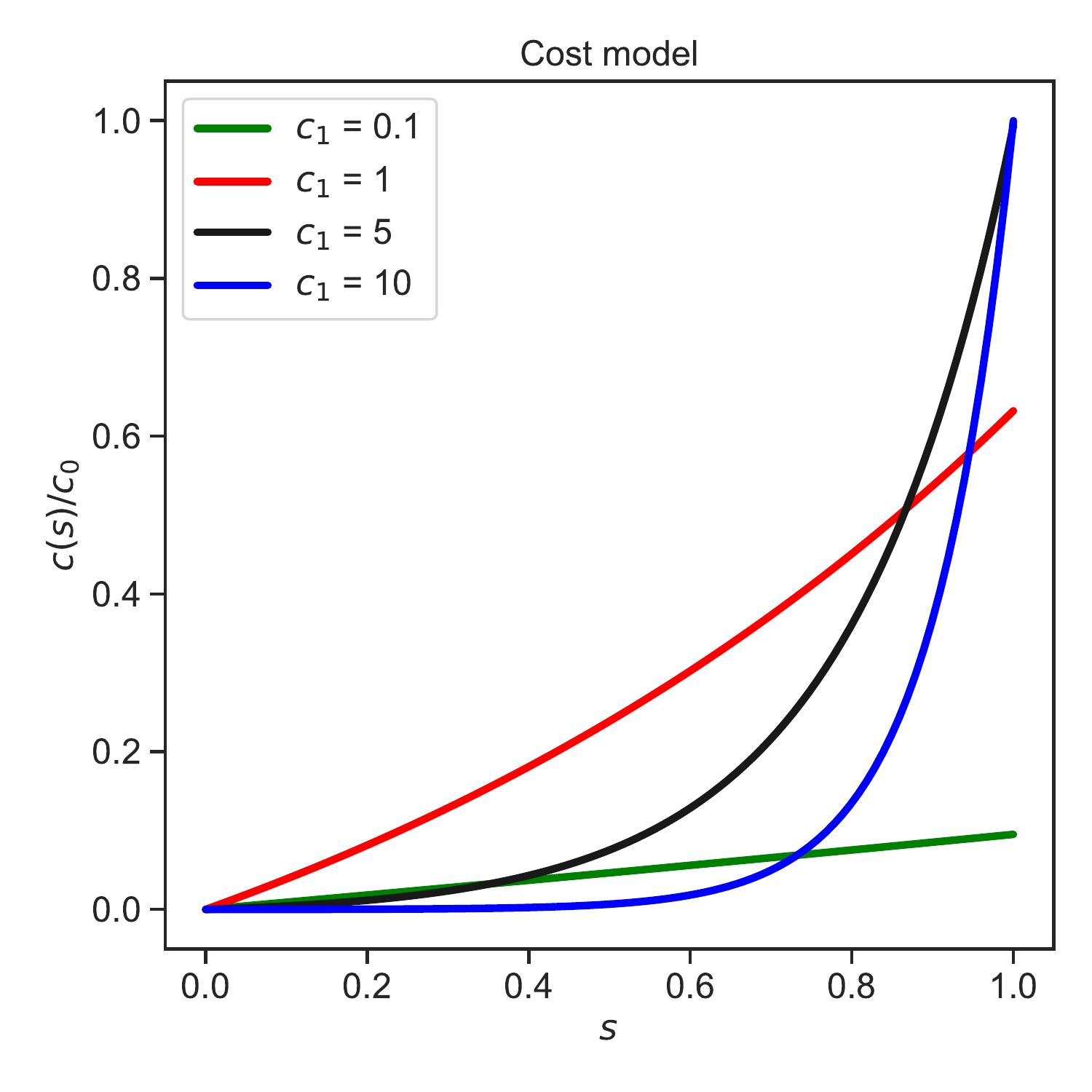}
        \caption{Cost model with varying $c_1$.}
        \label{sf:cost_models}
    \end{subfigure}%
    \begin{subfigure}{.5\textwidth}
        \includegraphics[width=1\linewidth]{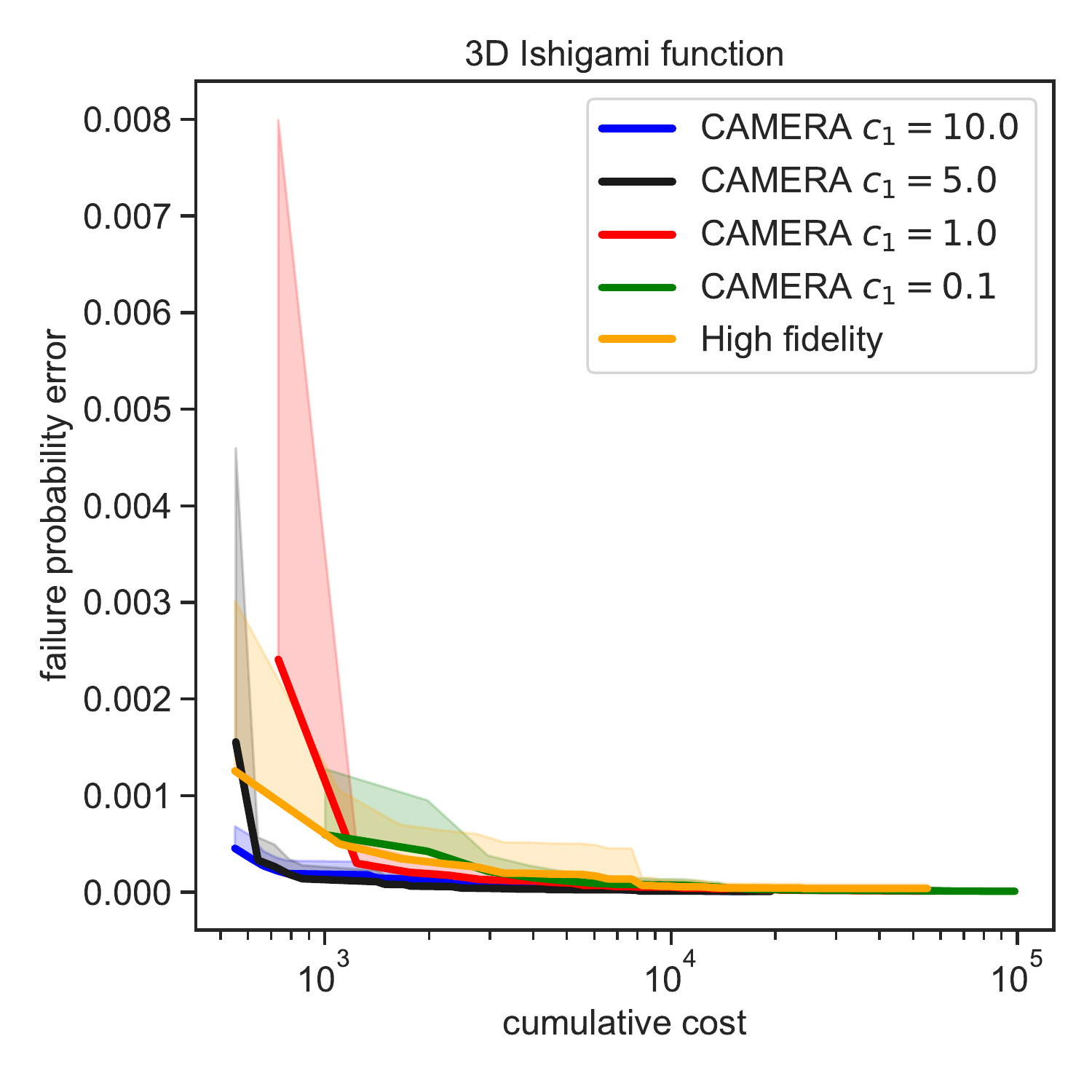}
        \caption{Ishigami test case.}
        \label{sf:ishigami_cost}
    \end{subfigure}%
    \caption{Impact of the choice of cost model on the algorithm performance.}
    \label{fig:ishigami_cost_models}
\end{figure}

\section{Expected improvement acquisition function of \cite{ranjan2008sequential}}
The closed-form expression for the expected improvement acquisition function by \citet{ranjan2008sequential}, obtained by setting $\delta ^2(\x, s) - \min\{(y(\x, s) - a)^2, \delta^2(\x, s)\}$ and $\bw = \{a, \eta\}$ in \eqref{e:value_function}, is given by
\begin{equation}
\begin{split}
    EI_r(\x, s) = & \left[\delta^2(\x,s) - (\mu(\x,s) - a)^2 - \sigma^2(\x,s)\right] \left(\Phi(z^+) - \Phi(z^-)\right) + \\ &\sigma(\x,s)^2 \left(z^- \phi(z^-) - z^+ \phi(z^+)\right) \\
    & 2\left(\mu(\x,s)-a \right)\sigma(\x,s)\left(\phi(z^-) - \phi(z^+)\right).
\end{split}
\label{e:Ranjan}
\end{equation}

\section*{Acknowledgments}
This work is partially supported by the Laboratory Directed Research and Development (LDRD) funding from Argonne National Laboratory, provided by the Director, Office of Science, of the U.S. Department of Energy under contract DE-AC02-06CH11357, and by the Faculty Startup Funds at the University of Utah. We thank Carlo Graziani from Argonne National Laboratory for his valuable feedback.

\bibliographystyle{plainnat}
\bibliography{sample, ref}








\end{document}